\documentclass[12pt,a4paper]{article}

\usepackage{pdflscape}
\usepackage{amsmath}
\usepackage{amssymb, amsfonts, amscd, mathrsfs}
\usepackage{latexsym}
\usepackage{srcltx}
\usepackage{graphics}
\usepackage{color}
\usepackage{epsfig}
\usepackage{color}
\usepackage{hyperref}

\newcommand{\co}{{\mathbb C}}

\newcommand{\R}{{\mathbb R}}
\newcommand{\re}{{\mathbb R}}

\newcommand{\z}{{\mathbb Z}}
\newcommand{\Z}{{\mathbb Z}}

\newtheorem{theorem}{Theorem}
\newtheorem{propos}{Proposition}
\newtheorem{lemma}{Lemma}
\newtheorem{cor}{Corollary}
\newtheorem{ex}{Example}

\newtheorem{definition}{Definition}
\newtheorem{remark}{Remark}
\newtheorem{proof}{Proof}
\newtheorem{mytheorem}{Theorem}

\DeclareMathOperator{\Bs}{Bs}

\usepackage{caption}
\usepackage{subcaption}
\usepackage{float}

\date{}
\author{
Tatyana Zaitseva
\thanks{Moscow Center for Fundamental and Applied Mathematics, Moscow State University, Department of Mechanics and Mathematics, Russia {e-mail: \tt\small
zaitsevatanja@gmail.com}} 
}

\title{Multivariate tile B-splines 
\thanks{
This research was supported by the Russian Science Foundation (project no. 21-11-00131) at Lomonosov Moscow State University.
}}

\begin{document}
\maketitle

\begin{abstract}

Tile B-splines in $\mathbb{R}^d$ are defined as autoconvolutions of the indicators of tiles, which are special self-similar compact sets whose integer translates tile the space $\mathbb{R}^d$. These functions are not piecewise-polynomial, however, being direct generalizations of classical B-splines, they enjoy many of their properties and have some advantages. In particular, the precise values of the H{\" o}lder exponents of the tile B-splines are computed in this work. They sometimes exceed the regularity of the classical B-splines. The orthonormal systems of wavelets based on the tile B-splines are constructed and the estimates of their exponentional decay are obtained. Subdivision schemes constructed by the tile B-splines demonstrate their efficiency in applications. It is achieved by means of the high regularity, the fast convergence, and small number of the coefficients in the corresponding refinement equation. 

\bigskip

\smallskip

\begin{flushright}

\end{flushright}

\end{abstract}
\bigskip

\vspace{1cm}

\begin{section}{Introduction} \label{intro}
B-splines represent one of the most famous and simple piecewise-polynomial bases. They are widely studied in the literature (see, for example, \cite{Boor}). B-splines are used in the construction of orthogonal Battle-Lemarie wavelets \cite{Daub, NPS, Woj}, in effective algorithms of piecewise-polynomial approximation \cite{Boor, Shad, BVR93, P06, Ter}, in approximate formulas for Gaussian distribution, in formulas for volumes of slices of the multivariate cube, etc. Depending on application, different B-splines may be considered, such as those with irregular knots, penalized B-splines, B-splines defined on various areas, etc. There exist several ways to construct B-splines in the multivariate case \cite{BHR, CCJZ, VBU}. The most popular one is the direct product of the univariate B-splines. The common property of this and other approaches is that those B-splines are indeed splines, i.e., piecewise polynomial functions. 

However, there is another natural generalization of B-splines that exploits the fact that the univariate B-spline of order $k$ is the convolution of $k + 1$ indicator functions of the segment $[0, 1]$. In this work we consider \textit{tile B-splines} defined as an autoconvolution of the indicator of a special compact set (tile). Similar construction was considered in \cite{Zakh2, Zube}, see Remark \ref{zz} for details. Each tile is a union of its several contractions by means of affine operators with the same linear part. Tiles are known in wavelet theory since they are a key ingredient in the efficient approach for construction of multivariate Haar systems developed in the works of Lagarias, Wang, Gr\"ochenig, Haas and others. (\cite{LW97, GH, GM}). The case when a tile consists of two contractions is especially interesting, we call such tiles \textit{two-digit} tiles or \textit{2-tiles}. The properties of systems which are based on two-digit tiles are in some sense closest to the univariate case.  
On the plane there are three types of affinely non-equivalent 2-tiles, in $\R^3$ there are seven types. In this work we consider in detail B-splines based on these three classes of plane 2-tiles (we call them Square, Dragon, and Bear). 

The tile B-splines inherit many advantages of classical B-splines, however, their use is complicated due to the following problems: 

1) How to compute the values of the tile B-splines and of their derivatives? Unlike the classical B-splines, the tile B-splines are not defined explicitly, their straightforward computation requires the calculation of convolution, i.e., a numerical integration. 

2) How to analyze their properties, in particular, the regularity, which is a key parameter in many applications (for example, for the wavelet-Galerkin method)? 

3) How effective are tile B-splines in applications? Is it possible to construct wavelet systems and subdivision schemes (SubD algorithms) generated by such splines? 

In this work we make an attempt to answer all these questions and to apply tile B-splines to the construction of wavelets and for the design of subdivision schemes in the geometric surface modeling. 

Our results would be of theoretical interest only, if the tile B-splines did not have the advantages over the classical B-splines and were not effective in applications. 
However, some of the constructed two-digit B-splines (for example, the so-called Bear-3, Bear-4) are surprisingly smoother than the  corresponding classical B-splines (Theorem \ref{th_smooth}). Thus, the standard (direct product) B-splines do not possess the highest regularity. 
The rate of convergence of some numerical algorithms based on B-splines, such as the cascade algorithm for computation of coefficients of wavelet decomposition, subdivision algorithms for curve and surface modeling \cite{CDM, CC}, depends on the regularity of B-splines. Therefore, using tile B-splines we obtain faster convergence and better quality of limit functions and surfaces. 

Moreover, we show that one of the classes of two-digit B-splines (the so-called Square-$(n + 1)$) coincides with the classical multivariate B-splines of order $n$, but its refinement equation has much fewer nonzero coefficients. Therefore, the complexity of one its iteration in numerical algorithms is lower (Theorem \ref{th_subd}). 

This work is organized as follows: in Section \ref{par_tile} we recall the definition of tiles and their properties, Section \ref{classic} is devoted to the univariate B-splines. In Section \ref{tile_spl} we define tile B-splines and prove their fundamental properties. In Sections \ref{par_orthogonalization} -- \ref{phik} we construct the orthogonalization of these B-splines in two-digit case in a standard way. Further, in Section \ref{wavelet}, we find the corresponding wavelet function, i.e., we construct a wavelet basis based on two-digit tile B-splines, similarly to Battle-Lemarie wavelets.  
Since the wavelet function after orthogonalization is no longer compactly supported, it is important to analyze its rate of decay. It is estimated in Section \ref{par_finite} by means of the  multivariate complex analysis (the Laurent series, the Reinhardt domains). This gives a good approximation of tile wavelet function with finite functions. 
In Section \ref{par_smooth} we compute H{\" o}lder regularity of tile B-splines. 
Finally, in Section \ref{par_subdiv} tile B-splines are applied to subdivision schemes for surface modeling.  These theoretical results are the basis for our software package for construction of B-splines, wavelets and smoothness calculation \cite{gitTZ}. 
\end{section}

\begin{section}{Tiles} \label{par_tile}
Every integer matrix $M \in \Z^{d \times d}$ defines the partition of the lattice $\Z^d$ to $m = |\det M|$ equivalent classes $y \sim x \Leftrightarrow y - x \in M\Z^d$. Choosing one representative $d_i \in \Z^d$ from each coset, we obtain a \textit{digit set} $D(M) = \{d_0, \ldots, d_{m - 1}\}$. Assume that $d_0 = 0$. In the univariate case, if $M$ is a number, $D(M)$ is a digit set in the number system with the base $m$. Thus, an integer matrix and a proper digit set define a ``number system'' in $\Z^d$. Further we assume that the matrix $M$ is expanding, i.e., all of its eigenvalues are larger than one in absolute value. In this case, similarly to the unit segment we can consider the following set in this number system  
$$G = \left\{ \sum \limits_{k = 1}^{\infty}M^{-k}\Delta_k \mid \Delta_k \in D(M)\right\}.$$ 

It is known (see, for example, \cite{GH, GM}) that for every expanding integer matrix $M$ and for an arbitrary digit set $D(M)$, the set $G$ is compact, has a non-empty interior and possesses the following properties: 

\begin{enumerate}
\item the Lebesgue measure $\mu(G)$ is a positive integer;
\item (self-affinity) $G = \bigcup_{\Delta \in D(M)} {M^{-1}(G + \Delta)}$, all the sets $M^{-1}(G + \Delta)$ have pairwise intersections of measure zero;
\item the indicator $\varphi = \chi_G(x)$ of the set $G$ satisfies a refinement equation almost everywhere on $\R^d$; 
$$\varphi(x) = {\sum \limits_{\Delta \in D(M)}{\varphi(Mx - \Delta)}}, \quad x \in \R^d;$$
\item $\sum_{k \in \Z^d}\varphi(x + k) \equiv \mu(G)$ a.e., i.e., integer shifts of $\varphi$ cover $\R^d$ with $\mu(G)$ layers;
\item $\mu(G) = 1$ if and only if the function system $\{\varphi(\cdot + k)\}_{k \in \Z^d}$ is orthonormal.
\end{enumerate}

The last property allows us to introduce the following notion: 

\begin{definition}
Let us fix an expanding matrix $M \in \Z^{d\times d}$ and the digit set $D(M) = \left\{d_0, \ldots, d_{m - 1}\right\}$. If the measure of the set $$G = \left\{ \sum \limits_{k = 1}^{\infty}M^{-k}\Delta_k \mid \Delta_k \in D(M)\right\}$$ is one, i.e., all integer shifts of $G$ form a  disjoint, up to measure zero, covering of $\R^d$, then $G$ is called a \textit{tile}.  
\end{definition}

In some sense, a tile is a multivariate generalization of the segment $[0, 1]$ for the ``number system'' with the matrix base $M$. 

\begin{ex}\label{ex1}
For the univariate case $d = 1$, if $M = 2$, we can choose $D(M) = \{0, 1\}$. Then 
$$G = \left\{ \sum \limits_{k = 1}^{\infty}2^{-k}\Delta_k \mid \Delta_k \in \{0,1\}\right\} = [0, 1].$$ 
The segment $[0, 1]$ satisfies all the properties above. In particular, its indicator $\varphi(x) = \chi_{[0, 1]}$ satisfies a refinement equation $\varphi(x) = \varphi(2x) + \varphi(2x - 1)$. The segment $[0, 1]$ is indeed a tile, its integer shifts tile the whole line $\R$. 
\end{ex}

\begin{ex} \label{ex2} Consider the matrix $M = \begin{pmatrix}1 & 2 \\ 1 & -1\end{pmatrix}$, $m = |\det M| = 3$. The possible choice of digits is $D(M) = \left\{\begin{pmatrix}0 \\ 0\end{pmatrix}, \begin{pmatrix}1 \\ 0\end{pmatrix}, \begin{pmatrix}0 \\ 1\end{pmatrix}\right\}$. The corresponding set $G$ is depicted in Fig. \ref{extile}. Fig. \ref{exaff} illustrates the self-affinity of the set $G$, i.e., it shows the partition of $G$ to $m = 3$ affinely-similar copies. In this case the set $G$ is a tile, the tiling of the plane with its integer shifts is depicted in the Fig. \ref{extiling}. The indicator $\varphi = \chi_{G}$ satisfies the refinement equation 
$$\varphi(x) = \varphi(Mx) + \varphi\left(Mx - \begin{pmatrix}1 \\ 0\end{pmatrix}\right) + \varphi\left(Mx - \begin{pmatrix}0 \\ 1\end{pmatrix}\right).$$
\end{ex}

\begin{figure}[ht!]
\begin{minipage}[h]{0.32\linewidth}
\center{\begin{subfigure}[t]{\textwidth}{\includegraphics[width=1\linewidth]{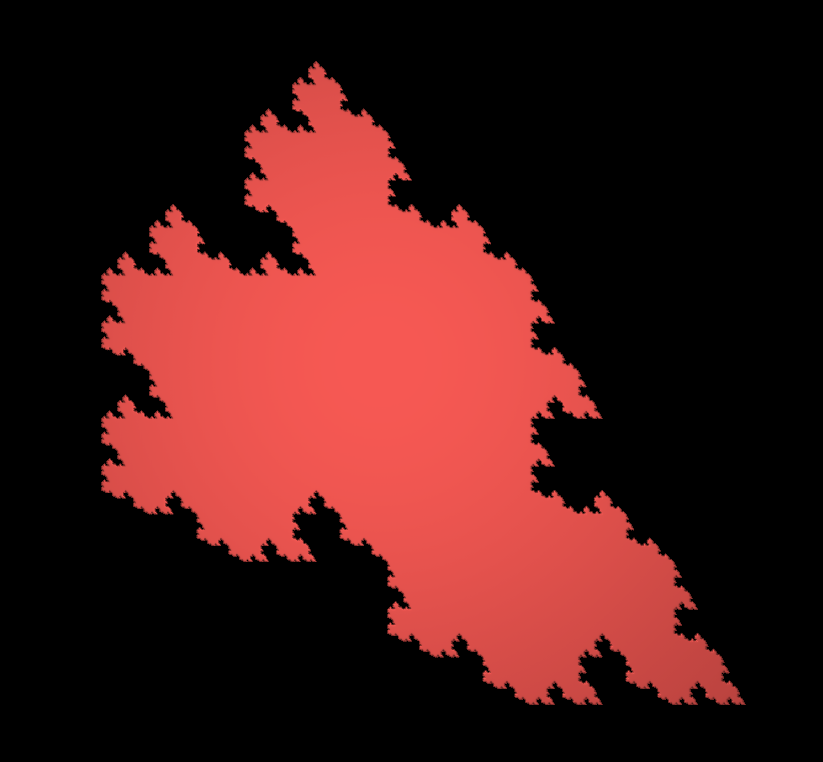} \caption{Set $G$.} \label{extile} }\end{subfigure}}
\end{minipage}
\hfill
\begin{minipage}[h]{0.32\linewidth}
\center{\begin{subfigure}[t]{\textwidth}{\includegraphics[width=1\linewidth]{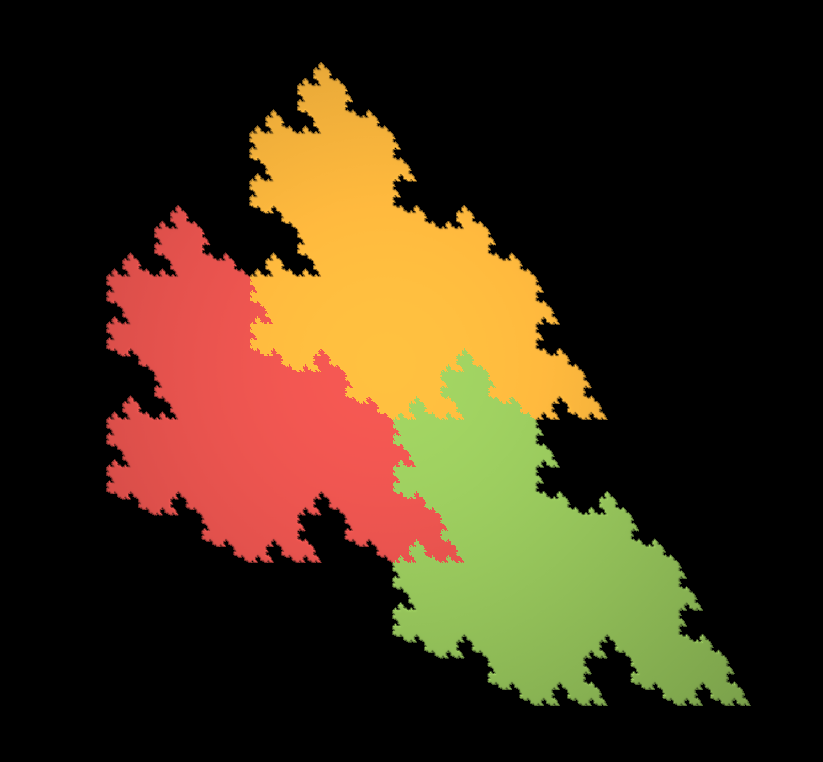} \caption{Self-affinity of $G$.} \label{exaff} }\end{subfigure}}
\end{minipage}
\hfill
\begin{minipage}[h]{0.32\linewidth}
\center{\begin{subfigure}[t]{\textwidth}{\includegraphics[width=1\linewidth]{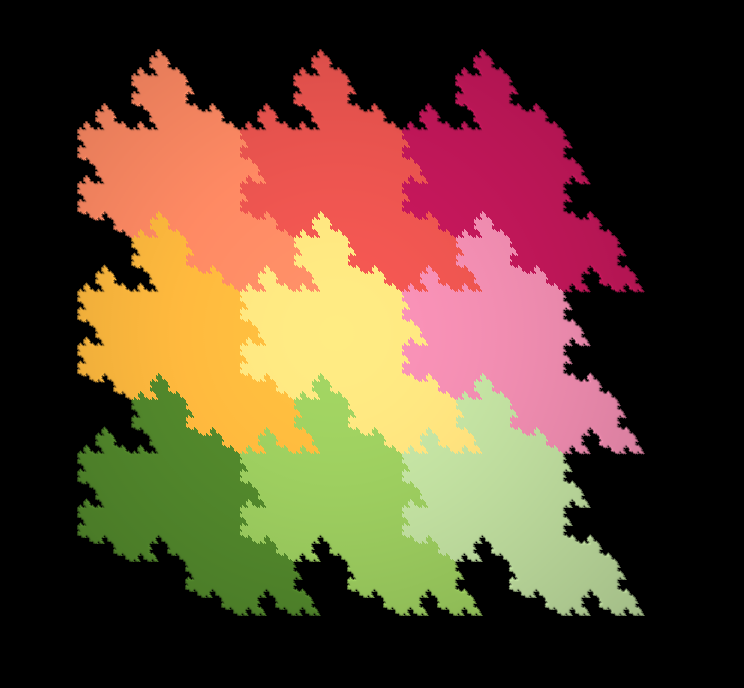} \caption{Tiling of the plane.} \label{extiling} }\end{subfigure}}
\end{minipage}
\caption{The tile $G$ from example \ref{ex2} and its properties.}
\label{pic_tile}
\end{figure} 

Each tile defines its own multivariate Haar basis in $\R^d$ (the construction is described, for example, in \cite{LW97, GH, GM, CHM}). Unlike the univariate case, in which the Haar system is generated by the shifts and contractions of a unique function, in the multivariate case $m - 1= |\det M| - 1$ generating functions required. The case $|\det M| = 2$ is especially interesting since there is only one generating function. In what follows, we mainly restrict ourselves to this case. 
\end{section}

\begin{section}{The classical B-splines}\label{classic}

Recall that the univariate cardinal \textit{B-spline} of order $n$, denoted as $B_n$, is the convolution of $n + 1$ functions $\chi_{[0, 1]}$ (see Fig. \ref{oned}). In particular, $B_0 = \chi_{[0, 1]}$, $B_1 = \chi_{[0, 1]} * \chi_{[0, 1]}$, etc. 

\begin{figure}[ht]
\centering
\includegraphics[width=1\linewidth]{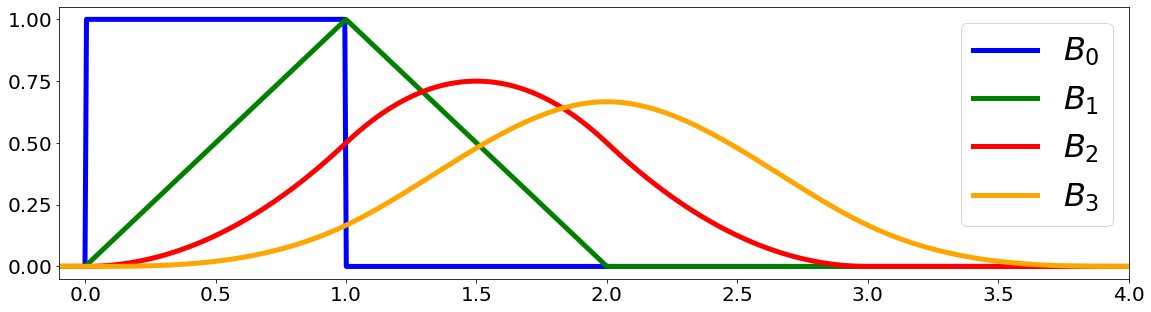}
\caption{Univariate B-splines $B_0, B_1, B_2, B_3$.}
\label{oned}
\end{figure}

B-spline $B_n$ of order $n$ belongs to $C^{n - 1}(\R)$; it is an algebraic polynomial of degree $n$ on each of the segments $[k, k + 1)$, $k = 0, \ldots, n$, out of the segment $[0, n + 1]$ the function $B_n$ is equal to zero. 

Recall that a univariate \textit{refinable function} $\varphi(x)$ with the dilation coefficient $2$ is a solution of the univariate \textit{refinement equation}
\begin{equation}\label{mu}
\varphi(x) = \sum \limits_{k = 0}^{N} c_{k} \varphi(2x - k),
\end{equation}
and the \textit{mask} of this equation is the trigonometric polynomial $$a(\xi) = \frac{1}{2} \sum \limits_{k = 0}^{N} c_{k} e^{-2 \pi i k \xi}.$$

We always assume that $\int \limits_{\R} \varphi(x) dx \ne 0$ and normalize the solutions of refinement equation so that $\int \limits_{\R} \varphi(x) dx = 1$. Applying the Fourier transform to both sides of the equation \eqref{mu}, we obtain 
\begin{equation}
\widehat{\varphi}(2 \xi)=a(\xi) \widehat{\varphi}(\xi).
\label{eq_fou}
\end{equation}
Substituting $s = 0$ and using the equality $\widehat{\varphi}(0) = \int \limits_{\R} \varphi(x) dx = 1$, we obtain  
$\widehat{\varphi}(0)=a(0) \widehat{\varphi}(0)$, hence, $a(0) = 1, \sum_{k=0}^N c_{k}=2.$ 
If two functions $\varphi_1$, $\varphi_2$ satisfy refinement equations with masks $a_1(x)$, $a_2(x)$, then their convolution also satisfies a refinement equation with the mask $a_1(x)a_2(x)$ due to \eqref{eq_fou}.  
Since the function $\varphi(x) = \chi_{[0, 1]}$ satisfies the refinement equation $\varphi(x) = \varphi(2x) + \varphi(2x - 1)$ with the mask $a_0(\xi) = \frac{1 + e^{-2\pi i \xi}}{2}$ (Example \ref{ex1}), the B-spline $B_n$, which is a convolution of $n + 1$ functions $\varphi(x) = \chi_{[0, 1]}$, also satisfies the refinement equation with the mask $a_n(\xi) = a_0^{n + 1}(\xi)$. 
Thus, the function $B_n$ is a solution of the refinement equation with the mask $\left(\frac{1 + e^{-2\pi i \xi}}{2}\right)^{n + 1}$.  The coefficients of this equation are   
 $c_0 = 2^{-n}{{{n+1}\choose0}}$, $c_1 = 2^{-n}{{n+1}\choose1}$, $\ldots$, $c_{n + 1} = 2^{-n}{{n+1}\choose {n+1}}$, where ${{{n+1}\choose k}}$ are the binomial coefficients. 

The classical generalization of B-splines to multivariate functions is a direct product of several univariate B-splines: $B_n(x_1, \ldots, x_d) = B_n(x_1) \cdots B_n(x_d)$. This function also satisfies a refinement equation. Its mask is $a_n(\xi_1, \ldots, \xi_d) = a_n(\xi_1) \cdots a_n(\xi_d)$. In particular, in two-dimensional case the B-spline of zero order has the form  
$$B_0(x, y) = \chi_{[0, 1]}(x) \chi_{[0, 1]}(y) = \chi_{[0, 1]^2}(x, y).$$ 
Its refinement equation can be obtained by multiplication of the univariate refinement equations: 
\begin{multline*}
B_0(x, y) = \left(\chi_{[0, 1]}(2x) + \chi_{[0, 1]}(2x - 1)\right)\left(\chi_{[0, 1]}(2y) + \chi_{[0, 1]}(2y - 1)\right) = \\
= B_0(2x, 2y) + B_0(2x - 1, y) + B_0(2x, 2y - 1) + B_0(2x - 1, 2y - 1), 
\end{multline*}
or by obtaining its coefficients from the mask $a_0(\xi_1, \xi_2)$. 
The equation is illustrated in Fig. \ref{2d0}. 
The B-spline $B_n(x, y)$ of arbitrary order $n$ is equal to the convolution of $n + 1$ B-splines $B_0(x, y)$: 
\begin{multline*}
B_n(x, y) = B_n(x)B_n(y) = (\chi_{[0, 1]}(x) * \ldots * \chi_{[0, 1]}(x)) (\chi_{[0, 1]}(y) * \ldots * \chi_{[0, 1]}(y)) = \\
= B_0(x, y) * \ldots * B_0(x, y).
\end{multline*}
The linear B-spline $B_1(x, y)$ is depicted in Fig. \ref{2d1}. The same holds for the case of $d$ variables. 
 
 \begin{figure}[h]
\begin{center}
\begin{minipage}[h]{0.4\linewidth}
\includegraphics[width=1\linewidth]{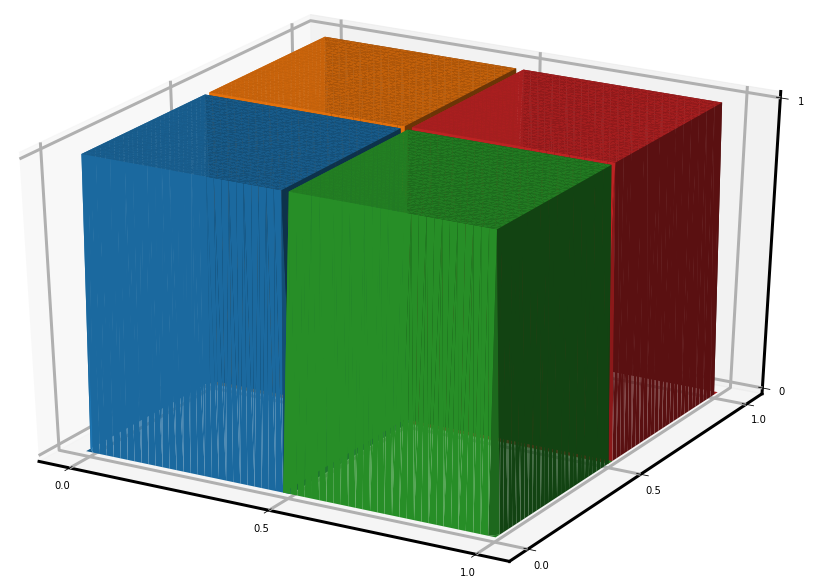}
\caption{Refinement equation for $B_0(x, y)$.}
\label{2d0}
\end{minipage}
\hfill
\begin{minipage}[h]{0.4\linewidth}
\includegraphics[width=1\linewidth]{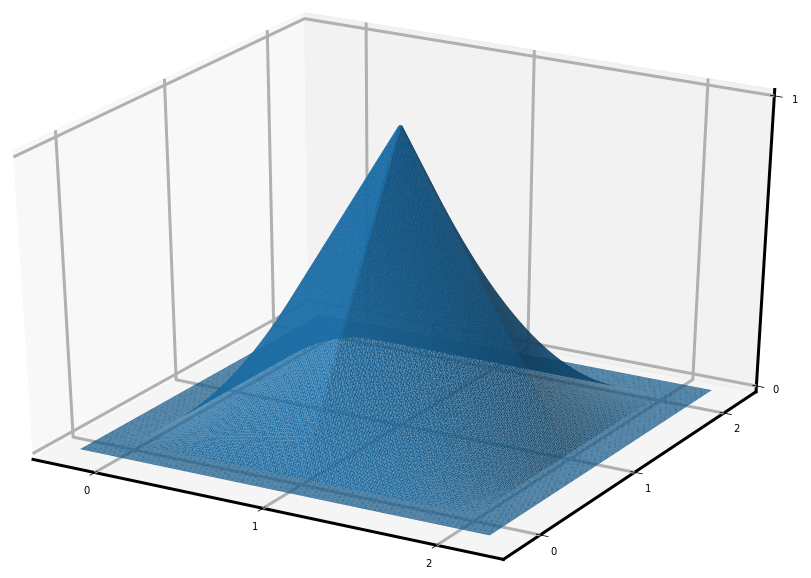}
\caption{Linear B-spline of two variables.}
\label{2d1}
\end{minipage}
\end{center}
\end{figure}

Thus, the B-spline $B_n(x_1, \ldots, x_d)$ is a solution of refinement equation with $(n + 2)^d$ positive coefficients. Since the number of  coefficients grows exponentionally in dimension, the use of the classical $d$-variate B-splines often leads to non-effective algorithms when $d$ is large. One of them, the subdivision algorithm, is considered in detail in Section \ref{par_subdiv}. 
\textit{The cascade algorithm} (the fast discrete wavelet transform) is closely related to this method and is used to obtain the coefficients of wavelet expansion. 
Its complexity also depends on the number of nonzero coefficients in the refinement equation. 
Using another construction of multivariate B-splines one can obtain the less number of coefficients and, in some cases, higher smoothness without loss of the structure and of approximation properties of the classical B-splines. 
In the next section we define B-splines based on tiles. By an appropriate choice of a tile it is possible to obtain only $(n + 2)$ coefficients of refinement equation independent of the dimension $d$. We show that some of these tile B-splines have a higher smoothess than the classical B-splines of the same order. 
\end{section}

\begin{section}{The construction of tile B-splines}\label{tile_spl}
We start with a definition of the tile B-splines. 

\begin{definition} For a given tile $G \subset \R^d$ and for an integer number $n \ge 0$, we say that the convolution of $n + 1$ functions $\chi_G$ is the \textit{tile B-spline} of order $n$ and write $B_n^G$. 
\end{definition}

In particular, $B_0^G = \chi_G$, $B_1^G = \chi_G * \chi_{G}$. 

\begin{definition} The convolution of $n + 1$ functions $\chi_G * \chi_{-G}$ is called the \textit{symmetrized tile B-spline} of order $n$ and is denoted by $\Bs_{n}^G$. 
\end{definition}

We fix a tile $G$ and further use the short notation $B_n = B_n^G, \Bs_n = \Bs_n^G$. 

In the multivariate case we consider refinement equations with a dilation matrix coefficient $M$ of the form 
\begin{equation} \label{multi_mu}
\varphi(x) = \sum \limits_{k \in \Z^d} c_{k} \varphi(Mx - k),
\end{equation}
the \textit{mask} of this equation is the trigonometric polynomial of variables $\xi_1, \ldots, \xi_d$ $$a(\xi) = \frac{1}{m} \sum \limits_{k \in \Z^d} c_{k} e^{-2 \pi i (k, \xi)},$$ where $m = |\det{M}|$.  

The tile B-spline $B_0^G$, i.e., the indicator of the tile $\chi_G$, a.e. satisfies the refinement equation
$$\chi_G(x) = {\sum \limits_{\Delta \in D(M)}{\chi_G(Mx - \Delta)}}, \quad x \in \R^d$$
(see Section \ref{par_tile}). In this case $c_k = 1$ for all $k \in D(M)$ and $c_k = 0$ for $k \notin D(M)$, and the mask is given by the formula 
$$a_0(\xi) = \frac{1}{m} \sum \limits_{\Delta \in D(M)} e^{-2\pi i (\Delta, \xi)}.$$ Similarly to the univariate case, applying the Fourier transform to both sides of the refinement equation \eqref{multi_mu} we obtain 
$$\widehat{\varphi}(\xi) = a(M_1^T \xi) \widehat{\varphi}(M_1^T \xi)$$ or 
\begin{equation} \label{fourie}
\widehat{\varphi}(M^T\xi) = a(\xi) \widehat{\varphi}(\xi).
\end{equation}
From this it follows that if two functions $\varphi_1$, $\varphi_2$ satisfy refinement equations with masks $a_1(x)$, $a_2(x)$ and with the dilation coefficient $M$, then their convolution also satisfies a refinement equation with the mask $a_1(x)a_2(x)$ and with the dilation coefficient $M$. 
Thus, similarly to the univariate case, the mask $a_n$ of the tile B-spline $B_n$ satisfies the formula $a_n = a_0^{n + 1}$. From this the  coefficients of the refinement equation \eqref{multi_mu} of the tile B-splines can be found explicitly: 

\begin{propos}\label{prop_mu}
Let $\varphi = B_n^G$ be a tile B-spline, where the tile $G$ is constructed by the matrix $M$ and the set of digits $D = \{d_0, \ldots, d_{m-1}\}$. 
For every vector $k \in \Z^d$, we denote by $C_k$ the number of its representations of the form $k = s_1 + \ldots + s_{n + 1}$ for all ordered sets $(s_1, \ldots, s_{n + 1})$, $s_i \in D$, with possible repetitions. Then the refinement equation of the tile B-spline $B_n^G$ has the form 
$$\varphi(x) = m^{-n} \sum \limits_{k \in \Z^d} C_k \varphi(Mx - k).$$
\end{propos} 
Note that the numbers $C_k$ can be explicitly expressed using multinomial coefficients. 

The symmetrized tile B-spline also satisfies a refinement equation. Indeed, if $\varphi(x)$ satisfies a refinement equation with coefficients $c_k$ and with mask $a(\xi)$, then $\varphi(-x)$ satisfies the refinement equation with the coefficients $c_{-k}$ and with the mask $\bar{a}(\xi)$. 
Therefore, $\chi_G * \chi_{-G}$ corresponds to the refinement equation with mask $|a_0|^2$, and $\Bs_n$ corresponds to the mask $|a_0|^{2(n + 1)}$. Note that for symmetrized tile B-splines coefficients of mask and its values for all $\xi \in \R^d$ are real.  

As the set of the coefficients of the polynomial $a_n$ depends only on the order $n$ and on the digit set $D$, the set of coefficients of  refinement equation of the tile B-spline $B_n$ depends only on $n$, $D$, and does not depend on the matrix $M$. 
Thus, we have  
\begin{cor} If the digit set $D$ is fixed, then all B-splines of the same order are defined by the same refinement equation up to the change of the dilation matrix $M$. The same holds for the symmetrized B-splines. 
\end{cor}

\begin{remark} \label{rem_constr_phi} The construction of the tile B-spline $B_n$ by definition requires the calculation of convolutions that uses the numerical integration. However, the function $B_n$ can be found in a different way, as a solution of the corresponding refinement equation. Every refinable function can be computed precisely on an arbitrarily dense lattice using the products of special transition matrices (we discuss them in more detail in  Section \ref{par_smooth}). 
In particular, the values of $B_n(k)$ at integer points $k \in \Z^d$ coincide with the components of the eigenvector $v$ of the transition matrix corresponding to the eigenvalue one. The values of the function $B_n(x)$ on the lattice $M^{-1}\Z^d$ can be obtained by the multiplication of the transition matrices by the vector $v$. The next multiplications by the transition matrices allow us to find $B_n(x)$ for $x \in M^{-2}\Z^d, M^{-3}\Z^d$, etc. Thus, after several iterations we get the precise values of the function $B_n(x)$ on the refined lattice. 
\end{remark}

\begin{propos} \label{prop1} The integer translates $\{B_n(x - k)\}_{k \in \Z^d}$ of the tile B-spline $B_n$ form a Riesz basis of their linear span. The same holds for the translates of the symmetrized tile B-splines. 
\end{propos}

We postpone the proof to the Section \ref{par_subdiv}. 

The Proposition \ref{prop1} implies that the integer shifts of the tile B-spline $\varphi(x) = B_n(x)$ generate a multiresolution analysis (MRA), and correspondingly a wavelet system (see, for example, \cite{KPS}).  
Besides, we can apply the orthogonalization procedure to our refinable function $\varphi(x) = B_n(x)$ and obtain a new function $\varphi_1$ which generates the same MRA and possesses orthogonal integer shifts. 
This will be done in Section \ref{par_orthogonalization}. Thus, we will obtain the orthogonalized tile B-splines and further corresponding orthonormal wavelet systems. 

\begin{propos} \label{prop_poly} Linear combinations of integer shifts $\{B_n(x - k)\}_{k \in \Z^d}$ of the tile B-spline $B_n$ generate algebraic polynomials of degree at most $n$. The same holds for the shifts of the symmetrized tile B-splines. 
\end{propos}
\begin{proof}
This holds for the order $n = 0$ since $B_0(x - k) = \chi_{G}(x - k)$ and the tile satisfies the property $\sum_{k \in \Z^d}\chi_G(x - k) \equiv 1$ almost everywhere. Thus, the linear combinations of the shifts $B_0(x - k)$ indeed generate identical constants. 

It is known that the shifts of a compactly supported function $\varphi$ generate algebraic polynomials of degree at most $n$ if and only if its Fourier transform $\varphi(\xi)$ has zeros of order at least $n + 1$ at all integer points except for zero (due to the Strang-Fix theorem, see \cite{SF, BVR, KPS}).  
Applying this statement first for $\varphi = B_0(x)$, we obtain that $\widehat{B}_0(\xi)$ has zeros of order one at all integer nonzero points. Since $\widehat{B}_n(\xi) = \widehat{B}_0^{n + 1}(\xi)$, it follows that the order of zeros of $\widehat{B}_n(\xi)$ is at least $n + 1$. Now we apply the converse statement for $\varphi = B_n$, this completes the proof. 
\end{proof}

\begin{cor} \label{cor_ord} The order of approximation by integer shifts of the tile B-spline $B_n$ is equal to $n$. 
\end{cor}
This means that the distance between an arbitrary smooth function $f$ and the space generated by the functions $\{B_n(a x - k)\}_{k \in \Z^d}$ is equal to $O(a^{-(n + 1)})$ as $a \to +\infty$ (see, for example, \cite{BVR94}). 

\begin{remark}\label{zz}
The tile B-splines were also considered in the work \cite{Zakh2}, where they were called elliptic refinable functions. They were defined using the Fourier transform. Analogues of Propositions \ref{prop1}, \ref{prop_poly} and  Corollary \ref{cor_ord} were proved for them in the isotropic case (when $M$ is similar to the orthogonal matrix multiplied by a scalar). 

The case of plane tile B-splines was investigated in the work \cite{Zube} under the name of $\alpha$-splines. 
A certain complex number $\alpha \in \mathbb{Q}[i]$ played a role of the matrix $M$ in the refinement equation, $m = |\alpha|^2$. This allows us, in particular, to obtain some analogues of the tile B-splines with two digits. The same work studies subdivision schemes based on the $\alpha$-splines; and regularity issue was left as an open problem. 
\end{remark}

\begin{subsection}{The case of two digits ($m = 2$)}
In what follows we often consider the case when $m = |\det M| = 2$ and therefore $D(M) = \{0, e\}$. We call such tiles 2-tiles or two-digit tiles. They possess many useful properties, some of them are considered below. 

As it was established in~\cite[Proposition 2.2]{KL00}, every 2-tile is centrally-symmetric. We include the proof for the convenience of the reader. 
\begin{propos}\label{p.symm} 
Every 2-tile is centrally-symmetric. 
\end{propos}
\begin{proof}
Let $D(M) = \{0 , e\}$ be the set of digits. 
Denote by  $c \, = \, \frac12 \, \sum\limits_{j=1}^{\infty} M^{-j} e$.  
Then the 2-tile $G \, = \left\{\, c \, + \, \sum \limits_{j=1}^{\infty} \pm \frac12 M^{-j} e\right\}$. 
The set of points $\left\{ \sum \limits_{j=1}^{\infty} \pm \frac12 M^{-j} e\right\}$ is symmetric about the origin, therefore $G$ is symmetric about the point~$c$. 
\end{proof}

In case of a centrally-symmetric tile, when $G = c + G_0 = c - G_0$, the following holds (here and in the sequel $\int$ means $\int_{\mathbb{R}^d}$)
\begin{multline*}
\chi_G * \chi_G(y) = \int \chi_{G}(x) \chi_{G}(y - x) dx = \int \chi_{G_0}(x - c) \chi_{G_0}(y - x - c) dx = \\  \int \chi_{G_0}(x) \chi_{G_0}(y - x - 2c) dx  = \chi_{G_0}*\chi_{G_0} (y - 2c);\end{multline*}
\begin{multline*}\chi_G * \chi_{-G}(y) = \int \chi_{G}(x) \chi_{-G}(y - x) dx = \int \chi_{G_0}(x - c) \chi_{G_0}(y - x + c) dx = \\ \int \chi_{G_0}(x) \chi_{G_0}(y - x) dx  = \chi_{G_0}*\chi_{G_0} (y).\end{multline*}
Therefore, for centrally-symmetric tiles, in particular, for all 2-tiles, the B-splines $B_{2n}$ and $\Bs_n$ differ only by a shift.  
Further we restrict ourselves only on the B-splines based on 2-tiles. 

Now we need the notion of isotropic tile:  
\begin{definition}
The tile is called {\em isotropic} if it is generated by an {\em isotropic matrix}~$M$, 
i.e., the diagonalizable matrix that has eigenvalues of equal moduli. 
\end{definition}
An isotropic matrix is affinely-similar to an orthogonal matrix multiplied by a scalar. The most popular tiles in applications are isotropic.

The 2-tiles have been studied in an extensive literature (see, for example, \cite{BG, Gel, G81, GJ, B10, B91, FG, LW95, Zai, Zakh}). It is known that for every $d$, there is a finite number of different 2-tiles in $\R^d$ up to affine similarity. For instance, there exist exacty three 2-tiles in $\R^2$. We call them the Square, the Dragon and the Bear (in the literature the terms respectively ``square'', ``twindragon'', ``tame twindragon'' are also used). All of them are isotropic. For their construction we can choose, for example, the matrices 
\begin{equation}\label{twotiles}
M_S = \begin{pmatrix} 0 & -2 \\ 1 & 0\end{pmatrix}, M_D = \begin{pmatrix} 1 & 1 \\ -1 & 1\end{pmatrix}, M_B = \begin{pmatrix} 1 & -2 \\ 1 & 0\end{pmatrix}
\end{equation}
 correspondingly, and the set of digits $D = \left\{\begin{pmatrix} 0 & 0\end{pmatrix}, \begin{pmatrix} 1 & 0\end{pmatrix}\right\}$. If we change the digits, the set transforms affinely. Further we shall use  these matrices and digits. 
The partition of 2-tiles to two affinely-similar parts is shown in Fig. \ref{pic1_planar}. The tiling of the plane with their integer shifts is in Fig. \ref{pic2_planar}. There are seven 2-tiles in $\R^3$, only one of them (the cube) is isotropic.

\begin{figure}[ht!]
\begin{minipage}[h]{0.3\linewidth}
\center{\includegraphics[width=1\linewidth]{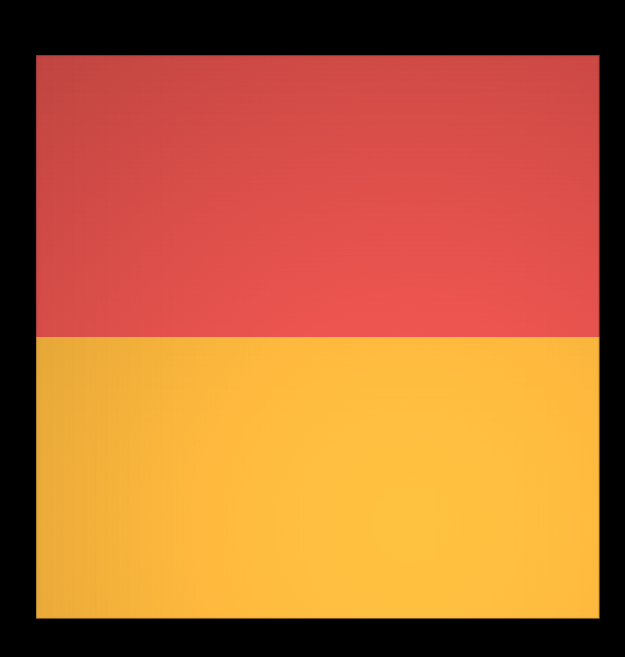}} Square \\
\end{minipage}
\hfill
\begin{minipage}[h]{0.3\linewidth}
\center{\includegraphics[width=1\linewidth]{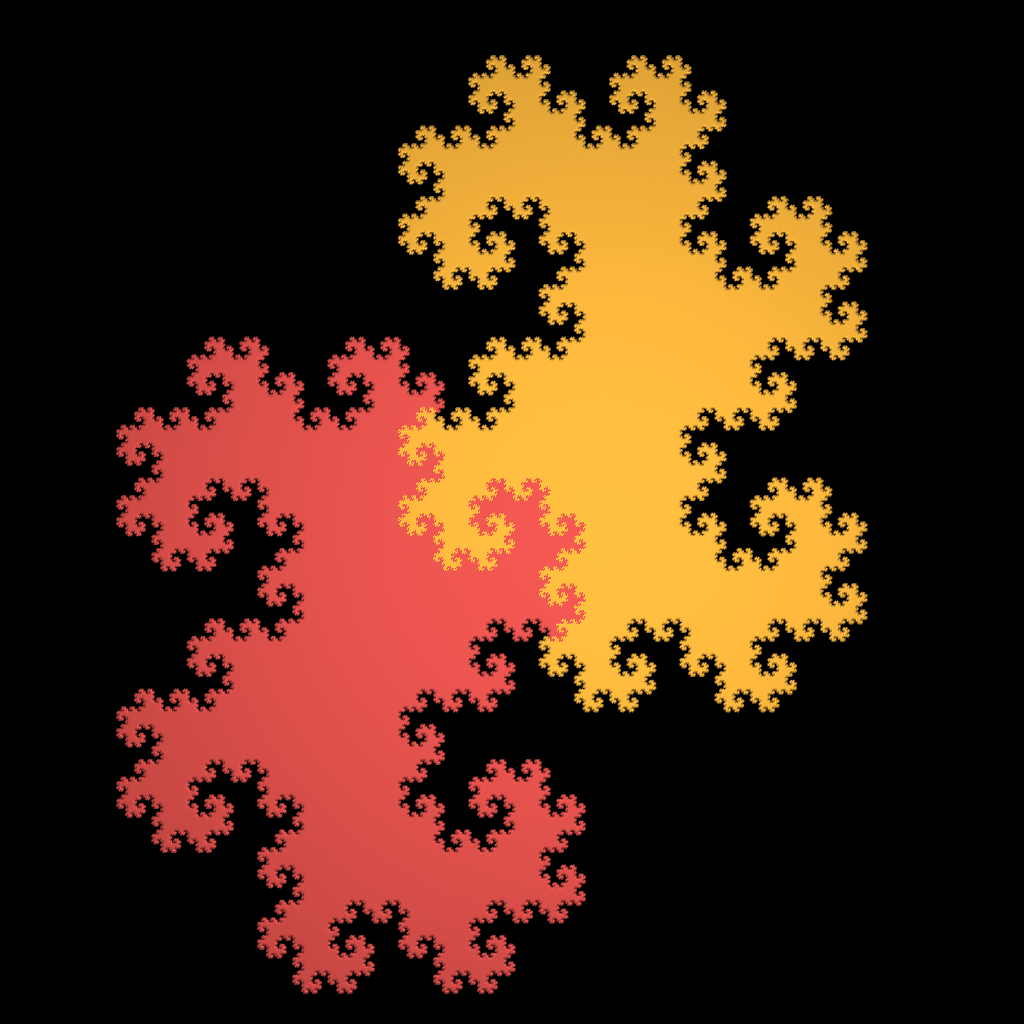}} Dragon \\
\end{minipage}
\hfill
\begin{minipage}[h]{0.3\linewidth}
\center{\includegraphics[width=1\linewidth]{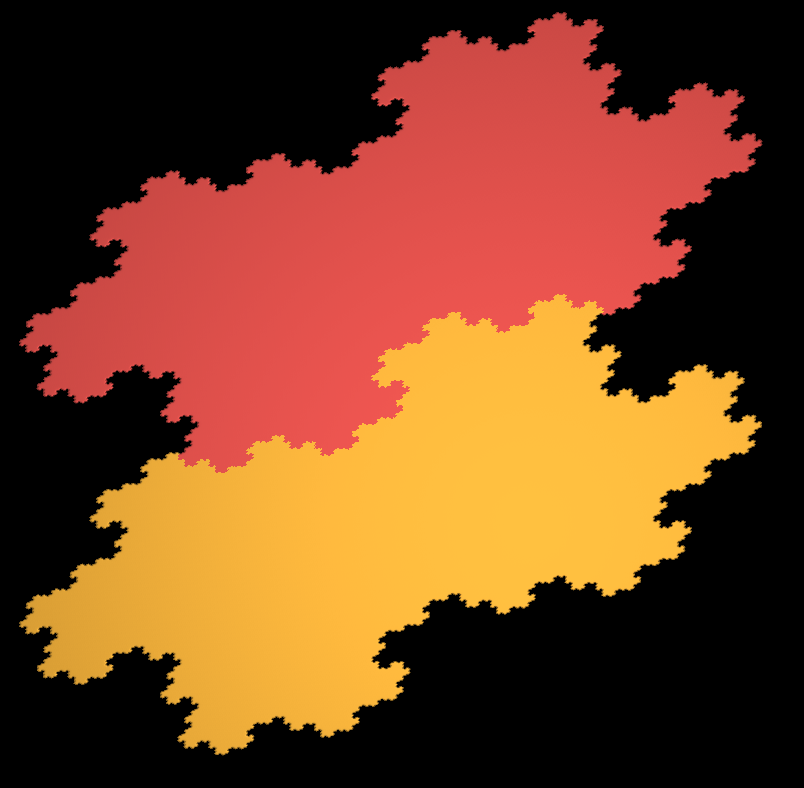}} Bear \\
\end{minipage}
\caption{The partition of the plane 2-tiles to two affinely-similar parts}
\label{pic1_planar}
\end{figure}

\begin{figure}[ht!]
\begin{minipage}[h]{0.3\linewidth}
\center{\includegraphics[width=1\linewidth]{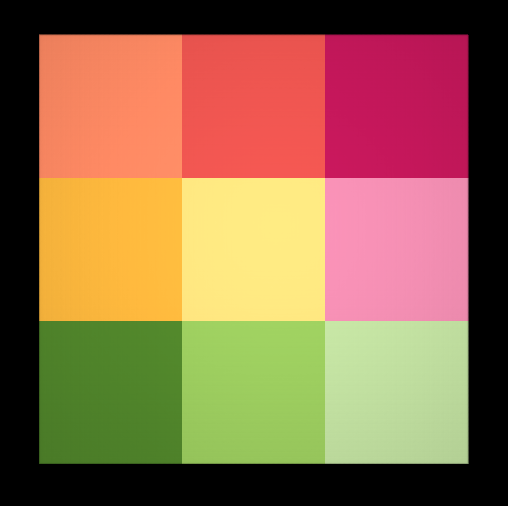}} Square \\
\end{minipage}
\hfill
\begin{minipage}[h]{0.3\linewidth}
\center{\includegraphics[width=1\linewidth]{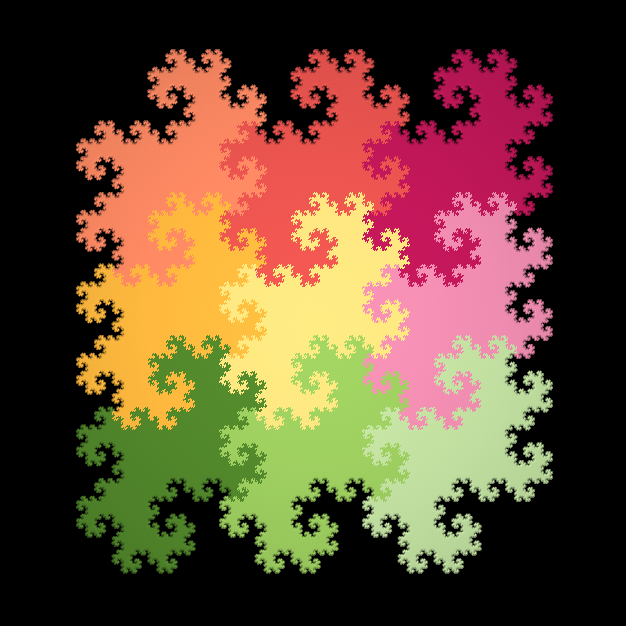}} Dragon \\
\end{minipage}
\hfill
\begin{minipage}[h]{0.3\linewidth}
\center{\includegraphics[width=1\linewidth]{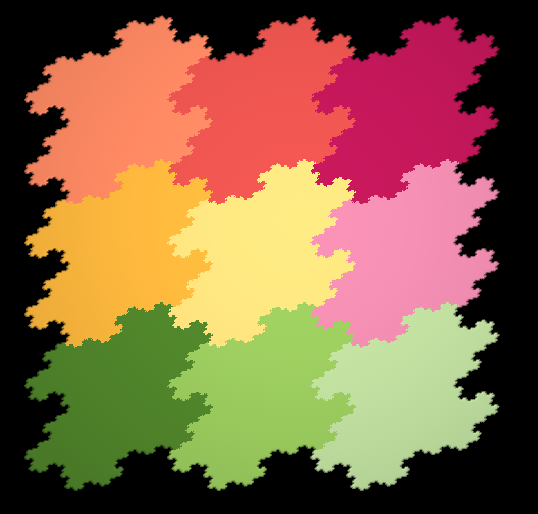}} Bear \\
\end{minipage}
\caption{The tiling of $\R^2$ by plane 2-tiles}
\label{pic2_planar}
\end{figure}

In the isotropic case the problem of classification of 2-tiles up to affine similarity can be solved completely \cite{PZ}. This classification turns out to be rather simple. In odd dimensions~$d=2k+1$ all isotropic 2-tiles are parallelepipeds. In even dimensions~$d=2k$ there exists three types of isotropic 2-tiles up to affine similarity. These are the parallelepiped, the direct product of $k$ (two-dimensional) Dragons and the direct product of $k$ (two-dimensional) Bears. 

In non-isotropic case finding the number $N(d)$ of non-equivalent classes of 2-tiles for every $d$ is reduced to the finding the total number of expanding monic polynomials with the constant coefficient $\pm 2$. See \cite{PZ}, where the following estimate was obtained  
$$\frac{d^2}{16}\,  - \, \frac{43d}{36}\, - \frac56 \, \le \, N(d) \, \le \, 2^{\, d\, \bigl(1+\frac{16 \ln \ln d}{\ln d}\bigr)}.$$ 

We call the tile B-splines constructed by 2-tiles according to their names but with the shift of the index by one. 
Thus, the indicator $B_0$ of the Bear tile is Bear-1, the convolution $B_{k - 1}$ of $k$ such functions is Bear-$k$. The shift of the index is due to the traditional terminology for the univariate B-splines, where the function $B_{k - 1}(x)$ is the spline of order $k - 1$, which is a convolution of $k$ functions $\chi_{[0, 1]}$. In our case the B-splines are defined not by polynomials but by convolutions of indicators. Therefore, it is more natural to use the number of multipliers in the name of the tile B-splines. 

The coefficients of the refinement equations of the tile B-splines were found in general case in Proposition \ref{prop_mu}. They are quite simple for plane tiles: 

\begin{cor} Let the tile B-spline $\varphi = B_n^G$ be constructed by the plane 2-tile with matrix $M$ and with digits $D = \left\{\begin{pmatrix} 0 & 0\end{pmatrix}, \begin{pmatrix} 1 & 0\end{pmatrix}\right\}$. Denote by $C_k = 2^{-n}{{{n+1}\choose k}}$ for $k = \{0, \ldots, n + 1\}$. 
Then the B-spline $B_n^G$ satisfies the refinement equation 
\begin{equation}\label{eq_two}
\varphi(x) = \sum \limits_{k \in \{0, 1, \ldots, n + 1\}}C_k \varphi \left(Mx - \begin{pmatrix} k \\ 0\end{pmatrix}\right).
\end{equation}
\end{cor}

Thus, all tile B-splines based on Bear, Dragon and Square tiles may be computed using refinement equation \eqref{eq_two}, see Remark \ref{rem_constr_phi}. The Fig. \ref{spl_Bear} shows\footnote{The source code for the programs by which all the computations were done is available on github \cite{gitTZ}.}  Bear-1, ..., Bear-4, the Fig. \ref{spl_Dragon} shows Dragon-1, ..., Dragon-4, the Fig. \ref{spl_Rect} shows Square-1, ..., Square-4. 

In Section \ref{par_smooth} we compute the H{\" o}lder regularity of these splines and establish that the Bear-2 is not from $C^1$, the Bear-3 is from $C^2$, and the Bear-4 is from $C^3$. Thus, the Bear-3 and Bear-4 have higher regularity than the Square-3 and Square-4 respectively even if it seems paradoxical (see Theorem \ref{th_smooth}).

\begin{figure}[ht]
\begin{center}
\begin{minipage}[h]{0.48\linewidth}
\center{\includegraphics[width=1\linewidth]{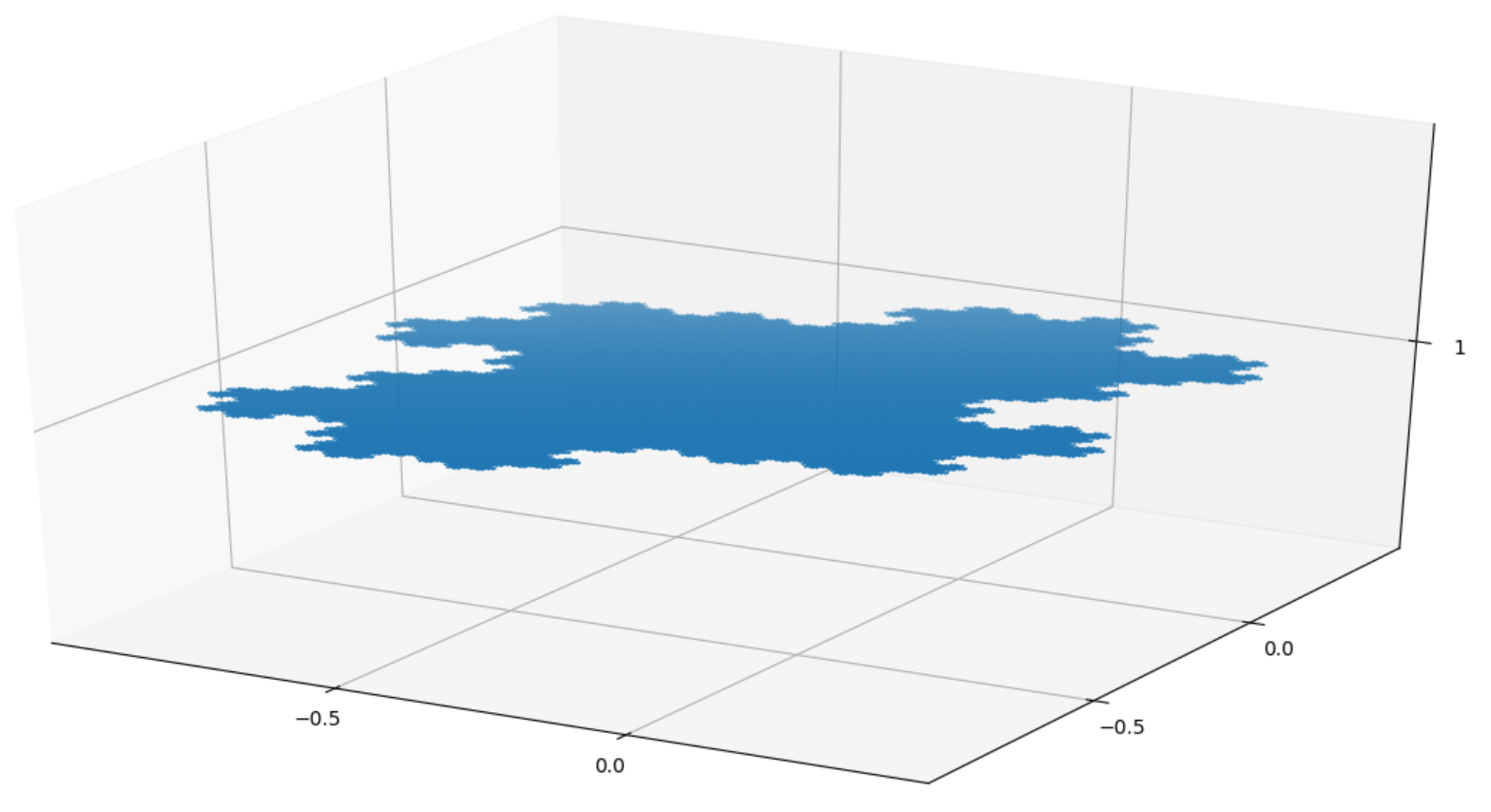}} \\ Bear-1.
\end{minipage}
\hfill
\begin{minipage}[h]{0.48\linewidth}
\center{\includegraphics[width=1\linewidth]{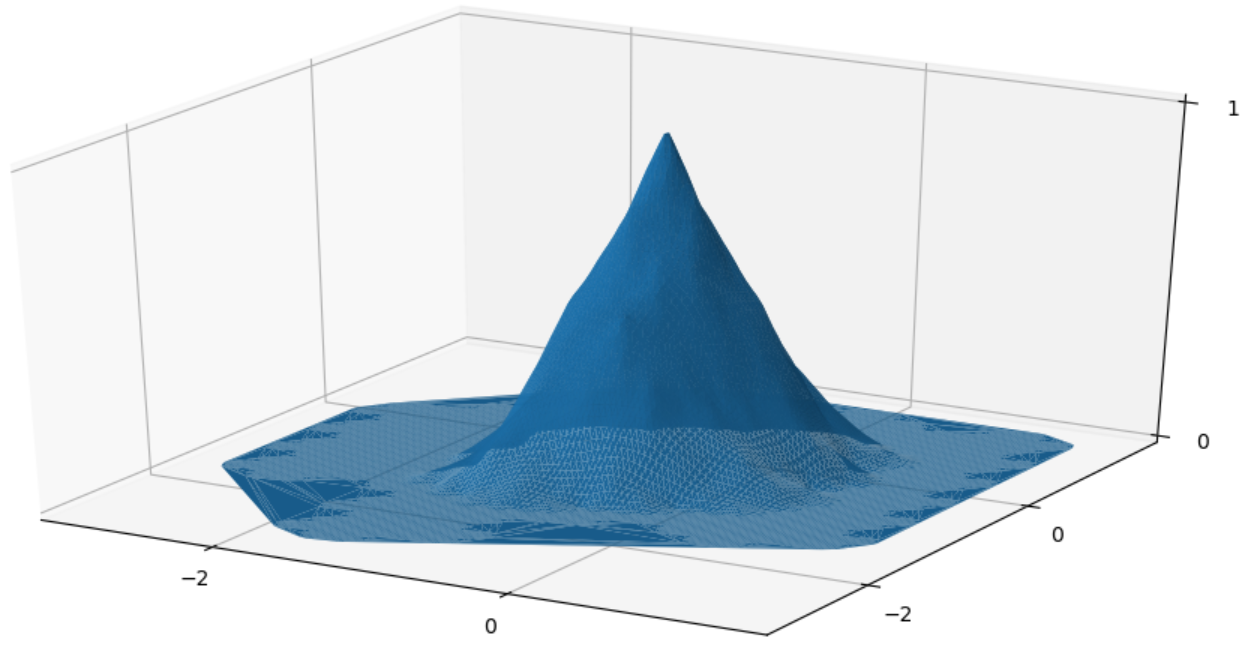}} \\ Bear-2.
\end{minipage}
\vspace{\baselineskip}
\begin{minipage}[h]{0.48\linewidth}
\center{\includegraphics[width=1\linewidth]{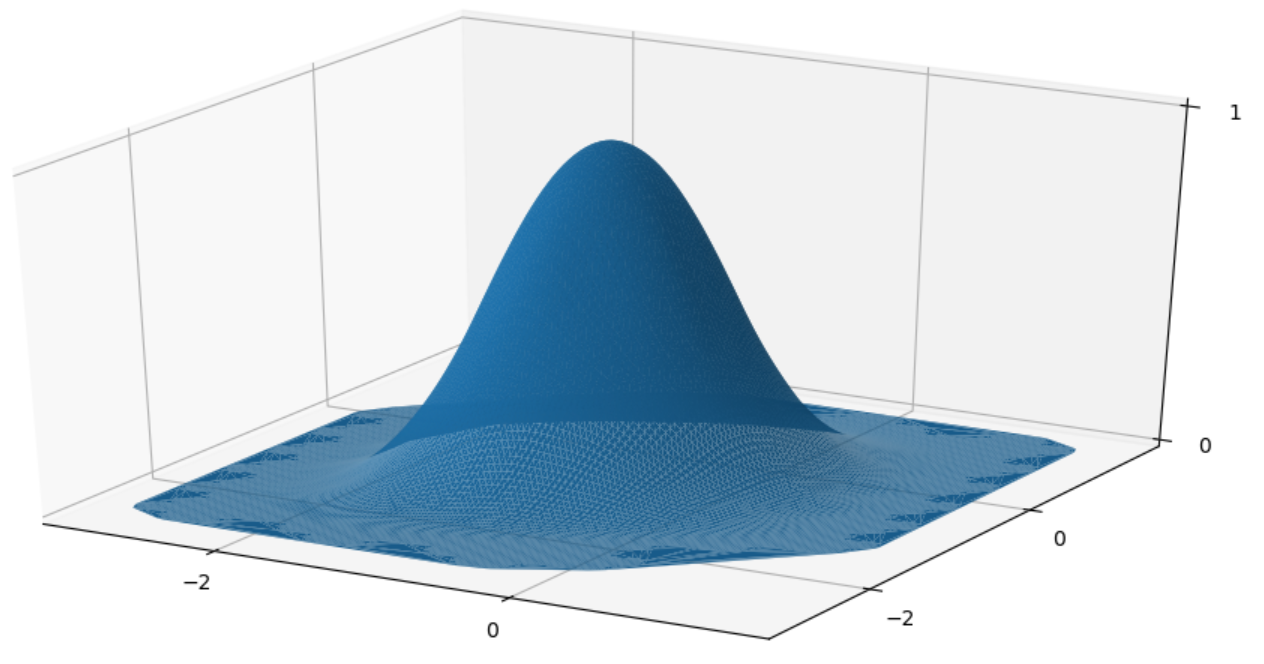}} \\ Bear-3.
\end{minipage}
\hfill
\begin{minipage}[h]{0.48\linewidth}
\center{\includegraphics[width=1\linewidth]{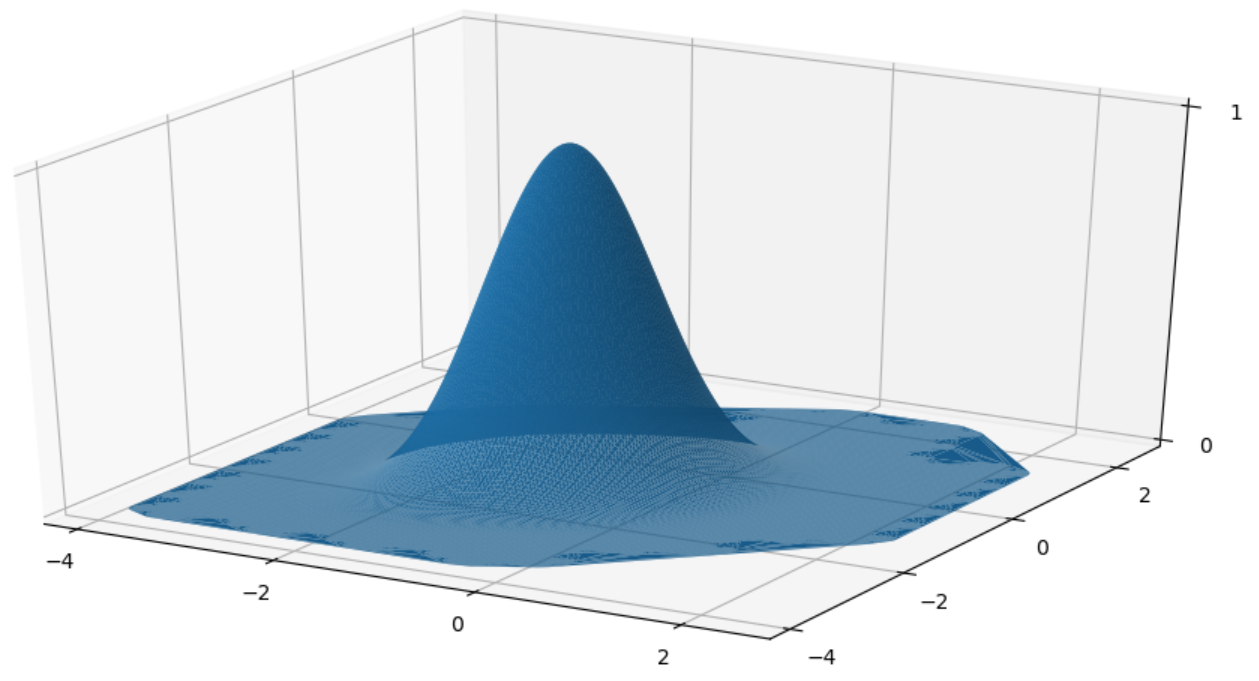}} \\ Bear-4. 
\end{minipage}
\caption{Tile B-splines with Bear matrix.}
\label{spl_Bear}
\end{center}
\end{figure}

\begin{figure}[ht]
\begin{center}
\begin{minipage}[h]{0.48\linewidth}
\center{\includegraphics[width=1\linewidth]{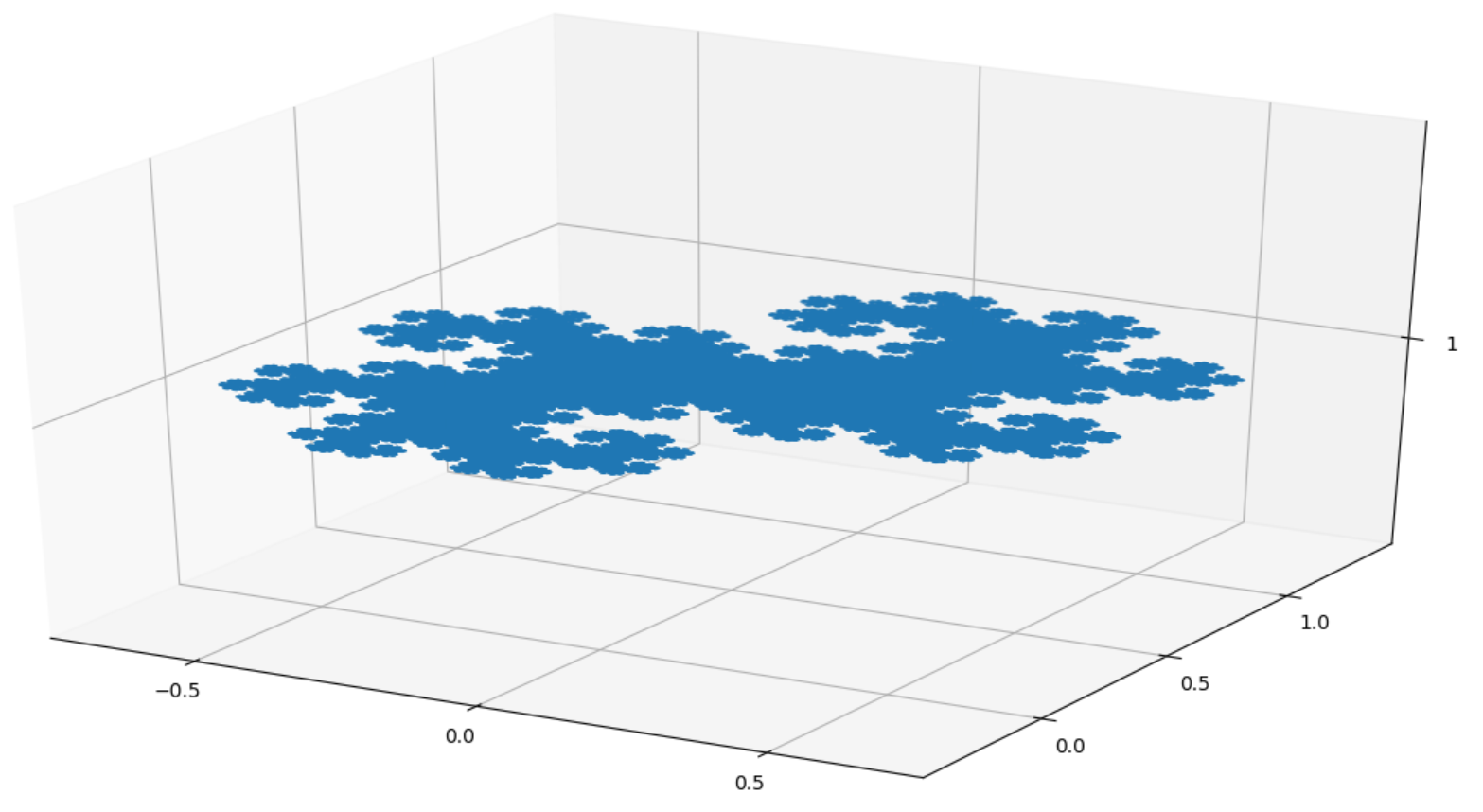}} \\ Dragon-1.
\end{minipage}
\hfill
\begin{minipage}[h]{0.48\linewidth}
\center{\includegraphics[width=1\linewidth]{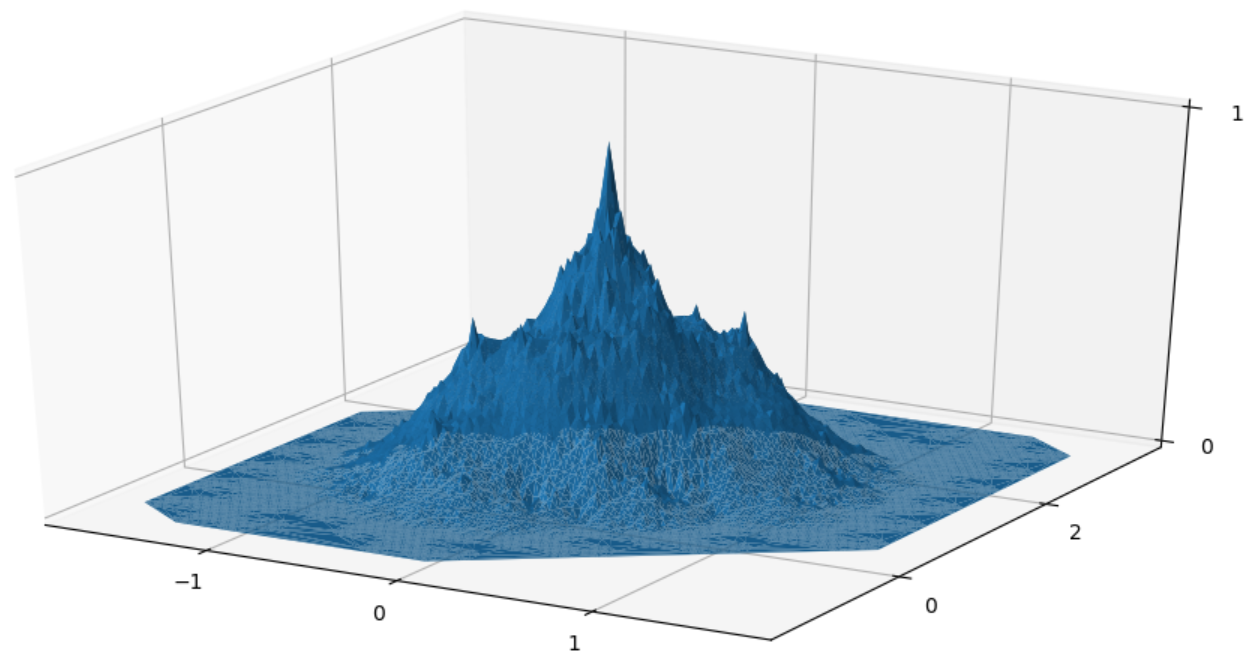}} \\ Dragon-2.
\end{minipage}
\vspace{\baselineskip}
\begin{minipage}[h]{0.48\linewidth}
\center{\includegraphics[width=1\linewidth]{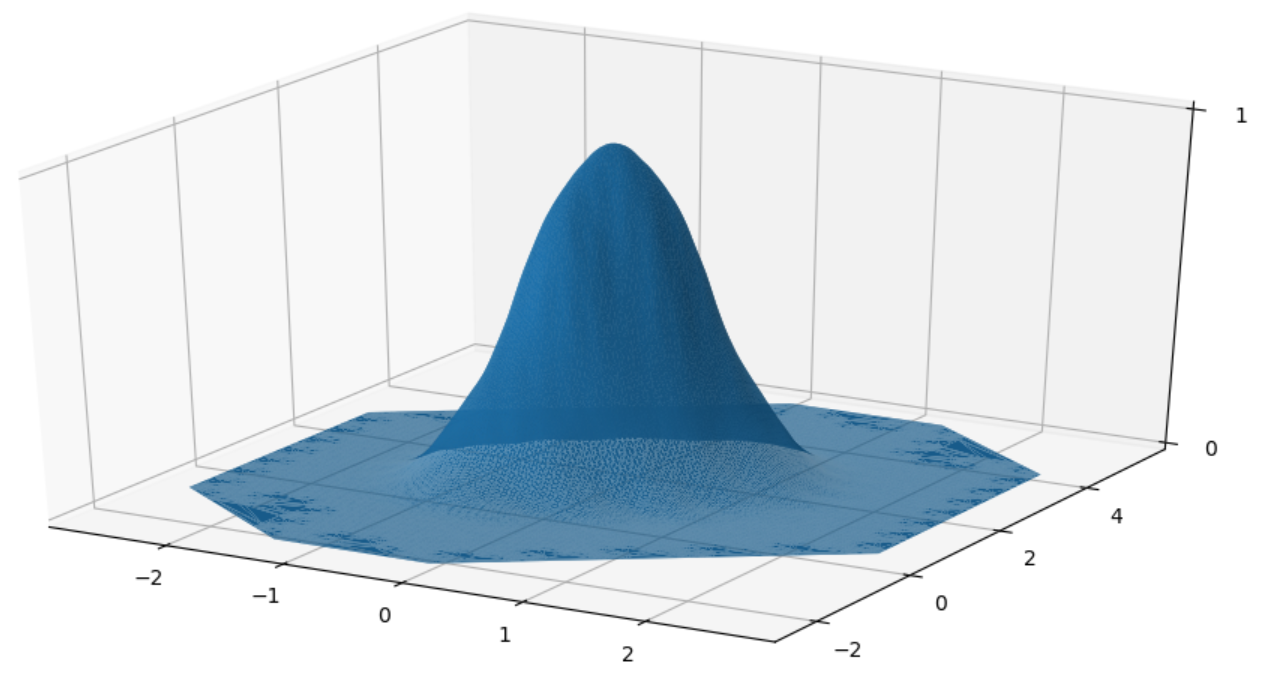}} \\ Dragon-3.
\end{minipage}
\hfill
\begin{minipage}[h]{0.48\linewidth}
\center{\includegraphics[width=1\linewidth]{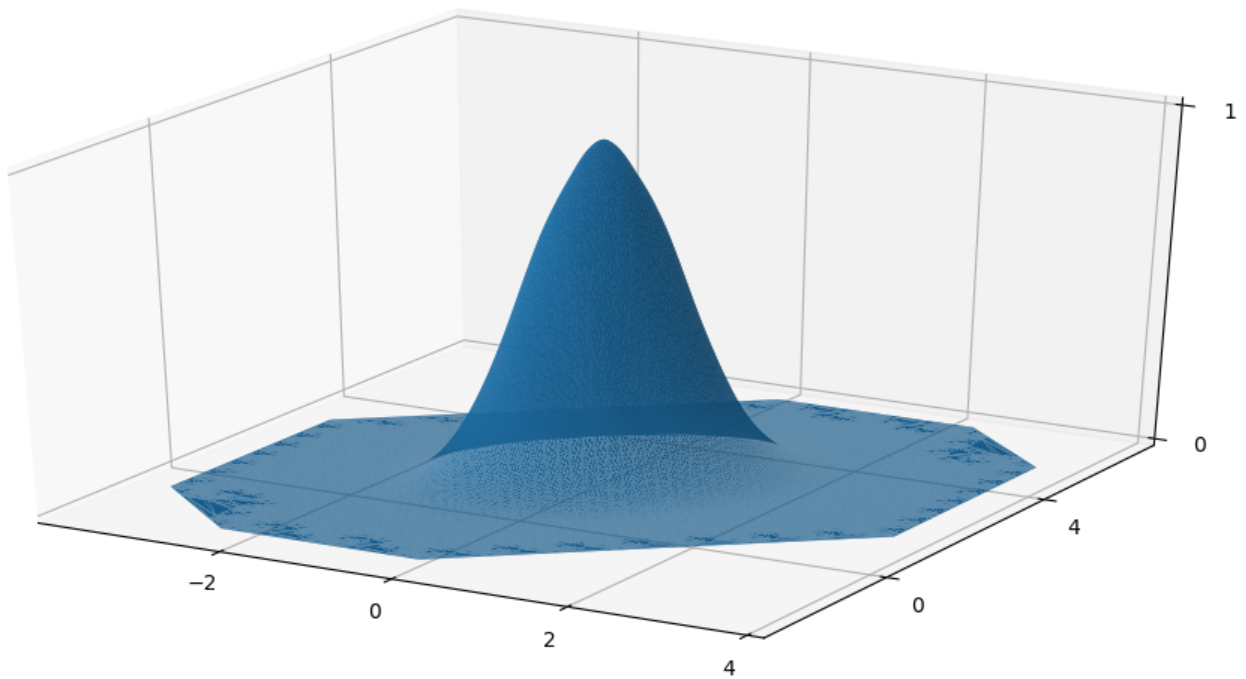}} \\ Dragon-4. 
\end{minipage}
\caption{Tile B-splines with Dragon matrix.}
\label{spl_Dragon}
\end{center}
\end{figure}

\begin{figure}[ht]
\begin{center}
\begin{minipage}[h]{0.48\linewidth}
\center{\includegraphics[width=1\linewidth]{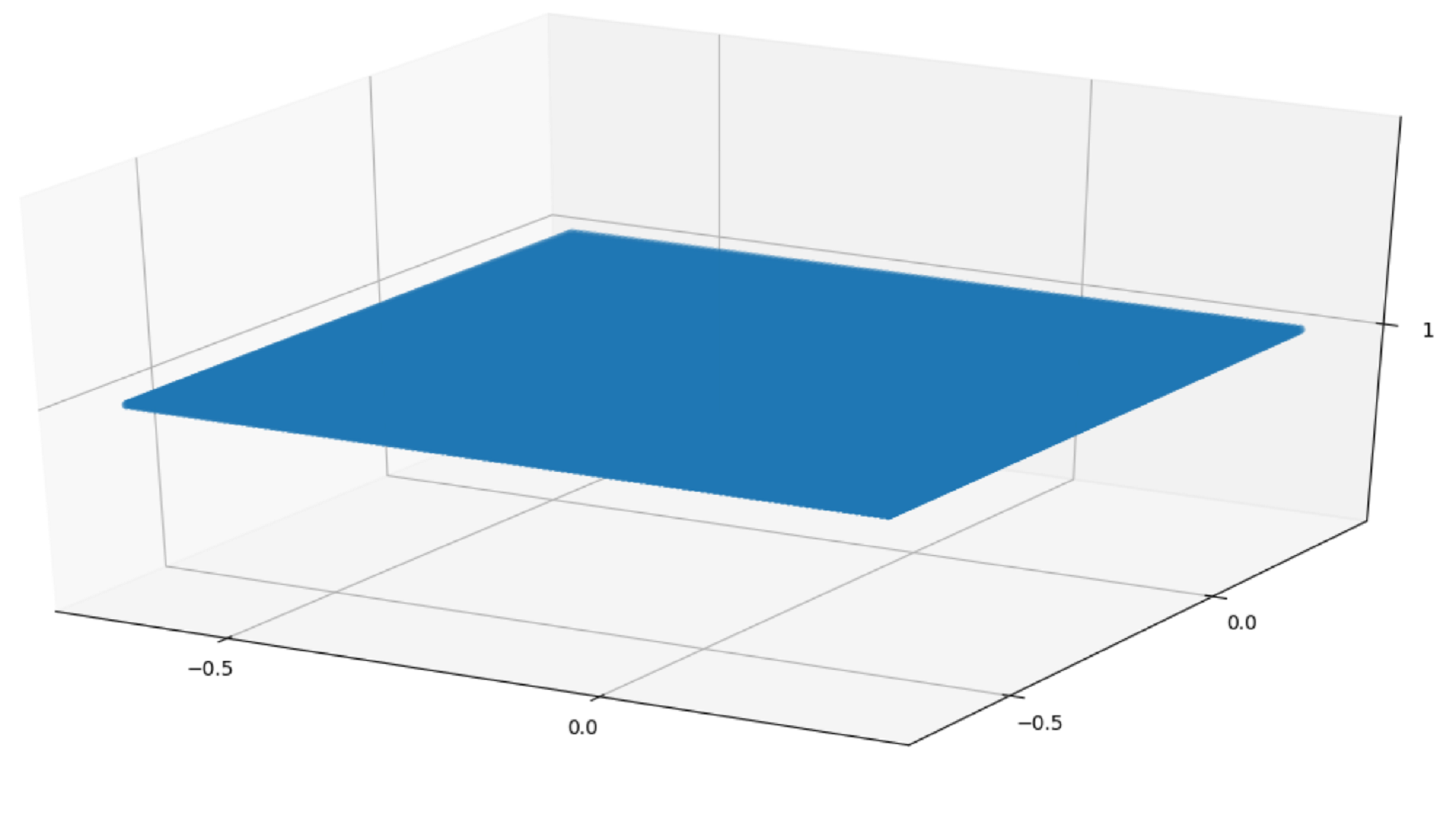}} \\ Square-1.
\end{minipage}
\hfill
\begin{minipage}[h]{0.48\linewidth}
\center{\includegraphics[width=1\linewidth]{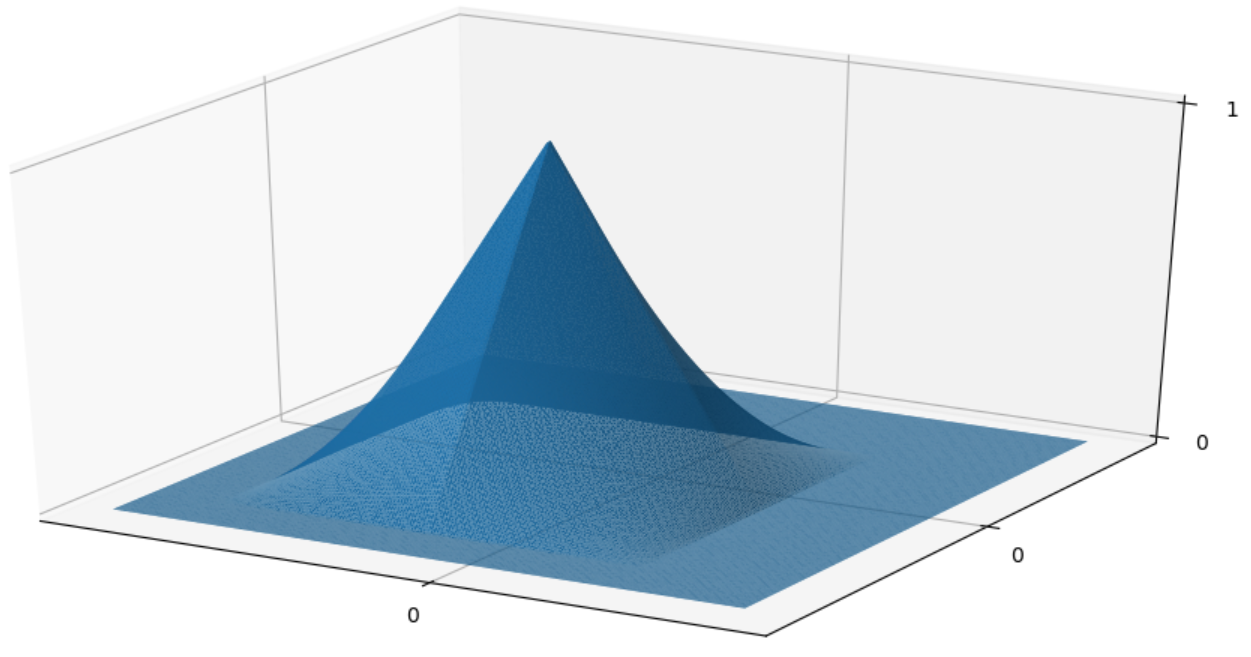}} \\ Square-2.
\end{minipage}
\vspace{\baselineskip}
\begin{minipage}[h]{0.48\linewidth}
\center{\includegraphics[width=1\linewidth]{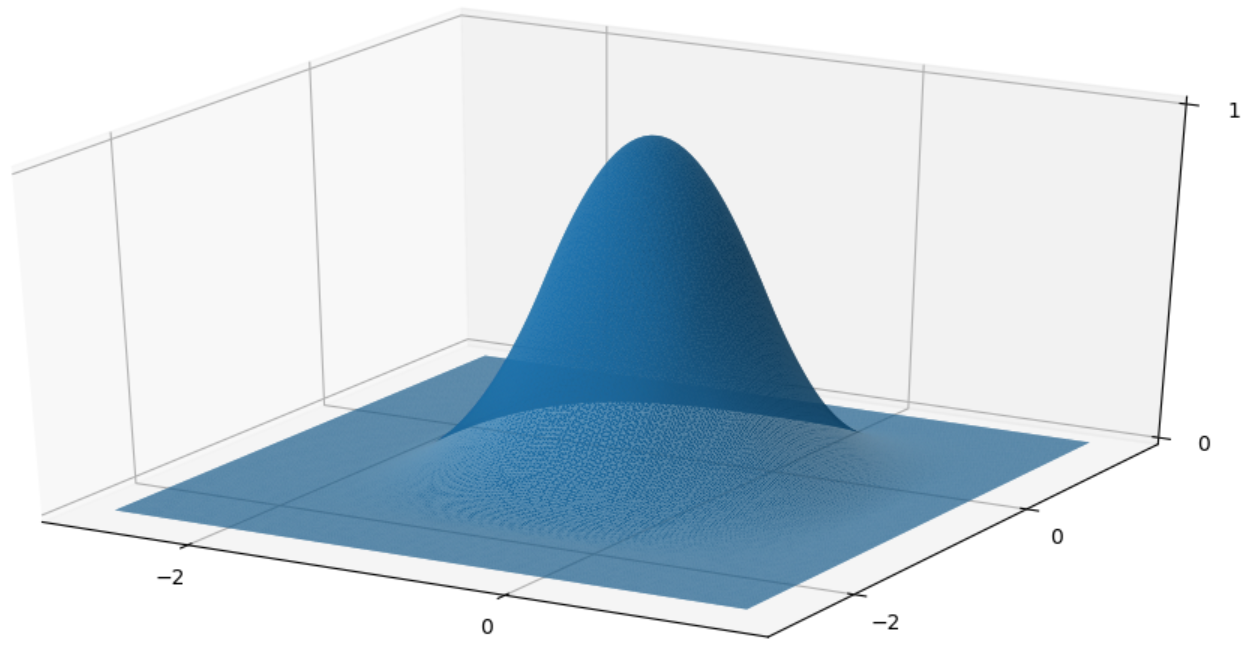}} \\ Square-3.
\end{minipage}
\hfill
\begin{minipage}[h]{0.48\linewidth}
\center{\includegraphics[width=1\linewidth]{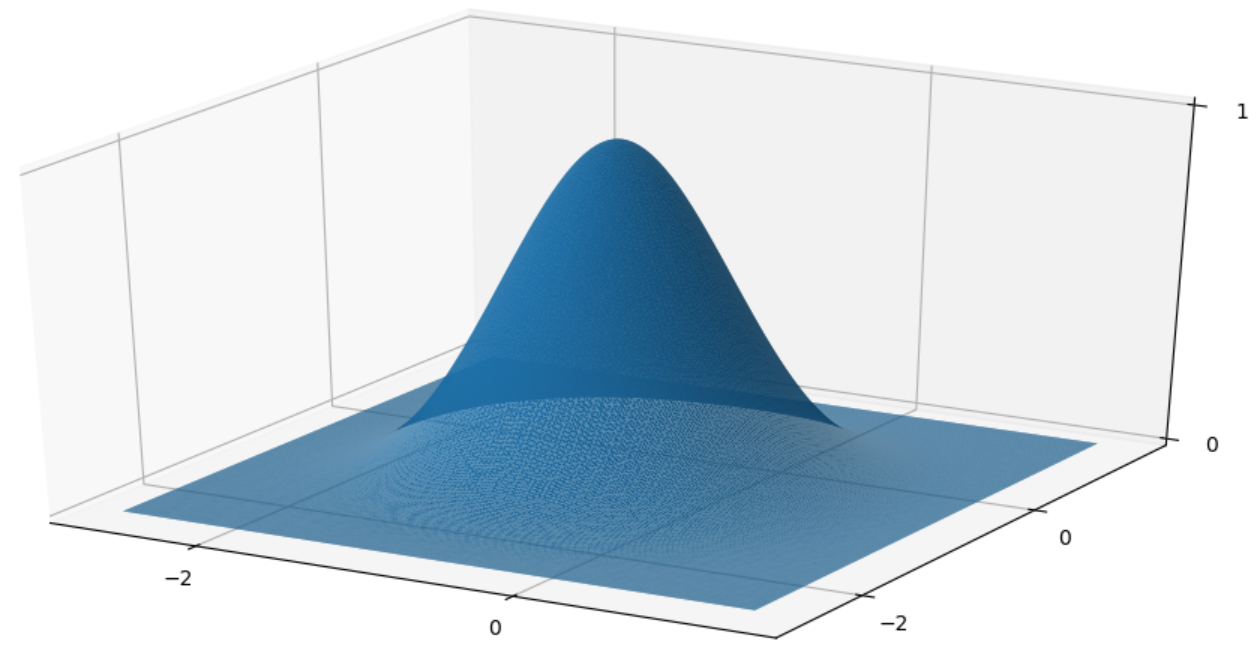}} \\ Square-4. 
\end{minipage}
\caption{Tile B-splines with Square matrix.}
\label{spl_Rect}
\end{center}
\end{figure}
\end{subsection}
\end{section}

\begin{section}{Orthogonalization of B-splines} \label{par_orthogonalization}
In the previous section we established that the B-spline $\varphi(x) = B_n(x)$ is a refinable function with compact support. Its integer shifts are not orthogonal to each other, therefore it does not generate an orthonormal wavelet system. Since these shifts form the Riesz basis of their linear span, there exists a standard way to orthogonalize them. Namely, we can construct another refinable function $\varphi_1(x)$ whose integer shifts are the orthonormal basis in the space of integer shifts of $\varphi(x)$ (in terms of wavelet theory, it should generate the same multiresolution analysis). The support of the function $\varphi_1$ is not finite anymore, but it decays fast at infinity (see the example of estimate in Section \ref{par_finite}).  The construction uses the following well-known fact (see, for example, \cite{KPS}). 

\begin{propos} The function $\eta(x) \in L_2$ has orthonormal integer shifts if and only if $\sum_{k \in \Z^d} |\widehat{\eta}(\xi + k)|^2 \equiv 1$.
\end{propos}

In particular, the function $\varphi_1$, given by formula in terms of Fourier transform 
\begin{equation}\label{varphi1}
\widehat \varphi_1(\xi) = \frac{\widehat \varphi(\xi)}{\sqrt{\sum \limits_{k \in \Z^d} |\widehat{\varphi}(\xi + k)|^2}},
\end{equation}
possesses this property. The transition from the function $\varphi(x)$ to the function $\varphi_1(x)$ by the formula \eqref{varphi1} is the standard Battle-Lemarie orthogonalization procedure. 

From formula \eqref{varphi1} it follows that the function $\varphi_1$ is expressed as a linear combination of integer shifts $\{\varphi(x - k)\}_{k \in \Z^d}$ of function $\varphi$, below we find the coefficients of decomposition. 

\begin{theorem}\label{th_ort}
Let $G$ be an arbitrary tile in $\R^d$, $\varphi(x) = B_n^G(x)$ be its corresponding tile B-spline, $\varphi_1(x)$ be its orthogonalization. Let $\Phi_k = (\varphi, \varphi(\cdot + k))$ for all $k \in \Z^d$ and $\Phi(\xi) = \sum_{k \in \Z^d} \Phi_k e^{-2 \pi i (k, \xi)}$.  
Then $\varphi_1(x)$ is a linear combination of integer shifts of $\varphi(x)$:  
$$\varphi_1(x) = \sum _{k \in \Z^d}{b_k \varphi(x - k)},$$
where $b_k$ are the Fourier coefficients of the function $\frac{1}{\sqrt{\Phi(\xi)}} = \sum b_k e^{-2 \pi i (k, \xi)}$. 
\end{theorem}
\begin{proof}
We decompose the orthogonalized spline $\varphi_1(x)$ by integer translates of $\varphi$: 
$$\varphi_1(x) = \sum _{k \in \Z^d}{b_k \varphi(x - k)}.$$  
Then for every $j \in \Z^d$, 
$$\varphi_1(x + j) = \sum \limits_{k \in \Z^d} b_k \varphi(x - (k - j)) \stackrel{(l = k - j)}{=} \sum \limits_{l \in \Z^d} b_{l + j} \varphi(x - l).$$
The shifts orthogonality property has the form (where $\delta_j^0$ denotes the Kronecker symbol) 
$$\delta_j^0 = (\varphi_1, \varphi_1(\cdot + j)) = \sum \limits_{k, l \in \Z^d} b_k \overline{b_{l + j}} (\varphi(\cdot - k), \varphi(\cdot - l)) = \sum \limits_{k, l \in \Z^d} b_k \overline{b_{l + j}} \Phi_{k - l}.$$
This can be rewritten as $$\sum \limits_{m \in \Z^d} \left(\sum \limits_{k \in \Z^d} b_k \overline{b_{k - m + j}}\right) \Phi_{m} = \delta_j^0.$$
Denote by $A_p = \sum_{k \in \Z^d} b_k \overline{b_{k - p}}$. Then we obtain that for every $j \in \Z^d$ 
$$\sum \limits_{m \in \Z^d} A_{j - m}\Phi_m = \delta_j^0$$
In other words, the convolution of sequences $A$ and $\Phi$ is the $\delta$-sequence. Then the product of their Fourier transforms $A(\xi) = \sum_{k \in \Z^d} A_k e^{-2 \pi i (k, \xi)}$ and $\Phi(\xi) = \sum_{k \in \Z^d} \Phi_k e^{-2 \pi i (k, \xi)}$ is identically equal to $1$. Thus, $A(\xi) = \frac{1}{\Phi(\xi)}$. 
We also consider the Fourier transform $B(\xi) = \sum_{k \in \Z^d} b_k e^{-2 \pi i (k, \xi)}$. Then we have   
$$A_p = \int B(\xi) \overline{B(\xi)} e^{2 \pi i (p, \xi)}d\xi = \int |B(\xi)|^2 e^{2 \pi i (p, \xi)}d\xi.$$
Hence, $|B(\xi)|^2 = \sum_{p \in \Z^d} A_p e^{-2 \pi i (p, \xi)} = A(\xi)$. Then, 
$$|B(\xi)| = \frac{1}{\sqrt{\Phi(\xi)}},$$
i.e., we express the coefficients $b_k$ of the expansion of the function $\varphi_1$ with respect to the shifts of the function $\varphi$ using the numbers $\Phi_k$. \end{proof}
By definition of the coefficients $\Phi_k = (\varphi, \varphi(\cdot + k))$, their calculation requires numerical integration. It turns out, however, that they could be found easily as the components of an eigenvector of a special matrix. This will be done in next section. 
\end{section}

\begin{section}{Formulas for coefficients $\Phi_k$} \label{phik}
In order to find the new refinable function $\varphi_1$ whose integer shifts will be orthonormal, we need to find the auxiliary numbers $\Phi_k = (\varphi, \varphi(\cdot + k))$. 
\begin{theorem} \label{th_four} 1) For every integer $k$, the number $\Phi_k$ is the value of function $\varphi(x) * \varphi(-x)$ at point $-k$. 

2) For the Fourier series constructed by the coefficients $\Phi_k$, it holds that  
$$\Phi(\xi) := \sum_{k \in \Z^d} \Phi_k e^{-2 \pi i (k, \xi)} = \sum \limits_{k \in \Z^d} |\widehat{\varphi}(\xi + k)|^2.$$

\end{theorem}
From the part 1) it follows that only a finite number of coefficients $\Phi_k$ are nonzero, hence, $\Phi(\xi)$ is a  trigonometric polynom. 
\begin{proof}[of Theorem \ref{th_four}]
Denote by $f$ the function $\varphi(x) * \varphi(-x)$. Then given that the function $\varphi$ is real-valued, we have  
$$f(y) = \int \varphi(x)\varphi(x - y)dx,$$
therefore we obtain 
$$f(-k) = \int \varphi(x)\varphi(x + k)dx = \Phi_k$$
at the point $y = -k$. 

For the proof of the second part we note that by the Plancherel theorem,  
$$\Phi_k = \int \varphi(x)\varphi(x + k)dx = 
\int \widehat{\varphi}(\xi) \overline{\widehat{\varphi}(\xi)} e^{-2 \pi i (k, \xi)} d\xi = \int |\widehat{\varphi}(\xi)|^2  e^{-2 \pi i (k, \xi)} d\xi,$$
since $\widehat{\varphi(\cdot + k)}(\xi) = e^{2 \pi i (\xi, k)}\varphi(\xi)$.

From this it follows that the Fourier coefficients of the function $\sum \limits_{k \in \Z^d} |\widehat{\varphi}(\xi + k)|^2$ 
 also coincide with the numbers $\Phi_k$. Thus, $\Phi(\xi) = \sum_{k \in \Z^d} \Phi_k e^{-2 \pi i (k, \xi)} = \sum \limits_{k \in \Z^d} |\widehat{\varphi}(\xi + k)|^2$, this completes the proof. 
\end{proof}

\begin{remark}
In Theorem \ref{th_ort} the values $\Phi_k = (\varphi, \varphi(\cdot + k))$ are defined as scalar products, whose calculation requires the  numerical integration. In Theorem \ref{th_four} we showed that they are equal to the values of the function $f := \varphi(x) * \varphi(-x)$ at integer points. This function satisfies a refinement equation with the mask $a(\xi) \bar{a}(\xi) = |a(\xi)|^2$. In particular, if $\varphi(x)$ is the tile B-spline $B_n$, then $f = \varphi(x) * \varphi(-x)$ is the symmetrized tile B-spline $\Bs_n$. If $\varphi(x)$ is the symmetrized tile B-spline $\Bs_n$, then $f = \varphi(x) * \varphi(-x)$ is the symmetrized B-spline $\Bs_{2n}$. 
Therefore, knowing the refinement equation of $\varphi(x)$, we can find the coefficients $\Phi_k$ as the components of the eigenvector of the special matrix (see Remark \ref{rem_constr_phi}).
\end{remark}

Orthogonalized tile B-splines for Bear-2, Bear-4, Dragon-2, Dragon-4, Square-2, Square-4 are depicted in Fig. \ref{ort_spl}. 

\begin{figure}[ht]
\begin{center}
\begin{minipage}[h]{0.48\linewidth}
\center{\includegraphics[width=1\linewidth]{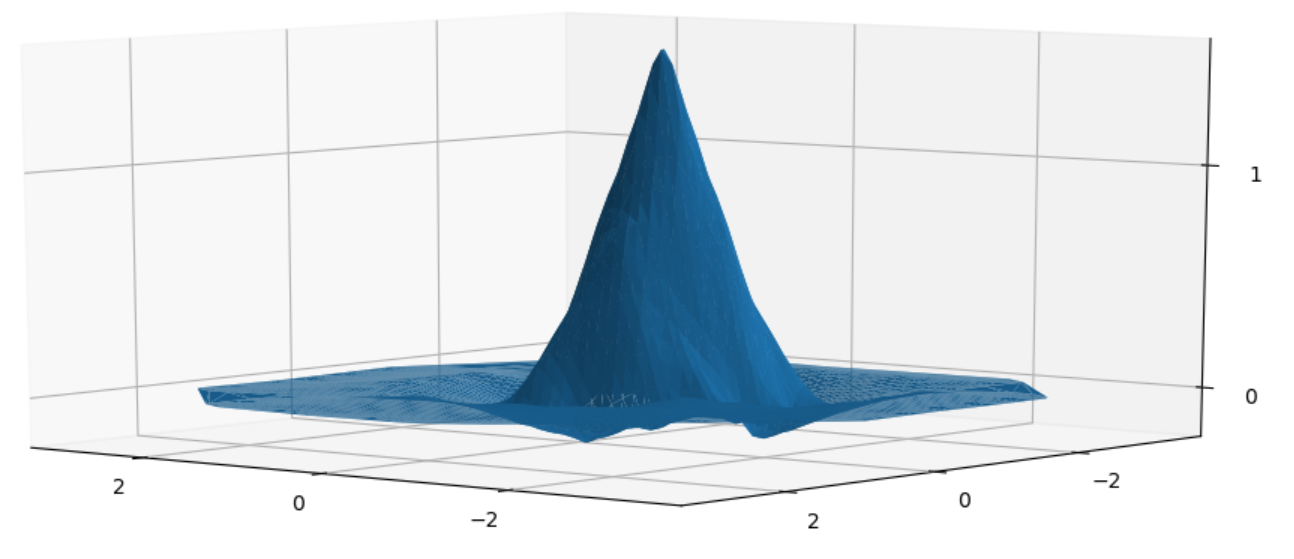}} \\ Orthogonalized Bear-2.
\end{minipage}
\hfill
\begin{minipage}[h]{0.48\linewidth}
\center{\includegraphics[width=1\linewidth]{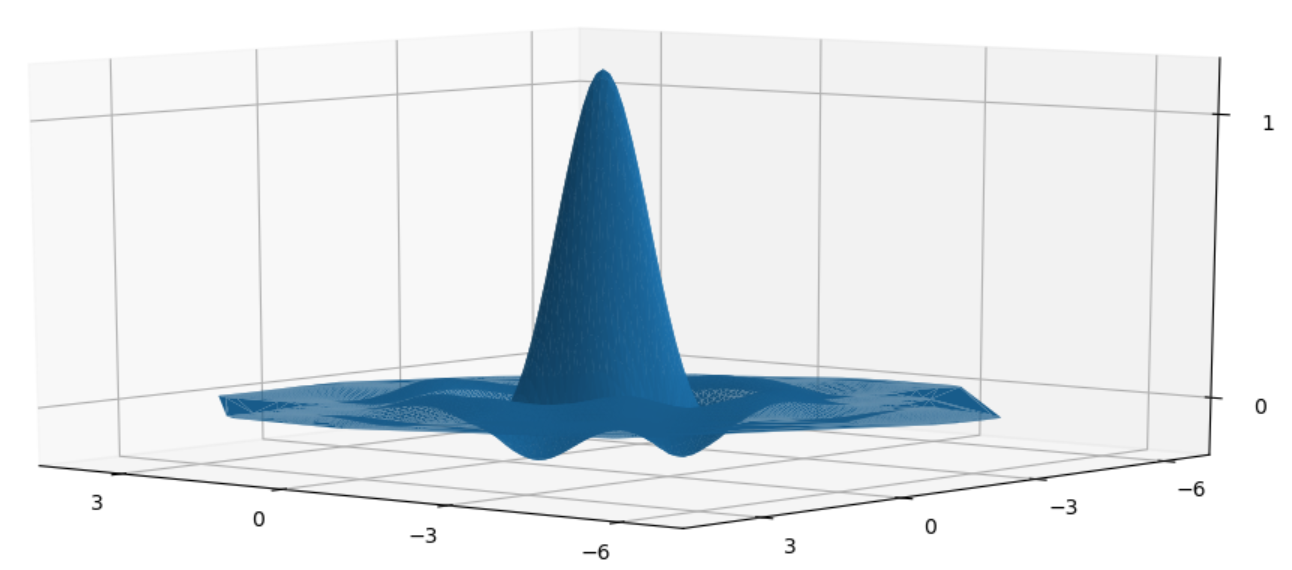}} \\ Orthogonalized Bear-4.
\end{minipage}
\vspace{\baselineskip}
\begin{minipage}[h]{0.48\linewidth}
\center{\includegraphics[width=1\linewidth]{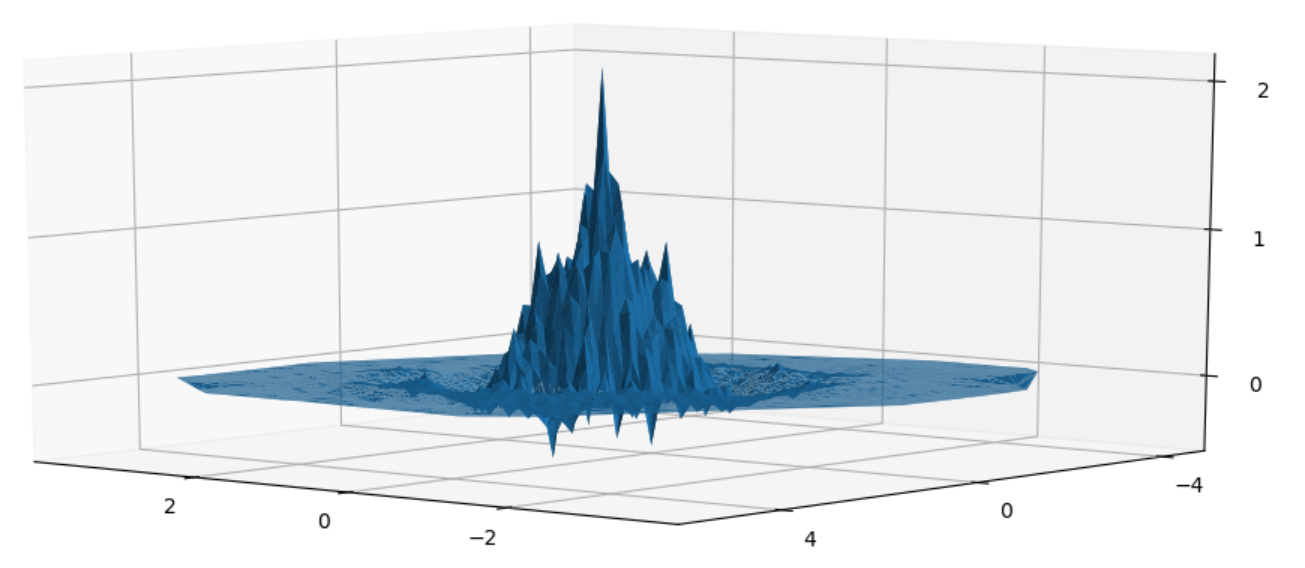}} \\ Orthogonalized Dragon-2.
\end{minipage}
\hfill
\begin{minipage}[h]{0.48\linewidth}
\center{\includegraphics[width=1\linewidth]{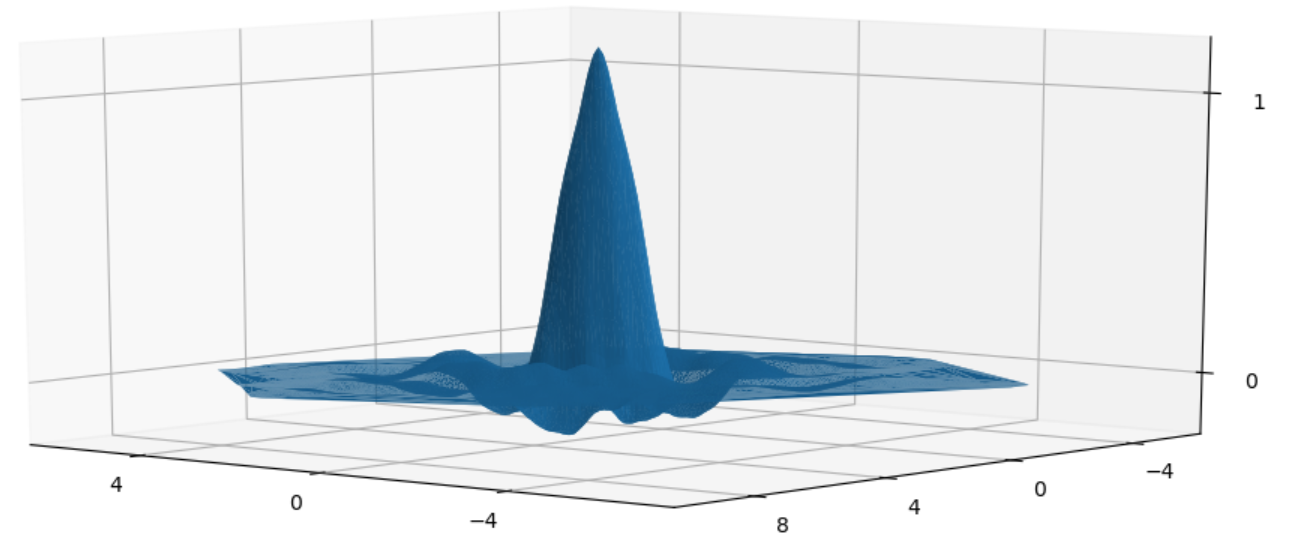}} \\ Orthogonalized Dragon-4. 
\end{minipage}
\vspace{\baselineskip}
\begin{minipage}[h]{0.48\linewidth}
\center{\includegraphics[width=1\linewidth]{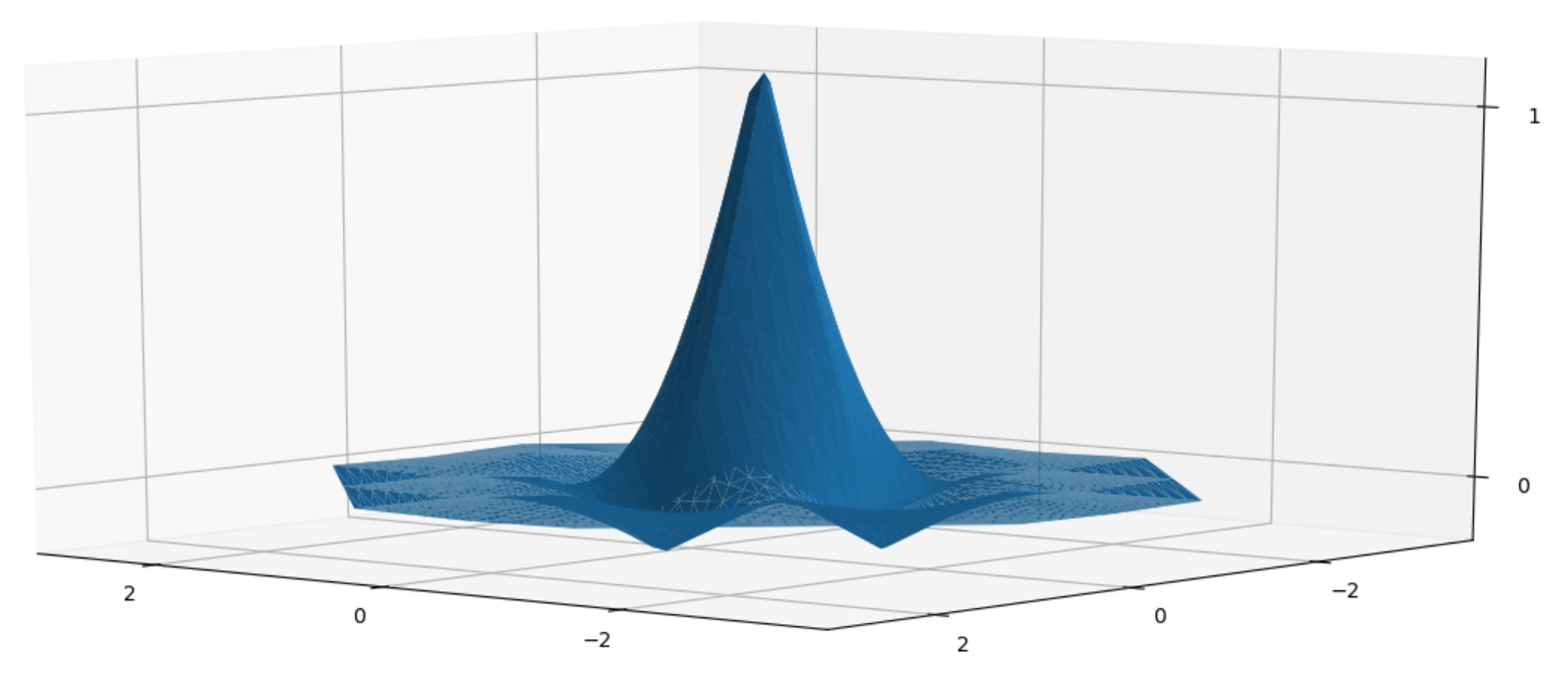}} \\ Orthogonalized Square-2.
\end{minipage}
\hfill
\begin{minipage}[h]{0.48\linewidth}
\center{\includegraphics[width=1\linewidth]{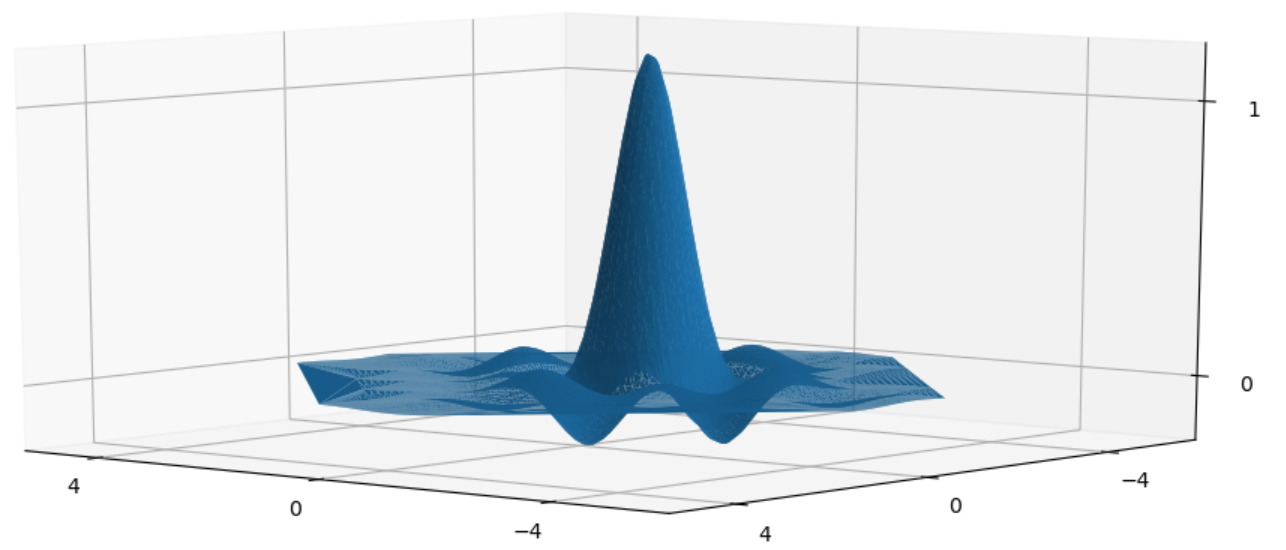}} \\ Orthogonalized Square-4. 
\end{minipage}
\caption{Orthogonalized tile B-splines.}
\label{ort_spl}
\end{center}
\end{figure}

\end{section}

\begin{section}{The construction of wavelet function}\label{wavelet}
Recall that $\Phi(\xi) = \sum_{k \in \Z^d} |\widehat{\varphi}(\xi + k)|^2$. It is easy to prove the following fact: 
\begin{propos} Let a tile B-spline $\varphi(x)$ satisfy a refinement equation with mask $a(\xi)$. Then its orthogonalization $\varphi_1(x)$ is the solution of a refinement equation with mask  
\begin{equation}\label{ort_mask}
a_1(\xi) = a(\xi)\frac{\sqrt{\Phi(\xi)}}{\sqrt{\Phi(M^T \xi)}}.
\end{equation}
\end{propos}
\begin{proof}
Using the representation of refinement equations after the Fourier transform \eqref{fourie}, it is sufficient to check the 1-periodicity of a function 
$$a_1(\xi) = \frac{\widehat \varphi_1(M^T \xi)}{\widehat \varphi_1(\xi)} = \frac{\widehat \varphi(M^T \xi) \sqrt{\Phi(\xi)}}{\widehat \varphi(\xi) \sqrt{\Phi(M^T \xi)}} = a(\xi)\frac{\sqrt{\Phi(\xi)}}{\sqrt{\Phi(M^T \xi)}}.$$
Since the functions $a(\xi), \Phi(\xi)$ are 1-periodic and the matrix $M$ is integer, it holds. 
\end{proof}

Thus, we can find the coefficients $c_k$ of the refinement equation for the function $\varphi_1$ by the Fourier expansion of the mask $a_1(\xi)$ defined by \eqref{ort_mask}: 
$$a_1(\xi) = \frac{1}{m} \sum \limits_{k \in \Z^d} c_{k} e^{-2 \pi i (k, \xi)}.$$

There are infinitely many nonzero coefficients $c_k$, therefore, the new refinable function $\varphi_1$ is not compactly supported. However, as we will see later, it has an exponentional decay as $\xi \to \infty$ that will allow us to effectively approximate it by compactly-supported functions (see Section \ref{par_finite}). 

Now we turn to the explicit construction of the wavelet function corresponding to the orthogonalized function $\varphi_1(x)$ in the two-digit case, i.e., when $m = 2$. In this case the Haar system has the simplest form since it is generated by a single wavelet function. The next theorem is the version of the general statement about the construction of orthonormal wavelets (see, for example, \cite{Woj}). Nevertheless, we give its full proof for two-digit tile B-splines. 

\begin{theorem} \label{th_psi}
Let $G$ be a two-digit tile (2-tile), i.e., $m = |\det{M}| = 2$, and $\varphi = B_n^G$ be its corresponding tile B-spline. Denote by $\varphi_1$ its orthogonalization, by $c_k$ its coefficients of refinement equation for $\varphi_1$. Then 

1) the corresponding wavelet function $\psi(x)$ is the linear combination of M-dilations of the function $\varphi_1$ with coefficients $\pm c_k$. 

2) For three types of affinely non-equivalent two-digit tiles, whose matrices are defined by formula \eqref{twotiles}, the following formulas for wavelet function hold: 

a) If $G$ is the tile ``Bear'', i.e., $M = M_B$, then  
\begin{equation} \label{psi}
\psi_B(x) = \sum \limits_{k \in K} c_k (-1)^{(k_2 - k_1)} \varphi_1 \left(M_B x + k - \begin{pmatrix} 0 \\ 1 \end{pmatrix}\right);
\end{equation}
b) If $G$ is the tile ``Square'', i.e., $M = M_S$, then
$$\psi_S(x) = \sum \limits_{k \in K} c_k (-1)^{k_1} \varphi_1 \left(M_S x + k - \begin{pmatrix} 1 \\ 0 \end{pmatrix}\right);$$
c) If $G$ is the tile ``Dragon'', i.e., $M = M_D$, then  
$$
\psi_D(x) = \sum \limits_{k \in K} c_k (-1)^{(k_2 - k_1)} \varphi_1 \left(M_D x + k - \begin{pmatrix} 0 \\ 1 \end{pmatrix}\right).$$
\end{theorem}
\begin{proof}
The wavelet function has the form  
$$\psi(x) = \sum \limits_{w \in W}p_w \varphi_1(Mx - w).$$
Consider the mask $p(\xi) = \frac{1}{m} \sum_{w \in W} p_w e^{-2\pi i (w, \xi)}$ for it, similarly to the mask \\ $a_1(\xi) = \frac{1}{m} \sum_{k \in K} c_k e^{-2\pi i (k, \xi)}$ for the refinement equation of $\varphi_1$. 
\begin{lemma}\label{lemma_ortho}
Let the vectors $0$, $u \in \Z^d$ be from two different cosets $\Z^d / M^T \Z^d$ of the matrix $M^T$. Let $v = M^{-T} u$. Then for the masks $a_1$ and $p$ of the scaling function and the wavelet function correspondigly, we have  

1) for every $s$ we have $|a_1(s)|^2 + |a_1(s + v)|^2 = 1$ (the orthonormality of $\varphi_1$); 

2) for every $s$ we have $p(s) \bar{a}_1(s) + p(s + v) \bar{a}_1(s + v) = 0$ (the orthogonality of $\varphi_1$ and $\psi$). 
\end{lemma}
\begin{proof}
From the refinement equation on $\varphi_1$ it follows that  
$$\widehat{\varphi}_1(\xi) = a_1(M_1^T \xi) \widehat{\varphi}_1(M_1^T \xi)$$ or  
$$\widehat{\varphi}_1(M^T\xi) = a_1(\xi) \widehat{\varphi}_1(\xi).$$
Since the integer shifts $\varphi_1$ are orthonormal, for every $s \in \R^d$, we have   
$$\sum \limits_{q \in \Z^d}|\widehat{\varphi}_1(s + q)|^2 = 1.$$ 
Choose the vectors $0$ and $u \in \Z^d$ from two cosets $\Z^d / M^T \Z^d$ defined by matrix $M^T$. Denote $v = M^{-T} u$.
\begin{multline*}1 = \sum \limits_{q \in \Z^d}|\widehat{\varphi}_1(M^Ts + q)|^2 = \sum \limits_{q \in \Z^d}|\widehat{\varphi}_1(M^Ts + M^Tq)|^2 + 
\sum \limits_{q \in \Z^d}|\widehat{\varphi}_1(M^Ts + M^Tq + M^Tv)|^2 = \\
= \sum \limits_{q \in \Z^d} |a_1(s + q)|^2 |\widehat{\varphi}_1(s + q)|^2 + \sum \limits_{q \in \Z^d} |a_1(s + q + v)|^2 |\widehat{\varphi}_1(s + q + v)|^2 = |a_1(s)|^2 + |a_1(s + v)|^2
\end{multline*}
Thus, for every $s$ we have $|a_1(s)|^2 + |a_1(s + v)|^2 = 1$ and 1) is proved. 

Similarly, from the orthogonality of $\varphi_1$ and $\psi$ we obtain 2).  
\end{proof}

Let us return to the proof of theorem. We will look for $p$ such that the orthogonality condition $p(s) \bar{a}_1(s) + p(s + v) \bar{a}_1(s + v) = 0$ from Lemma \ref{lemma_ortho} holds. Note that in the two-digit case the vector $2 \cdot v = 2 \cdot M_1^T u$ is integer, therefore $a_1(s + 2v) = a_1(s)$, $p(s + 2v) = p(s)$. 

Consider the Bear case with the matrix $M = M_B = \begin{pmatrix}1 & -2 \\ 1 & 0\end{pmatrix}$, then $M_1^T = \frac{1}{2}\begin{pmatrix}0 & -1 \\ 2 & 1\end{pmatrix}$. The vector $u$ can be chosen as $\begin{pmatrix}0 \\ 1\end{pmatrix}$. Then $v = \begin{pmatrix}-0.5 \\ 0.5\end{pmatrix}$. 

Therefore, we can propose the function $p(s) = e^{-2 \pi i s_2} \bar{a}_1(s + v)$ as a particular solution. 
Indeed, $p(s + v) = e^{-2 \pi i (s_2 + 0.5)} \bar{a}_1(s + 2 \cdot v)= -e^{-2 \pi i s_2} \bar{a}_1(s)$ and it is easy to check that the equality holds. 

Using the equality $\widehat{\psi}_B(\xi) = p(M_1^T \xi) \widehat{\varphi}_1(M_1^T\xi)$, we have   
$$\widehat{\psi}_B(\xi) = e^{-2 \pi i (M_1^T \xi)_2} \bar{a}_1( M_1^T \xi + v) \widehat{\varphi}_1(M_1^T\xi).$$
$$\widehat{\psi}_B(\xi) = e^{-2 \pi i ((0,1), M_1^T \xi)} \bar{a}_1(M_1^T \xi + v) \widehat{\varphi}_1(M_1^T\xi).$$
Since $$a_1(\xi) = \frac{1}{m} \sum \limits_{k \in K} c_k e^{-2\pi i (k, \xi)},$$
it follows that 
\begin{multline*}
\widehat{\psi}_B(\xi) = \frac{1}{m} \sum \limits_{k \in K} e^{-2 \pi i ((0,1), M_1^T \xi)} c_k e^{2\pi i (k, M_1^T \xi + v)} \widehat{\varphi}_1(M_1^T\xi) = \\
\frac{1}{m} \sum \limits_{k \in K} c_k e^{2\pi i (k, v)} e^{-2 \pi i (-k + (0,1), M_1^T \xi)} \widehat{\varphi}_1(M_1^T\xi)
\end{multline*}
Thus, for the Bear matrix, we conclude that  
$$\psi_B(x) = \sum \limits_{k \in K} c_k e^{2\pi i (k, v)} \varphi_1 \left(M_B x + k - \begin{pmatrix} 0 \\ 1 \end{pmatrix}\right),$$
where $v= \begin{pmatrix}-0.5 \\ 0.5\end{pmatrix}$. 
Note that since each of the vectors $k$ in the sum is integer, and the vector $v$ is half-integer, we see that the expronent $e^{2\pi i (k, v)}$ takes only  the values $\pm 1$ and hence,   
$$\psi_B(x) = \sum \limits_{k \in K} c_k (-1)^{(k_2 - k_1)} \varphi_1 \left(M_B x + k - \begin{pmatrix} 0 \\ 1 \end{pmatrix}\right),$$
this completes the proof. Similarly, we derive formulas for Dragon and Square. 
\end{proof}
The wavelet functions generated by Bears, Dragons, and Squares of order two and of order four are depicted in Fig. \ref{wave_spl}. 
\begin{figure}[ht]
\begin{center}
\begin{minipage}[h]{0.48\linewidth}
\center{\includegraphics[width=1\linewidth]{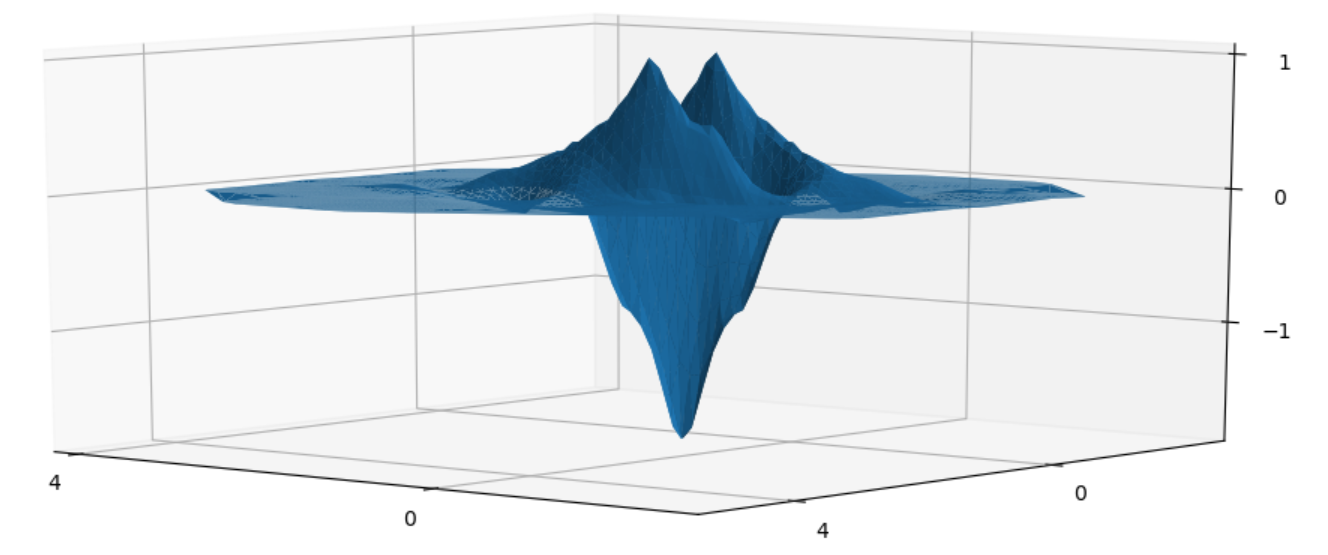}} \\ Bear-2 wavelet function.
\end{minipage}
\hfill
\begin{minipage}[h]{0.48\linewidth}
\center{\includegraphics[width=1\linewidth]{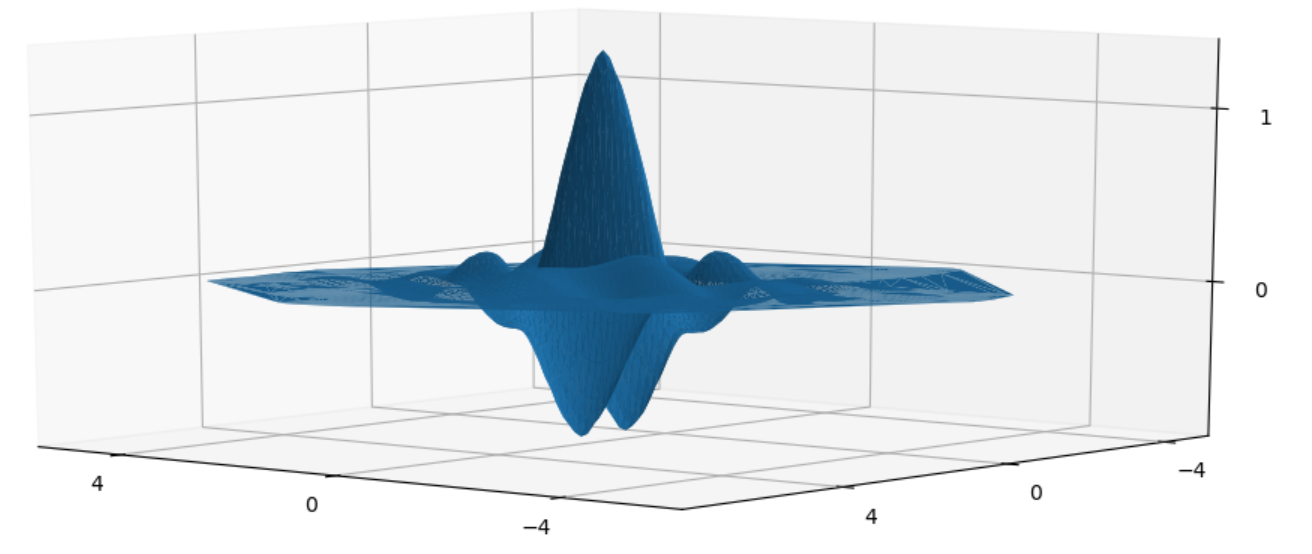}} \\ Bear-4 wavelet function.
\end{minipage}
\vspace{\baselineskip}
\begin{minipage}[h]{0.48\linewidth}
\center{\includegraphics[width=1\linewidth]{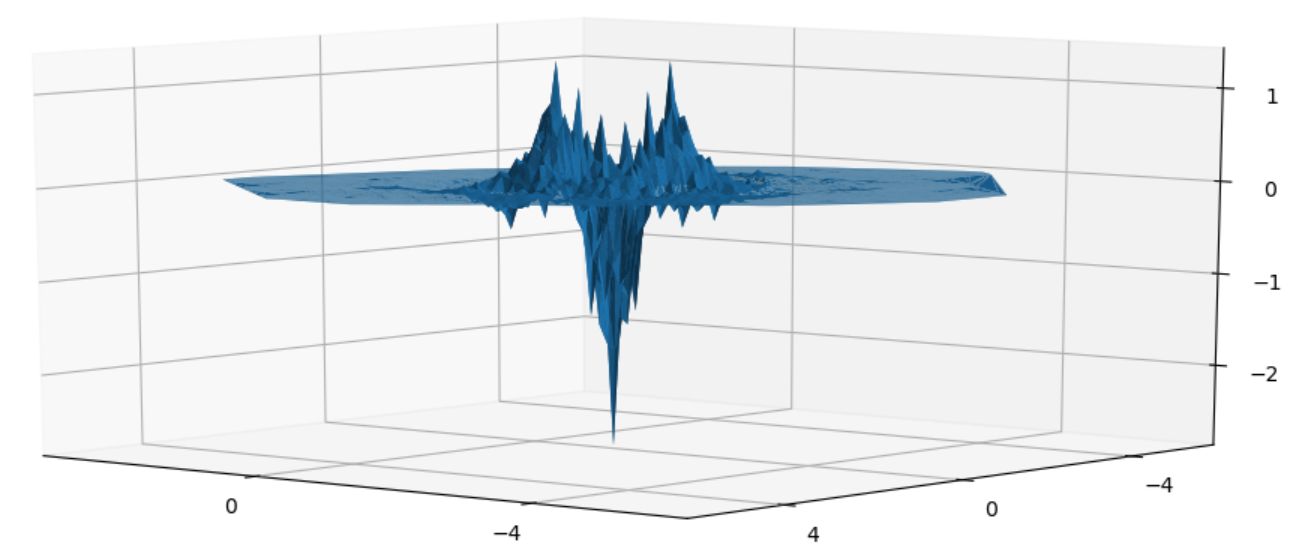}} \\ Dragon-2 wavelet function.
\end{minipage}
\hfill
\begin{minipage}[h]{0.48\linewidth}
\center{\includegraphics[width=1\linewidth]{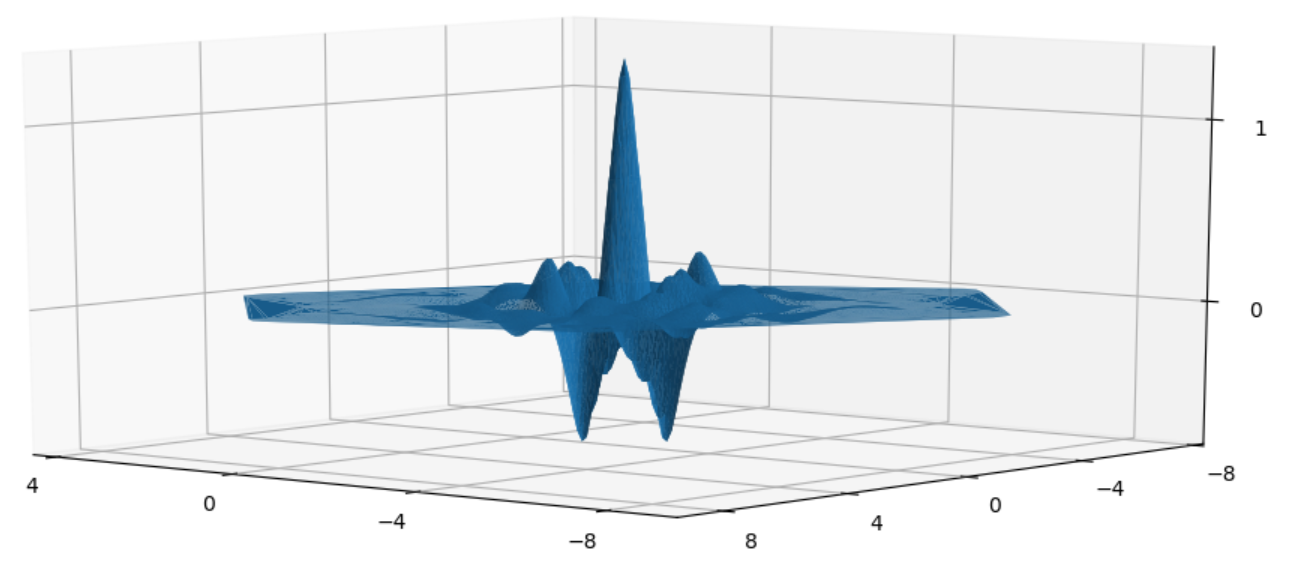}} \\ Dragon-4 wavelet function. 
\end{minipage}
\vspace{\baselineskip}
\begin{minipage}[h]{0.48\linewidth}
\center{\includegraphics[width=1\linewidth]{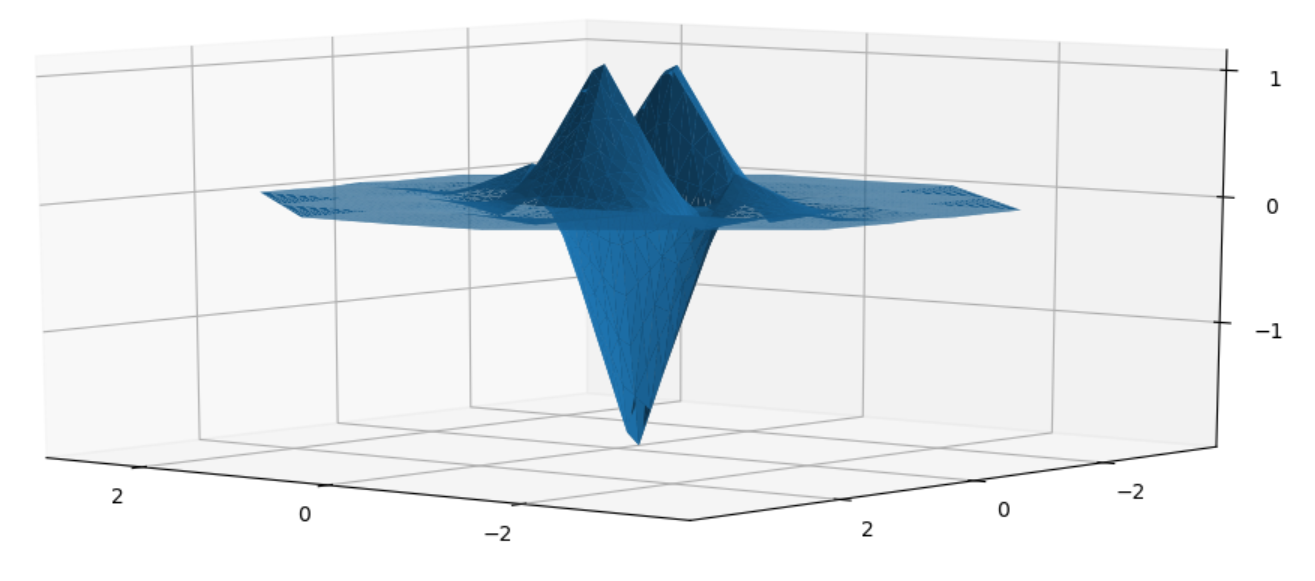}} \\ Square-2 wavelet function.
\end{minipage}
\hfill
\begin{minipage}[h]{0.48\linewidth}
\center{\includegraphics[width=1\linewidth]{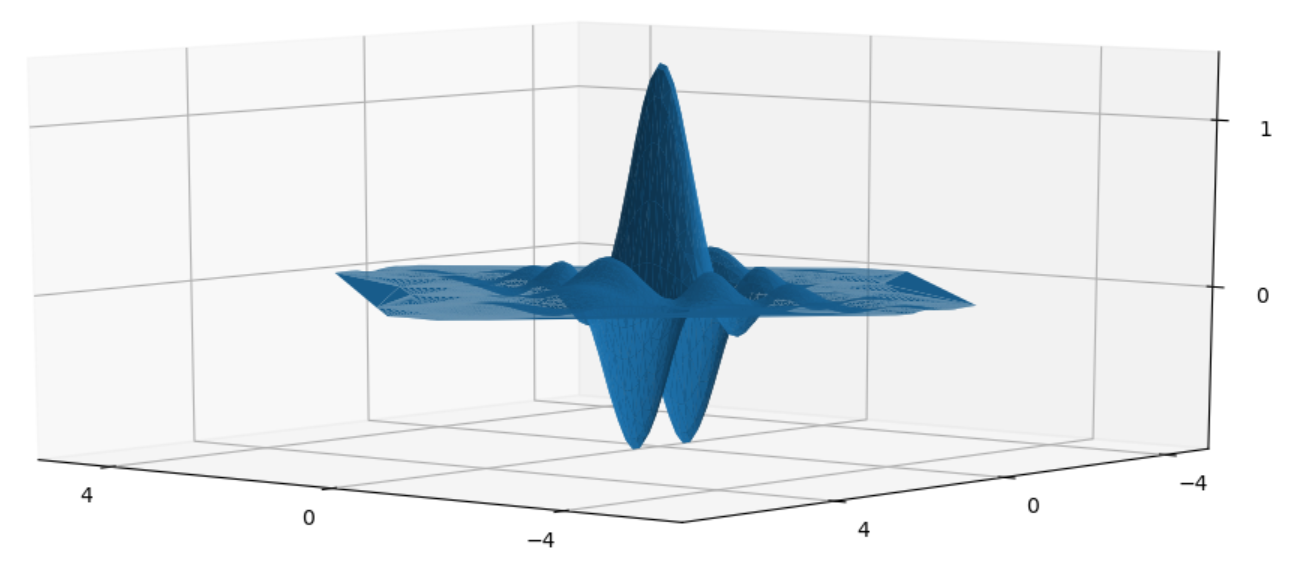}} \\ Square-4 wavelet function. 
\end{minipage}
\caption{Tile B-spline wavelet functions.}
\label{wave_spl}
\end{center}
\end{figure}
\end{section}

\begin{section}{The approximation of wavelet function by finite sums} \label{par_finite}
We have obtained formulas for orthogonal wavelet systems based on tile B-splines. 
Their use is complicated by the infinite summation. 
There are infinitely many summands in formula \eqref{psi}, since the mask $a_1(\xi)$, obtained after orthogonalization, is not a trigonometric polynomial. To estimate the accuracy of its approximation by trigonometric polynomials, one has to know the rate of decay of the coefficients $c_k$. 
We will prove that $|c_k| \le C_1 e^{-C_2 \|k\|}$, where $C_1, C_2$ are positive constants, and we will estimate the number $C_2$. Denote $z_1 = e^{-2 \pi i (e_1, \xi)}$, $z_2 = e^{-2 \pi i (e_2, \xi)}$, $z = (z_1, z_2)$, where $e_1 = \begin{pmatrix} 1 & 0 \end{pmatrix}$, $e_2 = \begin{pmatrix} 0 & 1 \end{pmatrix}$. Then $e^{-2 \pi i (k, \xi)} = z_1^{k_1}z_2^{k_2}$ and $a_1(z) = \frac{1}{m} \sum_{k \in \Z^2} c_k z_1^{k_1}z_2^{k_2}$. Considering the Laurent series of a function $a_1(z)$ in $\co^2$ and estimating the rate of decay of its coefficients, we obtain the estimate on $C_2$.   

We use the equality $a_1(\xi) = a(\xi)\frac{\sqrt{\Phi(\xi)}}{\sqrt{\Phi(M^T \xi)}}$.  
The original mask $a(z)$ (before orthogonalization) has a finite number of nonzero Fourier coefficients (since the initial equation is given by a finite number of coefficients $c_k$), therefore, the multiplier $a(\xi)$ does not influence the rate of decay of the coefficients $a_1(\xi)$. In addition, as we will see later, our estimate of the decay rate of the Laurent coefficients of expansion of a function into a Laurent series  depends only on its domain of holomorphy. Hence, we are interested only in zeros of the denominator $\sqrt{\Phi(M^T \xi)}$, i.e., zeros of the function $\Phi(M^T \xi)$. Thus, the final estimate on the rate of decay of the coefficients $c_k$ will be the estimate on Laurent coefficients of the function $\frac{1}{\Phi(M^T \xi)}$ after the change of variables to $z$. 
The location of zeros of $\Phi(M^T \xi)$ can be expressed in terms of zeros of the function $\Phi(\xi)$. Therefore, the rate of decay of $\frac{1}{\Phi(M^T \xi)}$ can be estimated of that of the function $\frac{1}{\Phi(\xi)}$. Since the Fourier coefficients of the denominator  are equal to $\Phi_k$, as it was shown in Section \ref{phik}, the Laurent coefficients of the function $\Phi(z)$ after change of variables are also equal to $\Phi_k$. Since we know the numbers $\Phi_k$, we can find the zeros of the denominator of the function $\frac{1}{\Phi(\xi)}$, as well as of the function $\frac{1}{\Phi(M^T \xi)}$. Further we will obtain the estimate on the decay rate of a function $f(z)=\frac{1}{g(z)}$ in general case, and then we will apply it to the function $\frac{1}{\Phi(M^T \xi)}$. 

\begin{subsection}{The rate of decay of the Laurent coefficients for bivariate holomorphic functions} \label{complexx}
Thus, we need to estimate the rate of decay of the coefficients of the function $f(z)=\frac{1}{g(z)}$ in the power expansion in $z \in \co^2$. To this end we invoke some facts from the  multivariate complex analysis. We study the structure of zeros of the function $g(z)$ to find the domain of holomorphy of function $f(z)$ and then estimate its coefficients. Note that for a holomorphic function of two complex variables its set of zeros is a union of continuous curves and, moreover, it does not have compact components. 

Let $B_R = B^-_R = \left\{z \in \co: \left|z\right| < R\right\}$ be a ball. \textit{The complement to the ball} of radius $r$ is the set  $B^+_r = \left\{z \in \co: \left|z\right| > r\right\}$. 
\textit{The annulus} is $A_{r, R} = \left\{z \in \co: r <  \left|z\right| < R\right\}$. \textit{The polydisk} of radius $R = (R_1, R_2)$ centered at $\overline 0 \in \co^2$ is the set $U(R) = \left\{z \in \co^{2}:\left|z_{v}\right|<R_{v}, v=1, 2 \right\}=B_{R_1} \times B_{R_2}$. 

We use the Reinhardt domains $\{(|z_1|, |z_2|) \mid (z_1,  z_2) \in U\}$ for depicting  
  subsets of $\co^2$. For example, a polydisk in $\co^2$ centered at the origin is presented in the Reinhardt domain as a rectangle with vertices $(0, 0)$, $(R_1, R_2)$. 

For given radii $r_1 < 1 < R_1$, $r_2 < 1 < R_2$ (we will choose them later) consider the product of annuli $A = A_{r_1, R_1} \times A_{r_2, R_2}$. On the diagram it is represented as a rectangle with vertices $(r_1, r_2)$, $(R_1, R_2)$. We introduce also the domains $P^{--}  = B^-_{R_1}\times B^-_{R_2}$, $P^{+-} = B^+_{r_1}\times B^-_{R_2}$,  $P^{-+}  = B^-_{R_1}\times B^+_{r_2}$, $P^{++}  = B^+_{r_1}\times B^+_{r_2}$. The domain $P^{--}$ is a polydisk, the other domains are the direct products of a ball and of the complement to a ball. The union of four domains is the whole complex plane $\co^2$, and their  intersection is the domain $A$. 

\begin{figure}[ht]
\centering
\includegraphics[width=0.5\linewidth]{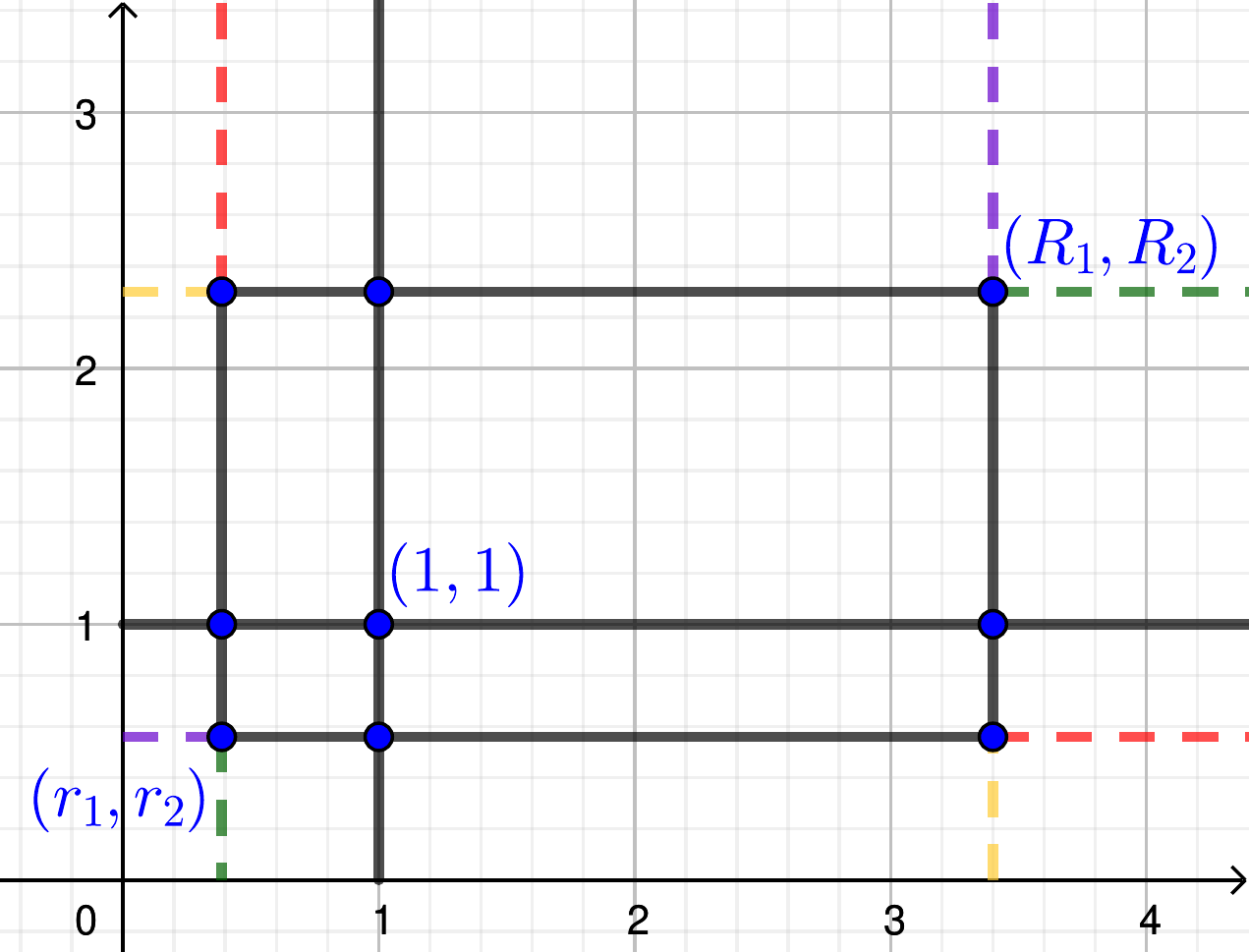}
\caption{Reinhardt domain of the subsets $P^{--}$, $P^{+-}$, $P^{-+}$, $P^{++}$}
\label{rein1}
\end{figure}

Suppose that the function $f=\frac{1}{g}$ is from $\mathscr{O}(A) \cap C(\overline{A})$, where $\mathscr{O}(A)$ denotes the set of functions holomorphic in the domain $A$, $C(\overline{A})$ is the set of functions continuous on the closure $A$. We apply the following classical theorem on the Laurent series expansion \cite{Shab}: 

\begin{mytheorem}
An arbitrary function $f(x_1, x_2) \in \mathscr{O}(A) \cap C(\overline{A})$ 
can be presented as a sum of four functions $f^{++}$, $f^{+-}$, $f^{-+}$, $f^{--}$ that are holomorphic in the domains $P^{++}$, $P^{+-}$, $P^{-+}$, $P^{--}$ respectively. 
\end{mytheorem} 

Thus, we obtain an expansion of the function $f=\frac{1}{g}$ into four summands. Let these sumands be $f^{++}$, $f^{+-}$, $f^{-+}$, $f^{--}$. One of them, $f^{--}$ is holomorphic in the polydisk $P^{--}  = B^-_{R_1}\times B^-_{R_2}$. 
The following theorem on power series in a polydisk holds: 

\begin{mytheorem}
Let $U$ be a polydisk in $\co^2$ of radius $R = (R_1, R_2)$ centered at $\overline 0 \in \co^2$. Every function $h \in \mathscr{O}(U) \cap C(\overline{U})$ 
can be written as a multiple power series  
$$h(z)=\sum_{k_1, k_2 \ge 0} c_{k_1, k_2}z_1^{k_1}z_2^{k_2}, \quad (z_1, z_2) \in U.$$ 
\end{mytheorem}

This theorem yields   
$f^{--}(x_1, x_2) = \sum \limits_{k_1, k_2 \ge 0} a_{k_1,k_2} x_1^{k_1}x_2^{k_2}$.  This series converges in the polydisk~$P^{--}$. 

The remaining functions could be reduced to holomorphic on the polydisk functions by a change of variables. For the  function $f^{-+}$, the change of variables $z_1 = x_1$, $z_2 = \frac{1}{x_2}$ gives the function $f^{-+}(z_1, z_2)$ holomorphic in a polydisk $P^{-+}(z)  = B^-_{R_1}(z_1)\times B^-_{\frac{1}{r_2}}(z_2)$, therefore, it is decomposed as $f^{-+} (z_1, z_2)= \sum \limits_{k_1, k_2 \ge 0} b_{k_1,k_2} z_1^{k_1}z_2^{k_2}$. Denote $a_{k_1, -k_2} = b_{k_1, k_2}$. Similarly, we obtain representations for the remaining two functions. 

The function ${f}(x_1, x_2)$ is then decomposed into the series  
$${f}(x_1, x_2) = \sum \limits_{k_1, k_2 \in \z} a_{k_1,k_2} x_1^{k_1}x_2^{k_2},$$
that converges in the domain $A$. 

We now estimate the coefficients $a_{k_1, k_2}$ for each of the four functions by Cauchy's formula and thus we obtain the resulting estimates for these coefficients. Let us recall the Caushy estimate for bivariate power series (see, for example, \cite{Shab}). 

\begin{mytheorem} \label{theorC}
If $h \in \mathscr{O}(U) \cap C(\overline{U})$, $|h| \le M$ in the domain $\left\{\left|z_{1}\right|=R_{1}\right\} \times \left\{\left|z_{2}\right|=R_{2}\right\}$,  
then for the coefficients of power series, we have the following inequality $$\left|c_{k_1, k_2}\right| \leqslant \frac{M}{R_1^{k_1}R_2^{k_2}}.$$ 
\end{mytheorem}

This theorem is applied separately for the decomposition coefficients  of $f^{++}$, $f^{+-}$, $f^{-+}$, $f^{--}$ in each of the four polydisks with variables $z_1$, $z_2$. For $k_1, k_2 \ge 0$ we obtain the estimates of the form $|a_{k_1, k_2}| \le \frac{C}{R_1^{k_1} R_2^{k_2}}$, $|a_{k_1, -k_2}| \le \frac{C}{R_1^{k_1} r_2^{-k_2}}$, $|a_{-k_1, k_2}| \le \frac{C}{r_1^{-k_1} r_2^{k_2}}$, $|a_{-k_1, -k_2}| \le \frac{C}{r_1^{-k_1} r_2^{-k_2}}$. 

The numbers $r_1, R_1, r_2, R_2$ have to be chosen so that $f=\frac{1}{g} \in \mathscr{O}(A) \cap C(\overline{A})$, 
where $A = A_{r_1, R_1} \times A_{r_2, R_2}$, i.e., the function $g$ does not have zeros in the closure of the domain $A$. For every possible choice of  $r_1, R_1, r_2, R_2$, we would obtain different estimates  on the coefficients. For simplicity, we choose $r_1 = \frac{1}{R_1} = r_2 = \frac{1}{R_2} = q$, in which case the rectangle on the Reinhardt domain in Fig. \ref{rein1} is a square with vertices on the line $y = x$ (see Fig. \ref{rein2}) and all the  estimates could be rewritten in the general form: $$|a_{k_1, k_2}| \le C q^{|k_1| + |k_2|}, k_1, k_2 \in \z. $$

\begin{figure}[ht]
\centering
\includegraphics[width=0.4\linewidth]{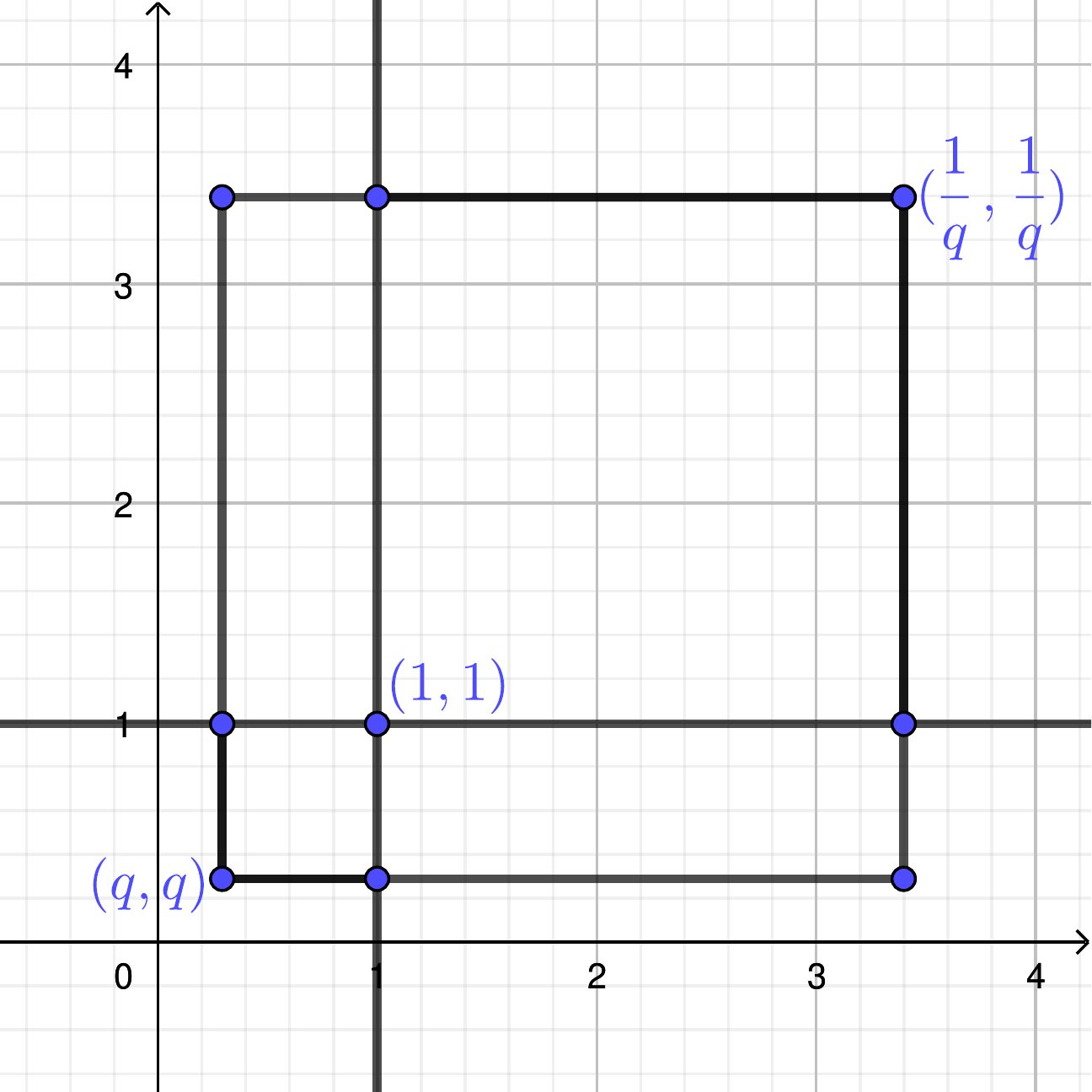}
\caption{Reinhardt domain of the subset $A(q, q) \times A(\frac{1}{q}, \frac{1}{q})$}
\label{rein2}
\end{figure}

The value $q$ is chosen so that the function $g(x_1, x_2)$ does not vanish on the closure of the domain $A(q, q) \times A(\frac{1}{q}, \frac{1}{q})$. It can always be done since the point $(1, 1)$ is not a zero of the function $g$. Hence, in view of continuity of $g$, there exists a neighborhood of the point $(1, 1)$ that does not contain zeros. 

Finally, we obtain the following theorem 

\begin{theorem}Let $f(z)=\frac{1}{g(z)}$, where $g$ is holomorphic in some domain $U \subset \co^2$ that contains the point $(1, 1)$, $g((1, 1)) \ne 0$.  
Let $q$ be such number that the domain $A = A(q, q) \times A(\frac{1}{q}, \frac{1}{q})$ is a subset of $U$ and has no zeros of $g(z)$ in its closure. Then we have in the domain $A$ 
$${f}(z_1, z_2) = \sum \limits_{k_1, k_2 \in \z} a_{k_1,k_2} z_1^{k_1}z_2^{k_2},$$
where coefficients are estimated as 
\begin{equation}\label{estpsi}
|a_{k_1, k_2}| \le C q^{|k_1| + |k_2|}, k_1, k_2 \in \z
\end{equation}
for some positive constant $C$. 
\end{theorem}
\end{subsection}

\begin{subsection}{The rate of decay of the coefficients of wavelet functions} \label{par_decreasing}
We now apply the general theorem from subsection \ref{complexx} on the coefficients of a holomorphic function to estimate the decay of the coefficients of wavelet functions build by the two-digit tiles, constructed in Theorem \ref{th_psi}. 

\begin{cor} \label{cor_dec} Let $G$ be a two-digit tile, $\varphi = B_n^G$ be the corresponding tile B-spline. We denote by $\varphi_1$ its orthogonalization, by $c_k$ the coefficients of refinement equation on $\varphi_1$. 
Let $\psi$ be the corresponding wavelet function. It is a linear combination of M-dilations of the function $\varphi_1$. For the  coefficients $a_{k_1, k_2} = \pm c_{k_1, k_2}$ of this linear combination, we have the following inequality  
\begin{equation}
|a_{k_1, k_2}| = |c_{k_1, k_2}| \le C q^{|k_1| + |k_2|}, k_1, k_2 \in \z, 
\label{estpsi}
\end{equation}
where $q$ is such that the function $\Phi(M^T \xi)$ after the change of variables $z_1 = e^{-2 \pi i (e_1, \xi)}$, $z_2 = e^{-2 \pi i (e_2, \xi)}$ has no zeros in the closure of the domain $A(q, q) \times A(\frac{1}{q}, \frac{1}{q})$, where $z = (z_1, z_2)$, $e_1 = (1, 0)$, $e_2 = (0, 1)$. 
\end{cor}

This implies the estimate on the decay of the wavelet function itself (the constant $C$ may change). 
\begin{cor} For the orthogonalized tile B-spline $\varphi_1$ and for the corresponding wavelet function $\psi$, we have 
$$
|\varphi_1(x_1, x_2)| \le C q^{|x_1| + |x_2|}, x_1, x_2 \in \R,$$ 
$$|\psi(x_1, x_2)| \le C q^{|x_1| + |x_2|}, x_1, x_2 \in \R, 
$$
where $q$ is chosen as in Corollary \ref{cor_dec}.
\end{cor}

We estimate the rate of decay of the Bear-2 and Bear-4 wavelet functions. That is, the set $G$ is the Bear tile with the matrix given by formula \eqref{twotiles}, $\varphi = B_2^G$ or $\varphi = B_4^G$. 

For Bear-4, we obtain $q = 0.85$. The approximate location of  zeros (restricted to a certain range) are given in Figs. \ref{nul1}, \ref{nul2}. The orange line is $y = x$, the green point is $(1, 1)$. For Bear-2, we can choose $q = 0.7$. 

We see that for Bear-4 the values $|c_{k_1, k_2}|$ for $|k_1| + |k_2| > 40$, are approximately $10^{-3}$ or less, and for Bear-2 it is approximately $10^{-6}$. Even for $|k_1| + |k_2| \le 40$ most of the coefficients are small. We now proceed with a more refined analysis of the coefficients. One needs to estimate how many of them have to be left to provide a good approximation of the function in $\ell_2$ or in $\ell_1$-norm.  

\begin{figure}[ht]
\centering
\includegraphics[width=0.6\linewidth]{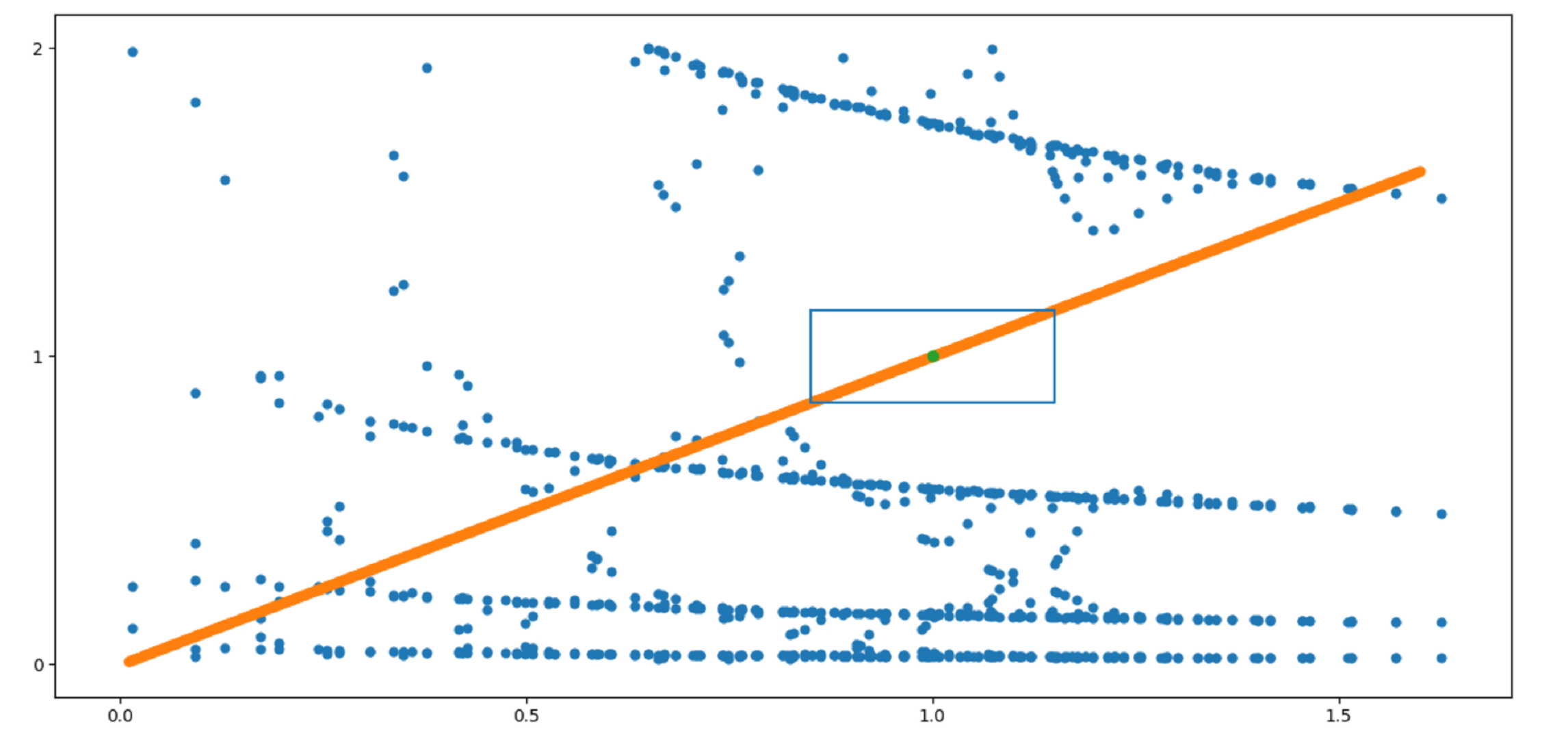}
\caption{Reinhardt domain of zeros of the function $\Phi(M_B^T \xi)$ for Bear-4 (after the change of variables to $z$)} 
\label{nul1}
\end{figure}

\begin{figure}[ht]
\centering
\includegraphics[width=0.6\linewidth]{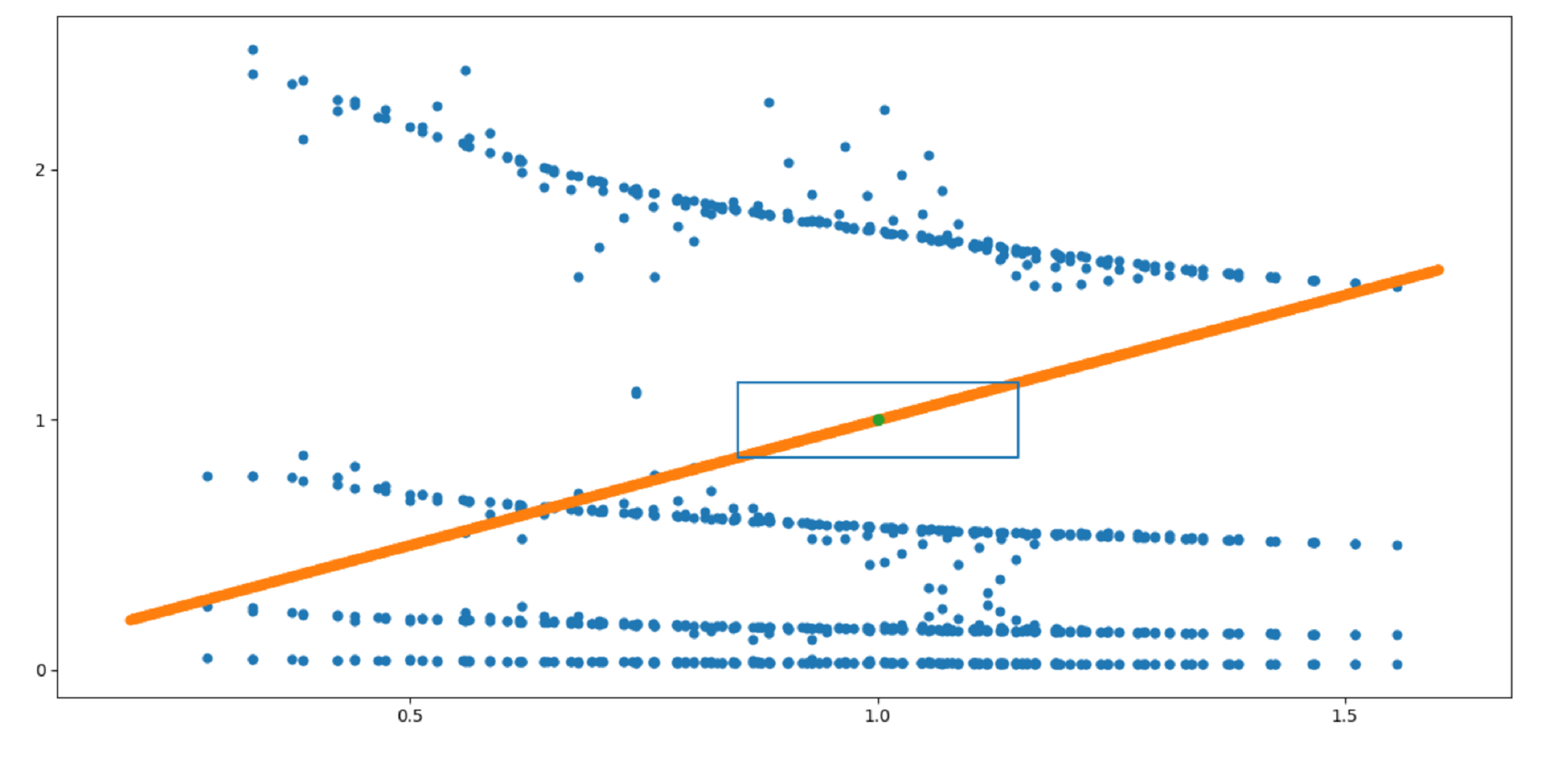}
\caption{Reinhardt domain of zeros of the function $\Phi(M_B^T \xi)$ for Bear-4 (after the change of variables to $z$): different scale. The depicted rectangle has the vertices $(q, q)$, $(\frac{1}{q}, \frac{1}{q})$.}
\label{nul2}
\end{figure}

Since we cannot store infinitely many coefficients $c_{k_1, k_2}$, we delete all those outside a large square, i.e., those satisfying $|k_1| + |k_2| > m$. How to choose $m$ so that the norm of the tail (the vector of the deleted coefficients) is small enough? 

\begin{propos} \label{tails}
For the $\ell_2$-norm of the coefficients of the tail $H_2 =\\ \sqrt{\sum_{|k_1| + |k_2| > m} c_{k_1, k_2}^2}$, 
we have   
\begin{equation}\label{est2}
H_2 \le \frac{2 C q^{m + 1}\sqrt{1 + m - mq^2}}{(1 - q^2)},
\end{equation}
and for the $\ell_1$-norm of the coefficients $H_1 = {\sum_{|k_1| + |k_2| > m} |c_{k_1, k_2}|}$, we have  
\begin{equation}\label{est1}
H_1 \le \frac{4 C q^{m + 1}(1 + m - mq)}{(1 - q)^2}, 
\end{equation}
where the parameters $q, C$ are from the inequality \eqref{estpsi}. 
\end{propos}
Thus, the norms of the tail both in $\ell_1$ and in $\ell_2$ are of order $O(q^m)$, where $q$ is smaller than one. The constants for $q^m$ for Bear-2 and Bear-4 will be estimated after the proof of Proposition \ref{tails}. 
\begin{proof}[of Proposition \ref{tails}]
Consider the case $k_1 > 0, k_2 \ge 0$, the final estimate will be two times larger. We use the inequality \eqref{estpsi}.  Then we have the following inequalities  
$$\sum \limits_{\substack{k_1 + k_2 > m \\ k_1 > 0, k_2 \ge 0}} c_{k_1, k_2}^2 \le \sum \limits_{s = m + 1}^\infty s C^2 q^{2s} \le \frac{C^2 q^{2m + 2}(1 + m - mq^2)}{(1 - q^2)^2}.$$ 
Thus,
$$H_2 \le \frac{2 C q^{m + 1}\sqrt{1 + m - mq^2}}{(1 - q^2)}$$ and the first statement is proved. 

Similarly, we estimate the $\ell_1$-norm of the coefficients of the tail.  
$$\sum \limits_{\substack{k_1 + k_2 > m \\ k_1 > 0, k_2 \ge 0}} |c_{k_1, k_2}| \le \sum \limits_{s = m + 1}^\infty s C q^{s} \le \frac{C q^{m + 1}(1 + m - mq)}{(1 - q)^2}$$ 
Then 
$$H_1 \le \frac{4 C q^{m + 1}(1 + m - mq)}{(1 - q)^2}.$$
\end{proof}

The constant $C$ could be estimated from Theorem \ref{theorC}. In what follows we suppose for simplicity that $C = 1$. The estimate of the value $q$ is illustrated in Section \ref{complexx}.

\begin{ex} (Bear-2). 
For Bear-2 we have $q= 0.7$. 
The values of the right-hand side of the estimate \eqref{est2} for  $q = 0.7$ and some $m$ are given in Table \ref{table0}. Consider $m = 22$, for that we have $H_2 \le 0.005$. 

\begin{table}[H]
\begin{center}
\begin{tabular}{c|c|c|c|c|c|c|c}
$m$ & 1 & 10 & 15 & 21 & 22 & 30 & 60 \\ \hline
$H_2 \le $ & 2.36 & 0.19 & 0.038 & 0.00525 & 0.00375 & 0.00025 & 0.00000025 \\ 
\end{tabular}
\end{center}
\caption{Estimates on $\ell_2$-norm of the tail of coefficients $H_2$ for $q = 0.7$}
\label{table0}
\end{table}
 
We choose as many as possible coefficients $c_{k_1, k_2}$ with $|k_1| + |k_2| \le m$ so that the square root of the sum of their squares is at most $0.005$. We delete these coefficients. Now the $\ell_2$-norm of the other coefficients is at most $0.01$, and we obtain a good precision. 

Numerical results show that only $65$ coefficients remain. Their location and sizes are shown in Fig. \ref{coef2}. The size of points  depends logarithmically on the corresponding coefficients. The values of the coefficients are also given in Table \ref{tablecoef2}. Only a half of them remains since in our case $c_{i, j} = c_{-i, -j}$.  

\begin{figure}[ht]
\centering
\includegraphics[width=1\linewidth]{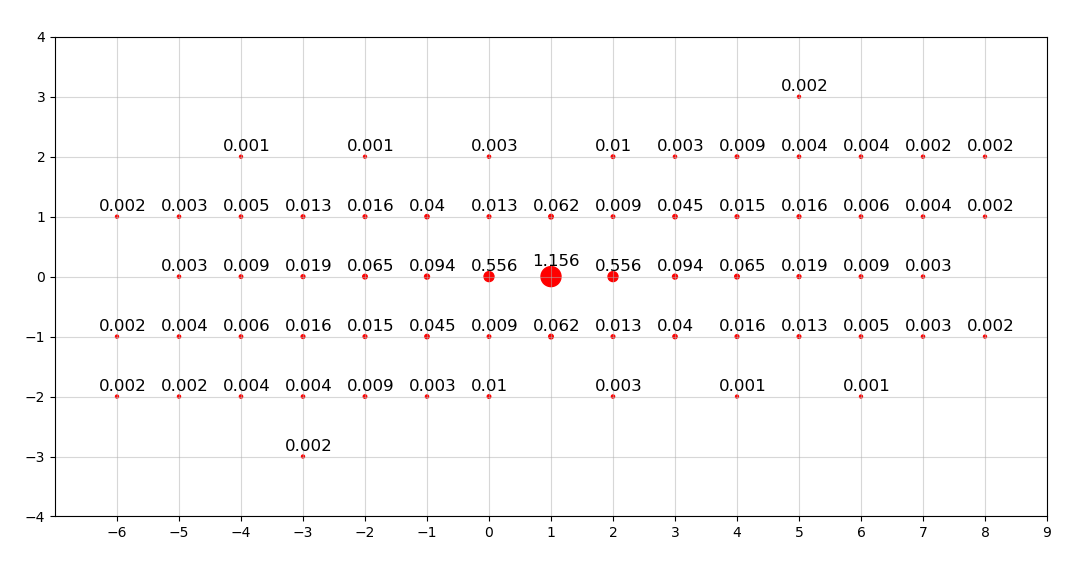}
\caption{Selected 65 coefficients that guarantee the $\ell_2$-norm of error $<0.01$}
\label{coef2}
\end{figure}

Similarly, the values of the right-hand side of the estimate \eqref{est1} for Bear-2 for $C = 1$, $q = 0.7$ are given in Table \ref{table1}. First, consider $m = 32$, for which we have $H_1 \le 0.005$. 

\begin{table}[H]
\begin{center}
\begin{tabular}{c|c|c|c|c|c|c|c}
$m$ & 1 & 10 & 20 & 30 & 32 & 45 & 60 \\ \hline
$H_1 \le  $ & 28.3 & 3.52 & 0.17 & 0.007 & 0.0036 & 0.000048 & 0.0000003 \\ 
\end{tabular}
\end{center}
\caption{Estimates on $\ell_1$-norm of tails of the coefficients $H_1$ for Bear-2}
\label{table1}
\end{table}

We again select the maximal possible number of coefficients $c_{k_1, k_2}$ from the square $|k_1| + |k_2| \le m$ with the sum of moduli at most $0.005$ and delete them. Numerical computations show that $149$ coefficients remain, they are given in Fig. \ref{coef1}. Then the $\ell_1$-norm of the deleted coefficients is at most $0.01$. The values of the coefficients are given in Table \ref{tablecoef2_1} in Appendix.

\begin{figure}[ht]
\centering
\includegraphics[width=1\linewidth]{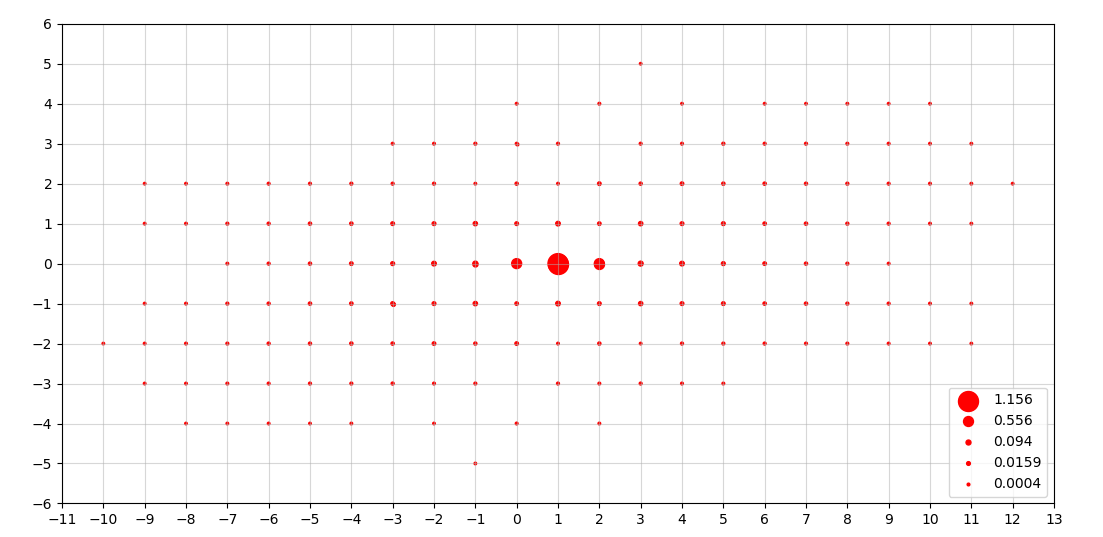}
\caption{Chosen 149 coefficients  that provide error in $\ell_1$-norm $<0.01$}
\label{coef1}
\end{figure}

\begin{table}[H]
\small
\begin{center}
\begin{tabular}{c|c|c|c|c|c|c|c|c}
$i$ & 1 & 0 & 3 & 4 & 1 & 3 & 3 & -3\\ \hline
$j$ & 0 & 0 & 0 & 0 & 1 & 1 & -1 & 0\\ \hline
\!\!\!\!\! $c_{i,j}$ \!\!\!\! & 1.15586 & 0.5563 & -0.09441 & -0.06459 & 0.06225 & -0.04478 & -0.0398 & 0.0191 \!\!\!\! \\ \hline \hline

$i$  & 5 & 4 & 4 & 2 & 5 & 0 & -4 & 4\\ \hline
$j$  & 1 & -1 & 1 & -1 & -1 & -2 & 0 & 2 \\ \hline
\!\!\!\!\! $c_{i,j}$ \!\!\!\!  & 0.01591 & 0.01557 & 0.01535 & -0.01304 & 0.01256 & -0.00979 & 0.00935 & 0.00911 \!\!\!\! \\ \hline \hline

$i$  & 2 & -4 & -4 & 6 & -5 & 5 & -5 & 2\\ \hline
$j$  & 1 & -1 & 1 & 2 & -1 & 2 & 1 & -2 \\ \hline
\!\!\!\!\! $c_{i,j}$ \!\!\!\!  & -0.00862 & -0.00644 & -0.00543 & -0.00430 & -0.00418 & -0.0041 & -0.00350 & 0.00340
 \!\!\!\! \\  \hline \hline
\end{tabular}
\begin{tabular}{c|c|c|c|c|c|c|c|c|c}
$i$  & 3 & -5 & -5 & -6 & 5 & -6 & 8 & 6 & 4\\ \hline
$j$  & 2 & 0 & -2 & -1 & 3 & -2 & -1 & -2 & -2\\ \hline
\!\!\!\!\! $c_{i,j}$ \!\!\!\!  & 0.0031 & -0.0029 & 0.0021 & 0.0017 & -0.0016 & 0.0015 & 0.0015 & -0.0012 & 0.0012\!\!\!\! \\ 

\end{tabular}
\end{center}
\caption{Coefficients of Bear-2 for approximation in $\ell_2$}
\label{tablecoef2}
\end{table}

\end{ex}

\begin{ex} (Bear-4). 
For Bear-4 we have $q = 0.85$. The Table \ref{table2} has the values of the right-hand side of the estimate \eqref{est2} for $q = 0.85$ and some $m$. Consider $m = 53$ with that for Bear-4 we have $H_2 \le 0.005$. We choose again as much of the coefficients $c_{k_1, k_2}$ ($|k_1| + |k_2| \le m$) so that the square root of the sum of their squares is at most $0.005$, and we delete them. Then the $\ell_2$-norm of the remained coefficients is at most $0.01$. Their values are given in Table \ref{tablecoef4_2} in Appendix.  

 \begin{table}[H]
\begin{center}
\begin{tabular}{c|c|c|c|c|c|c|c}
$m$ & 1 & 10 & 20 & 30 & 40 & 53 & 60 \\ \hline
$H_2 \le $ & 5.89 & 2.34 & 0.61 & 0.14 & 0.03 & 0.0044 & 0.0015 \\ 
\end{tabular}
\end{center}
\caption{Estimates on $\ell_2$-norm of tails of the coefficients $H_2$ for Bear-4}
\label{table2}
\end{table}
\end{ex}
\end{subsection}
\end{section}

\begin{section}{The regularity of the tile B-splines} \label{par_smooth}
The regularity is one of the most important parameters of refinable functions and of the corresponding wavelet systems. For wavelets, regularity implies good approximation properties and fast decay of  the coefficients of wavelet decompositions \cite{NPS, Woj}.  
In some applications, the regularity is crucial, for example, in the wavelet-Galerkin method. For the subdivision schemes, the regularity of the limit function defines both the quality of the limit surface and the rate of convergence of the algorithm \cite{CDM}. 

For the classical piecewise-polynomial splines regularity is defined by their order, which is not the case for tile B-splines. Applying the method developed in recent article \cite{CP} we compute the exact H{\" o}lder exponents for the tile B-splines of small orders.

\begin{definition} \textit{The general H{\" o}lder regularity of a function $\varphi$ in the space $C$} is the number $$\alpha_{\varphi} = k + \sup \left\{\alpha \geq 0 \, : \, \|\varphi^{(k)}(\cdot + h) - \varphi^{(k)}\|_C \leq C \|h\|^{\alpha}, \forall h \in \R^d\right\},$$ where $k$ is the maximal integer such that $\varphi \in C^k(\R^d)$. 
\end{definition}
If $\varphi \in C^{\infty}$, then we define $\alpha_{\varphi} = +\infty$. 

Similarly, the H{\" o}lder regularity in $L_2$ is defined by replacing of $C^k(\R^d)$ with $W_2^k(\R^d)$. 

It is known that  the value of H{\" o}lder regularity of a refinable function is defined by the so-called joint spectral radius (in case of $L_2$-regularity it is $L_2$-radius). It is defined as follows: 

\begin{definition} For linear operators $A_0, A_1$, their \textit{joint spectral radius} is the number  
$$\rho_{C}(A_0, A_1) = \lim \limits_{s \to \infty} \max \limits_{\sigma} \|A_{\sigma(1)}\ldots A_{\sigma(s)}\|^{1/s}, \, \sigma\colon \{1, \ldots, s\} \to \{0, 1\}.$$ 
\end{definition}

\begin{definition} For linear operators $A_0, A_1$ their \textit{$L_2$-radius} is the number $$\rho_{2}(A_0, A_1) = \lim \limits_{m \to \infty} \left(\frac{1}{2^s} \sum \limits_{\sigma} \|A_{\sigma(1)}\ldots A_{\sigma(s)}\|^2\right)^{1/2s}.$$
\end{definition}

Suppose we have a refinement equation with the finite number of summands. 
Consider the set 
$$K = \{x \in \R^d \, \, \colon \, \, x = \sum \limits_{j = 1}^{\infty} M^{-j} \gamma_j\}.$$
Then we choose an arbitrary set of digits $D(M)$ for the dilation matrix $M$ that generates a tile $G_0$. We call it a \textit{basis tile}.  
In the case of tile B-splines one can take the corresponding generating tile as a basis tile. 

\begin{definition}The set $\Omega \subset \Z^d$ is the minimal set of integer vectors such that $K \subset \Omega + G_0 = \bigcup \limits_{k \in \Omega} {(k + G_0)}$. 
\end{definition}
This set can be found using the algorithm from \cite{CM}. In the univariate case for $M = 2$, $D(M) = \{0, 1\}$, the basis tile $G_0 = [0, 1]$ is the unit segment. If a refinement equation is given by the coefficients $c_0, \ldots, c_N$, then $K = [0, N]$, $\Omega = \{0, 1, 2, \ldots, N - 1\}$. 

If the basis tile is fixed, then we can define \textit{transition matrices} $(T_d)_{a,b} = c_{Ma - b + \Delta}$ for all $a, b \in \Omega, \Delta \in D(M)$. 

Using these matrices we can find H{\" o}lder regularity of a refinable function $\varphi$ (see \cite{CP}). It is expressed in terms of the joint spectral characteristics of matrices restricted to a certain common invariant subspace, and in most cases this space is $W=\left\{x \in \mathbb{R}^{N} \mid \sum_{k} x_{k}=0\right\}$. Namely, if the H{\" o}lder regularity is at most one and the integer shifts of the function $\varphi$ are linearly independant (see Section \ref{par_subdiv} for details), then we have   
$$\alpha_{\varphi} = -\log_{\rho(M)}(\rho_{C}(T_0|_W, T_1|_W)),$$
$$\alpha_{\varphi, 2} = -\log_{2}(\rho_2(T_0|_W, T_1|_W)),$$
where $\rho(M)$ denotes the spectral radius of a matrix. 
In the general case, if there are no such constraints on the H{\" o}lder regularity, it is computed by similar formulas:  
\begin{equation} \label{holder_c}
\alpha_{\varphi} = -\log_{\rho(M)}(\rho_{C}(T_0|_{W_k}, T_1|_{W_k})),
\end{equation}
\begin{equation} \label{holder_l2}
\alpha_{\varphi, 2} = -\log_{2}(\rho_{2}(T_0|_{W_k}, T_1|_{W_k})),
\end{equation}
where $W_k$ is the space of vectors from $\R^d$ orthogonal to the space of polynomials of $d$ variables with degree at most $k$; the number $k$ in formulas \eqref{holder_c}, \eqref{holder_l2} is the maximal number such that the space $W_k$ is invariant with respect to matrices $T_0$, $T_1$. 

\begin{remark}
There is an explicit formula for the $L_2$-radius of $n \times n$ matrices $A_0$, $A_1$ expressing it in terms of maximal eigenvalue of linear operator $\mathscr{A}X = \frac{1}{2}({A_{0}^T X A_0 + A_{1}^T X A_1})$ 
which acts on the space of symmetric $n \times n$ matrices $X$. We have $$\rho_2 = \sqrt{\lambda_{\max}(\mathscr{A})}.$$  
Since the operator has an invariant cone (the cone of positive definite matrices), the largest eigenvalue $\lambda_{max}(\mathscr{A})$ 
is nonnegative by the Krein-Rutman theorem \cite{KR}. The matrix of the operator $\mathscr{A}$ 
is given by the  formula $$\frac{1}{2}(A_0 \otimes A_0 + A_1 \otimes A_1),$$ where $\otimes$ denotes the Kronecker product of matrices \cite{P97, BN}. Thus, the computation of the $L_2$-radius is reduced to the computation of the leading eigenvalue of the linear operator in dimension $\frac{n^2 + n}{2}$. 

Most likely, an explicit formula for the joint spectral radius does not exist. Moreover, it is known that the problem of its computation for general matrices with rational coefficients is undecidable  
and for boolean matrices it is NP-complete \cite{BT}. Nevertheless, in most cases, in relatively small dimensions (up to 25), it is possible to find the exact value of the joint spectral radius by the so-called invariant polytope algorithm \cite{GP}. We apply the upgraded version of this algorithm presented in \cite{TM}. 
\end{remark}

\begin{table}

\begin{center}
\begin{tabular}{|c|c|c|c|c|c|}\hline

\text{B-spline} & $B_0$ & $B_1$  & $B_2$ & $B_3$ & $B_4$ \\ \hline
\text{Square} & 0.5 & 1.5 & 2.5 & 3.5 & 4.5 \\ \hline
\text{Dragon} & 0.2382 & 1.0962 & 1.8039 & 2.4395 & 3.0557 \\ \hline
\text{Bear} & 0.3946 & {1.5372} & {2.6323} & {3.7092} & {4.7668} \\ \hline
\end{tabular}
\caption{The $L_2$-regularity of two-digit tile B-splines.}
\label{tab1}

\end{center}
\end{table}

\begin{table}

\begin{center}
\begin{tabular}{|c|c|c|c|c|}\hline
\text{B-spline} & $B_0$ & $B_1$  & $B_2$ & $B_3$\\ \hline
\text{Square} & 0 & 1 & 2 & 3 \\ \hline
\text{Dragon} & 0 & 0.47637 & 1.5584 & 2.1924 \\ \hline
\text{Bear} & 0 & 0.7892 & {2.2349} & {3.0744} \\ \hline
\end{tabular}
\caption{The regularity in $C$ of two-digit tile B-splines.}
\label{tab2}

\end{center}
\end{table}

The values of regularity of tile B-splines up to order 4 are given in Tables \ref{tab1}, \ref{tab2}. They lead us to the following theorem. 

\begin{theorem} \label{th_smooth}
Tile B-splines Bear-3 and Bear-4 are $C^2(\R^2)$ and $C^3(\R^2)$ respectively. 
\end{theorem}

This is well known that the classical bivariate B-splines of corresponding orders are not from $C^2$ ($C^3$ respectively). 

\begin{remark}
The statement of the theorem could seem paradoxical since the regularity of fractal B-spline turns out to be higher than the regularity of a rectangular one. Possibly it is related to the fact that the cube has plane faces that are slowly smoothed with autoconvolutions, and in tiles with fractal structure the smoothing is faster. Fig. \ref{pic_auto} illustrates the difference between the autoconvolution of indicators of a square and a disc, in case of a disc the area of intersection of $\varphi(x), \varphi(x + h)$ from formula for H\"older regularity decays much faster.     
\begin{figure}[ht!]
\begin{minipage}[h]{0.4\linewidth}
\center{\includegraphics[width=1\linewidth]{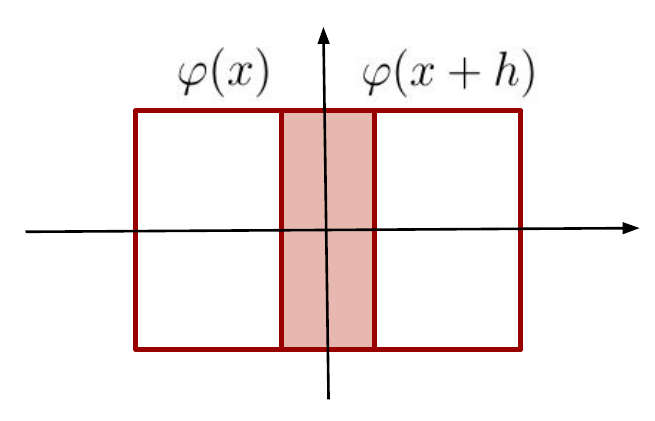}} $[\varphi * \varphi](h) \sim |h|$ \\
\end{minipage}
\quad
\begin{minipage}[h]{0.4\linewidth}
\center{\includegraphics[width=1\linewidth]{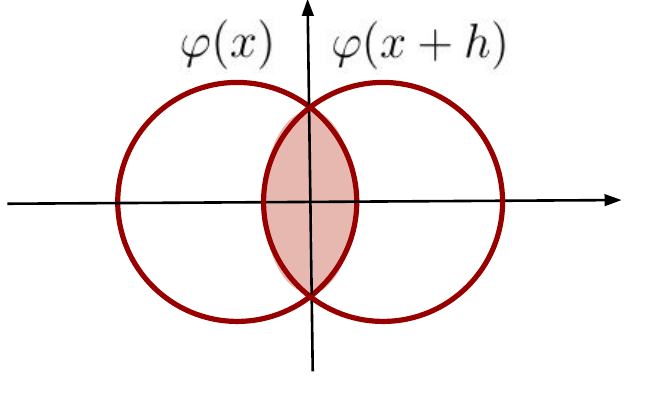}}
$[\varphi * \varphi](h) \sim |h|^{\frac{3}{2}}$ \\
\end{minipage}
\caption{The comparison of autoconvolutions of a square and of a disc.}
\label{pic_auto}
\end{figure}

Similar phenomenon was observed by P. Oswald in his investigations of subdivision schemes \cite{O, OS, JO}. Note also that in one-dimensional case among all refinement equations with a given number of coefficients the B-spline has the maximum regularity of a solution \cite{CDM}. As we see, Bears possess the maximal smoothness for functions of two variables, four and five coefficients. For B-splines of higher order we cannot calculate the regularity because of large amount of calculations of the spectral radius. 
\end{remark}

The figures \ref{riss:0} -- \ref{riss:8} contain the graphs of partial derivatives of order one, two and three for Bear-4. 

\begin{figure}[h]
\begin{center}
\begin{minipage}[h]{0.3\linewidth}
\includegraphics[width=1\linewidth]{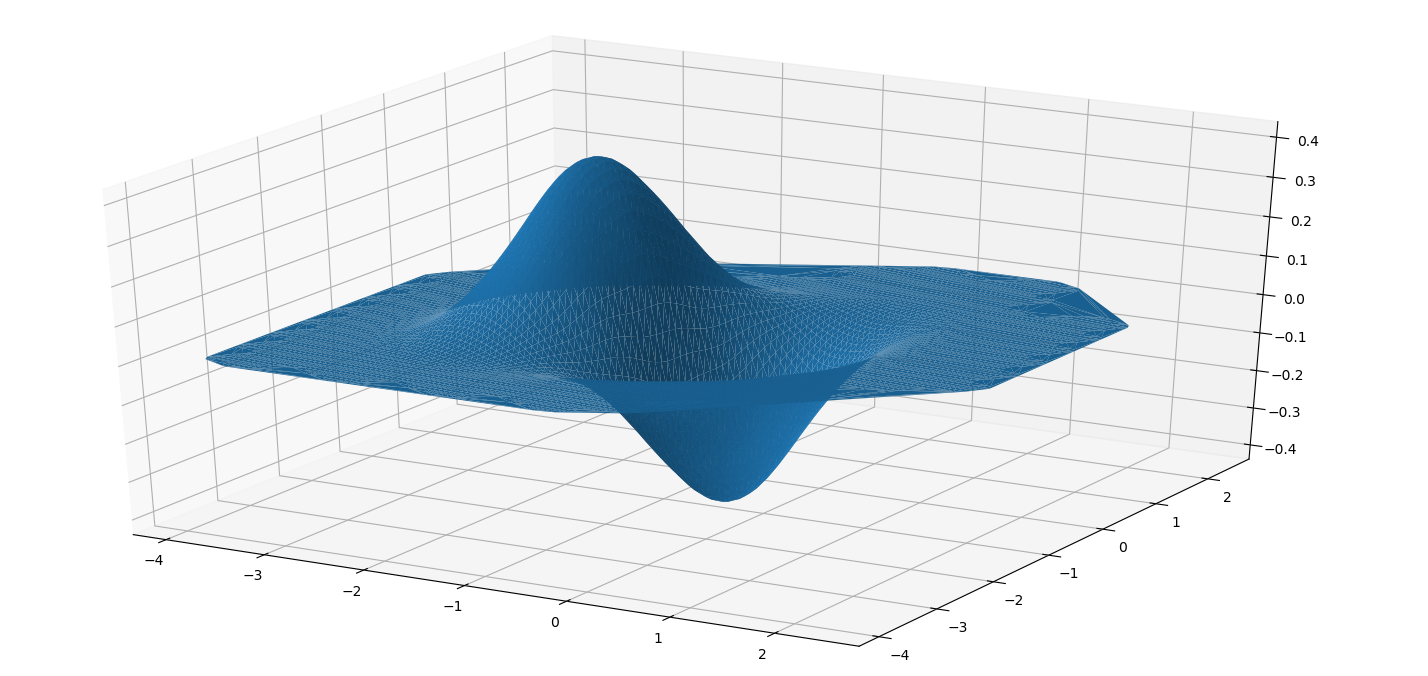}
\caption{The partial derivative of Bear-4 in $x$.}
\label{riss:0}
\end{minipage}
\hfill
\begin{minipage}[h]{0.3\linewidth}
\includegraphics[width=1\linewidth]{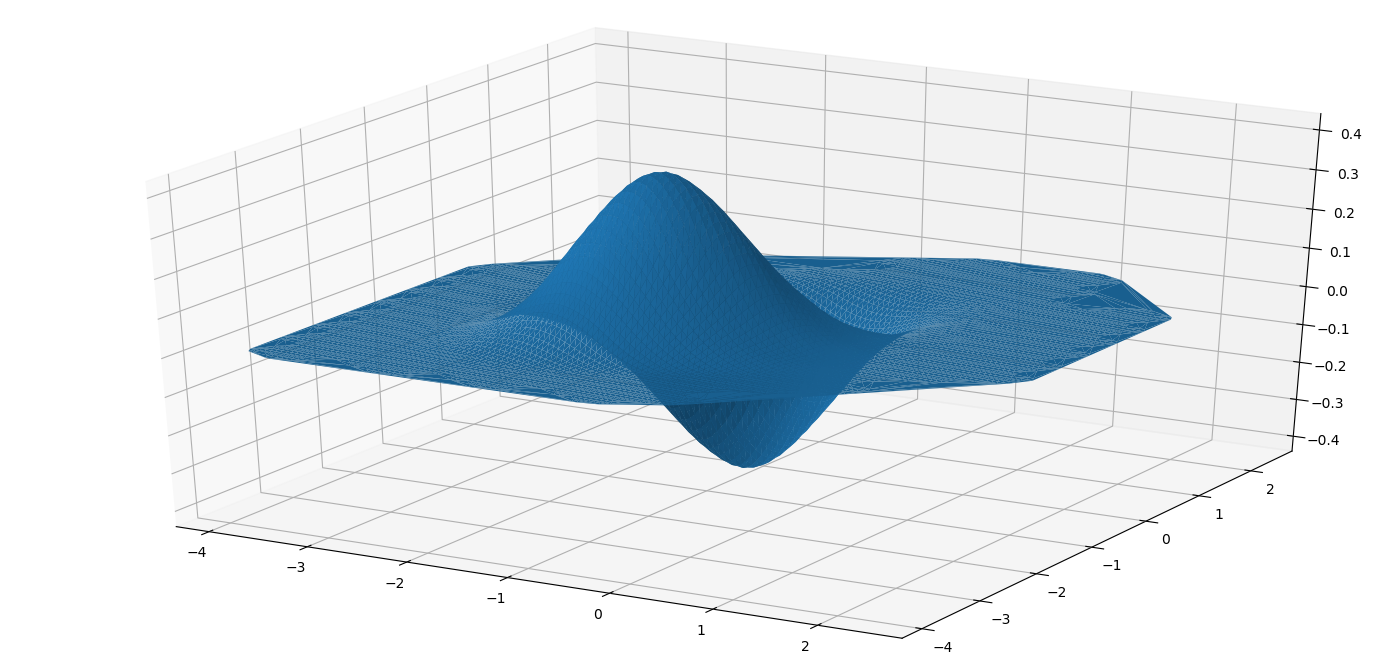}
\caption{The partial derivative of Bear-4 in $y$.}
\label{riss:1}
\end{minipage}
\hfill
\begin{minipage}[h]{0.3\linewidth}
\includegraphics[width=1\linewidth]{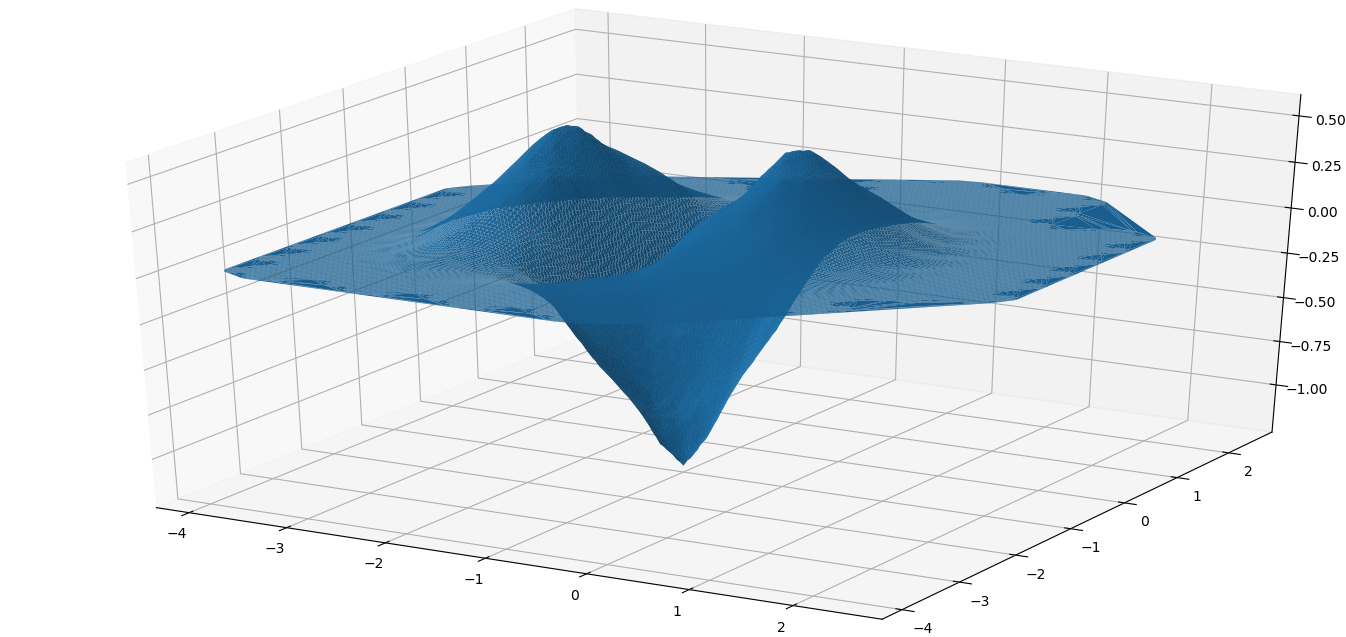}
\caption{The partial derivative of Bear-4 in $x, x$.}
\label{riss:2}
\end{minipage}
\vfill
\begin{minipage}[h]{0.3\linewidth}
\includegraphics[width=1\linewidth]{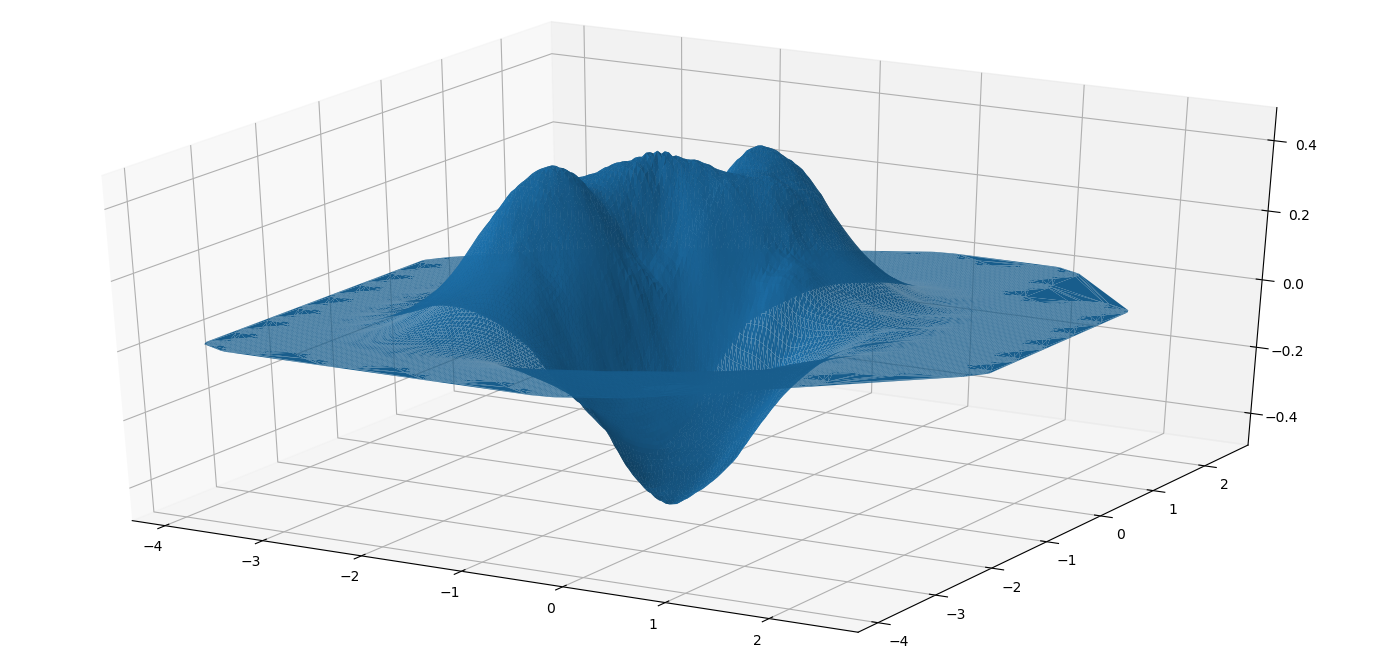}
\caption{The partial derivative of Bear-4 in $x, y$.}
\label{riss:3}
\end{minipage}
\hfill
\begin{minipage}[h]{0.3\linewidth}
\includegraphics[width=1\linewidth]{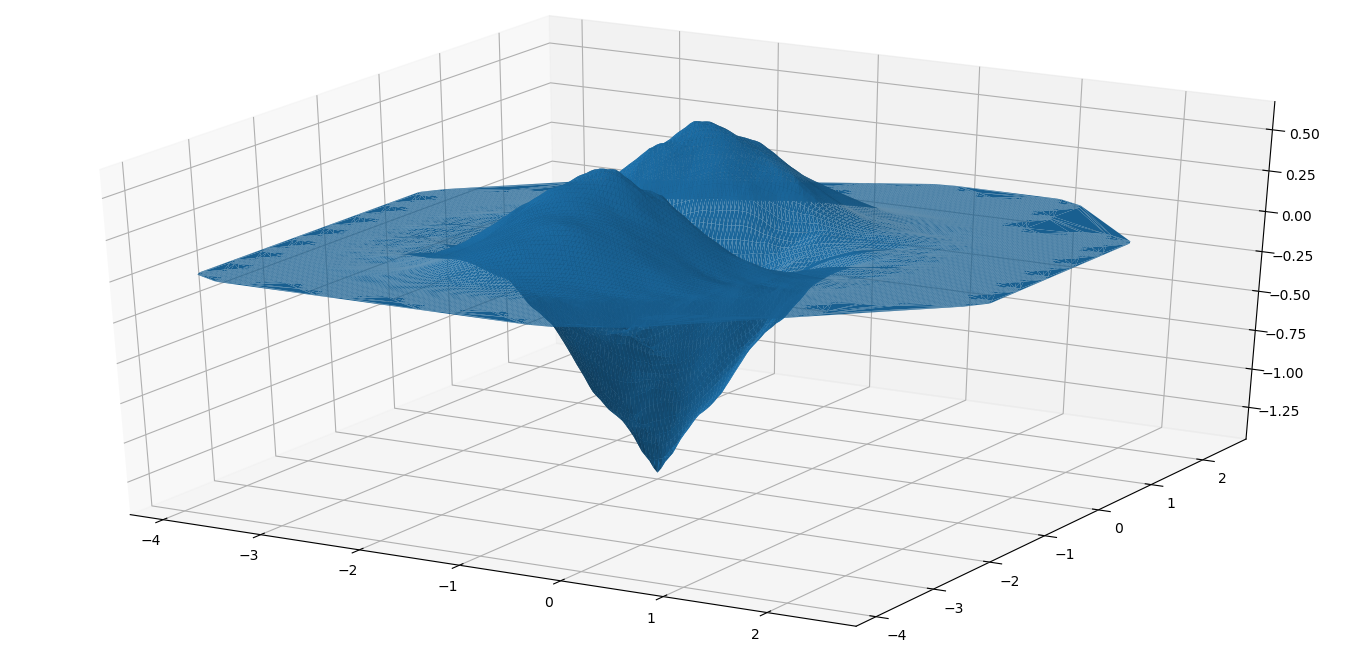}
\caption{The partial derivative of Bear-4 in $y, y$.}
\label{riss:4}
\end{minipage}
\hfill
\begin{minipage}[h]{0.3\linewidth}
\includegraphics[width=1\linewidth]{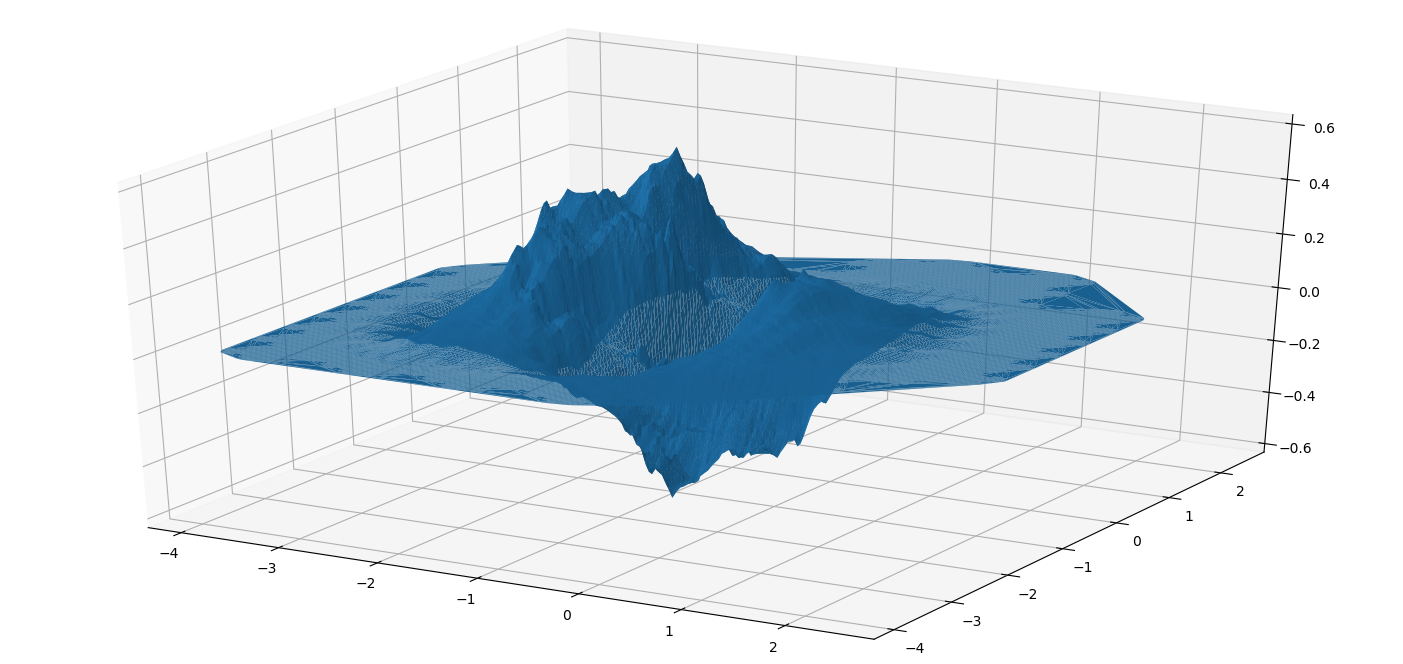}
\caption{The partial derivative of Bear-4 in $x, x, x$.}
\label{riss:5}
\end{minipage}
\vfill
\begin{minipage}[h]{0.3\linewidth}
\includegraphics[width=1\linewidth]{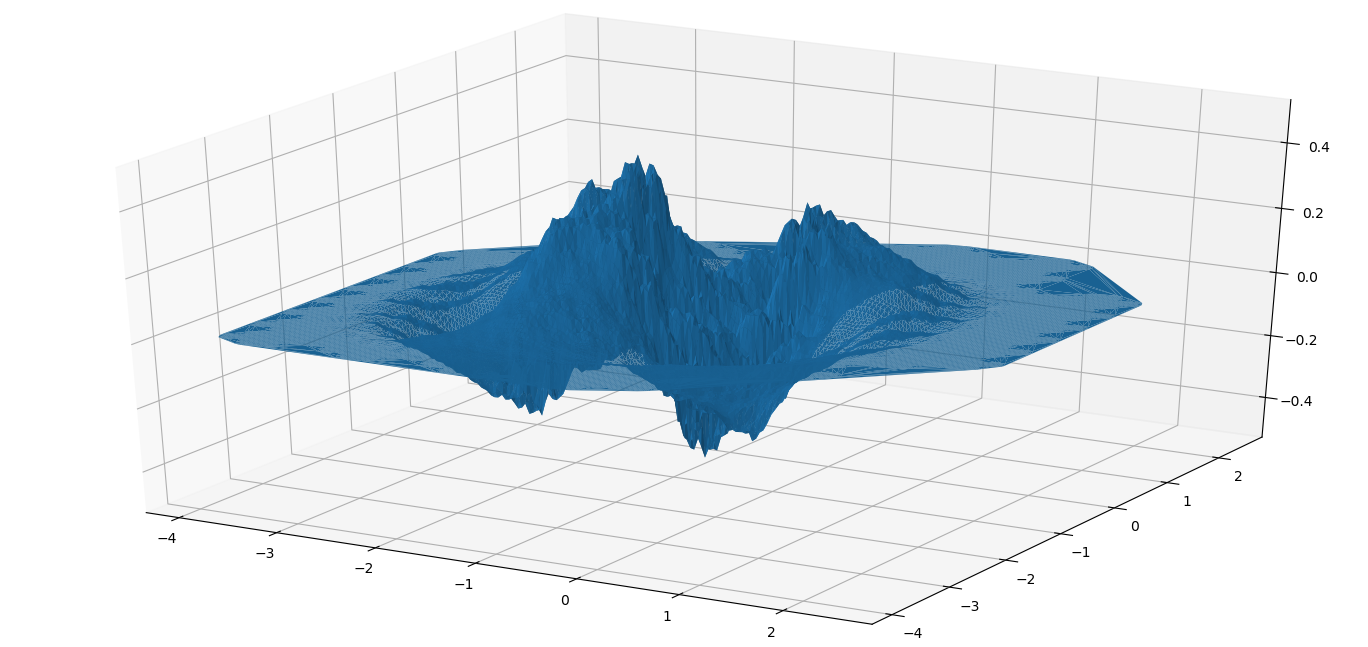}
\caption{The partial derivative of Bear-4 in $x, x, y$.}
\label{riss:6}
\end{minipage}
\hfill
\begin{minipage}[h]{0.3\linewidth}
\includegraphics[width=1\linewidth]{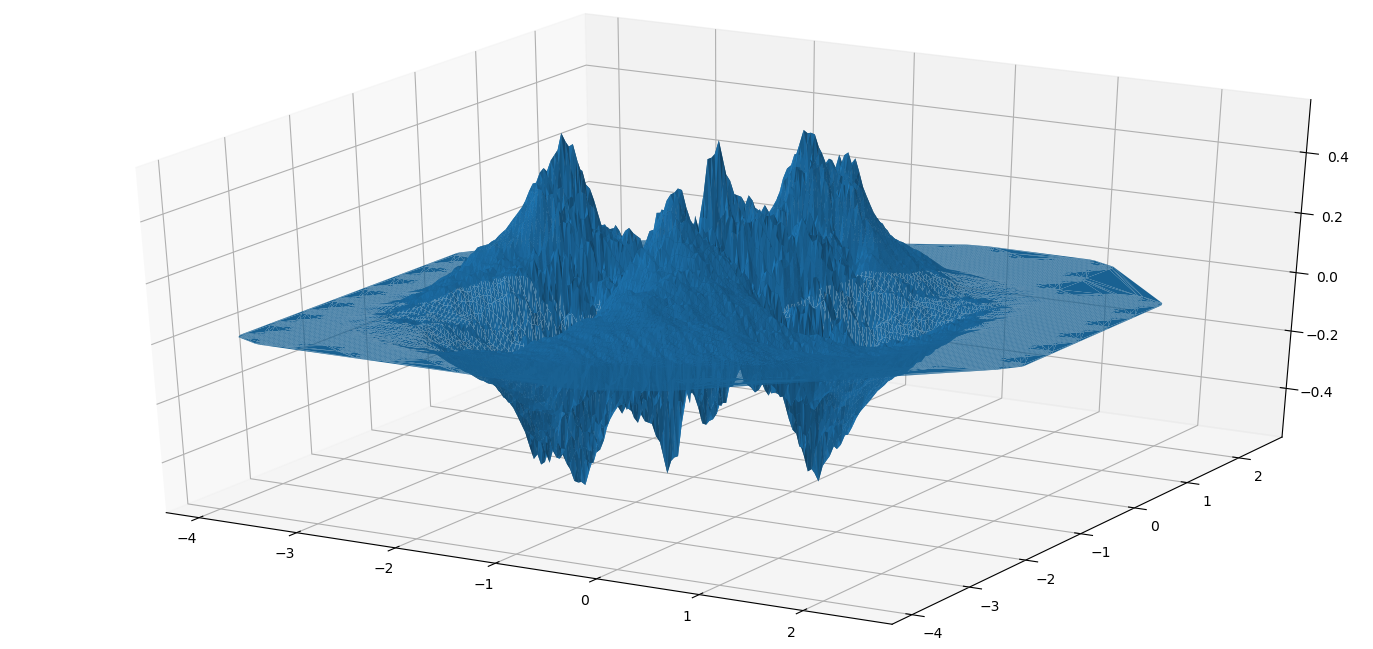}
\caption{The partial derivative of Bear-4 in $x, y, y$.}
\label{riss:7}
\end{minipage}
\hfill
\begin{minipage}[h]{0.3\linewidth}
\includegraphics[width=1\linewidth]{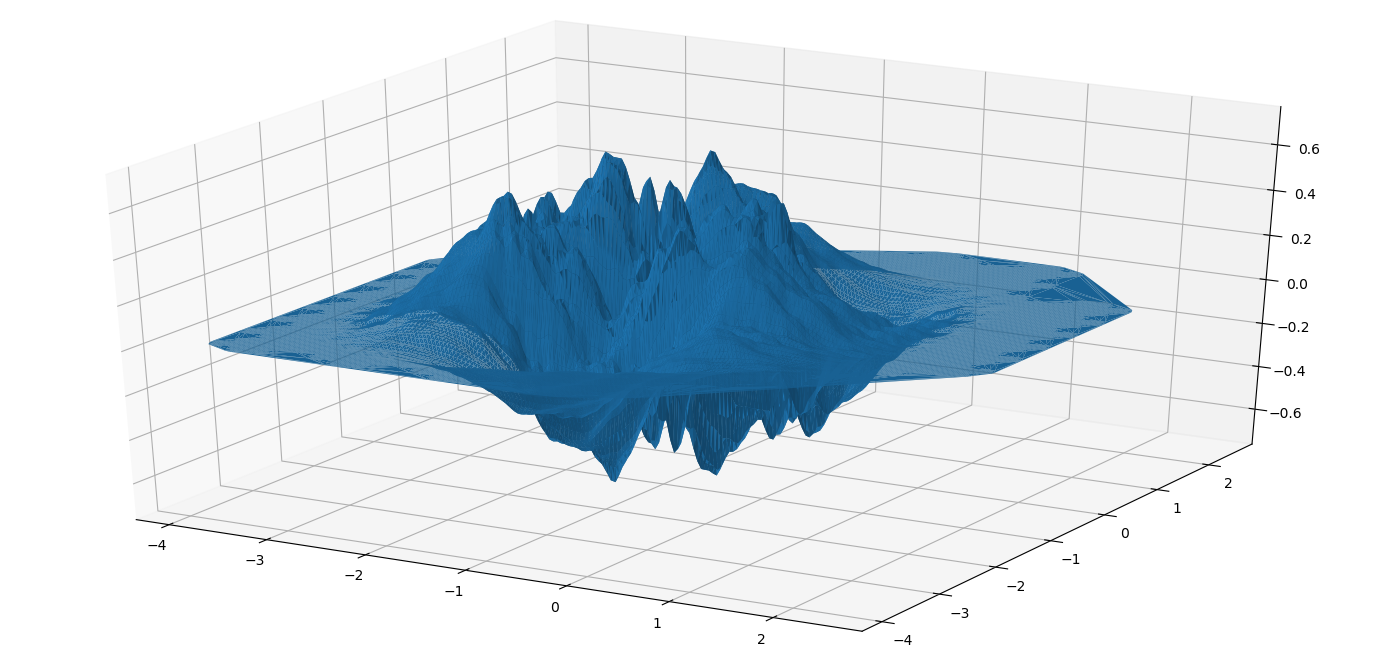}
\caption{The partial derivative of Bear-4 in $y, y, y$.}
\label{riss:8}
\end{minipage}
\end{center}
\end{figure}

\begin{remark}The values of regularity for large orders of tile B-splines are not given since the computation of joint spectral radius is hard for large sizes of matrices. The transition matrices grow fast when we increase the order of convolution. The question of the asymptotics of the regularity with the increase of order of B-splines remains open. 
\end{remark} 

\end{section}

\begin{section}{Subdivision schemes} \label{par_subdiv}
In this Section we apply the obtained tile B-splines to the construction of special class of subdivision schemes (SubD algorithms). 

Subdivision schemes are linear iterative algorithms for interpolating or extrapolating functions by given values on some rough lattice or mesh. The resulting surface is the limit of iterative approximations constructed on each iteration by the values computed on increasingly dense lattice. 
For the planar lattice, the limit surface is a graph of the limit function. The manifolds can also be obtained (see the end of Section \ref{par_subdiv}). 
Further we consider these algorithms and their properties. 

Let again $M \in \Z^{d \times d}$ be an expanding matrix. For an arbitrary mask (the set of numbers) $\{c_k\}$ the following subdivision (SubD) operator $S: \ell_{\infty}(\Z^d) \to \ell_{\infty}(\Z^d)$ is introduced: 
$$[Su](k) = \sum \limits_{j \in \Z^d}{c_{k - Mj} \cdot u(j)}, \quad u \in \ell_{\infty}(\Z^d).$$ 
After applying the subdivision operator several times, we obtain the sequence of values based on which the function could be constructed. 

\begin{ex}{The univariate case}. Let $M = 2$, $u$ be the sequence of values at integer points. One can construct a function $f_0(\cdot) = \sum_{k \in \Z} u(k) \chi_{[0, 1]}(\cdot - k)$ constant on the segments $[k, k + 1], k \in \Z$. After $t$ iterations of the subdivision operator we obtain a function 
$$f_q(\cdot) = \sum_{k \in \Z} [S^qu](k) \chi_{[0, 1]}(2^q\cdot - k),$$
that is constant on the segments $[\frac{k}{2^q}, \frac{k + 1}{2^q}], k \in \Z$. 

Instead of the function $\chi_{[0, 1]}$, one can consider any other function $h(x)$ satisfying a partition of unity property, i.e., $\sum_{j \in \Z}h(x - j) \equiv 1$, that is equivalent to $\widehat{h}(0) = 1, \widehat{h}(s) = 0, s \ne 0$. 
For example, $h(x) = B_1(x)$ gives piecewise-linear functions $f_q$ on each iteration.  
If for an admissible function $h$, there is a limit of functions $f_q$ as $q \to \infty$ in $L_{\infty}$ for every sequence $u$, then it is said that the subdivision scheme \textit{converges}. The corresponding limit is called \textit{the limit function} for given $u$. 
\end{ex}

We do the same in the general multivariate case. 
Define $$f_q = \sum \limits_{k \in \Z^d} [S^q u](k) \chi_{[0, 1]^d}(M^q \cdot - k) \text{ (piecewise-constant approximation)}.$$
Instead of $\chi_{[0, 1]^d}$ we can use any function $h(x)$ with the property $\sum_{j \in \Z^d}h(x - j) \equiv 1$. In particular, it could be the indicator of an arbitrary tile $h = \chi_G$, a multivariate B-spline $B_1$, etc.  

Further we consider the convergence of the algorithm in $C^n$. For $n \ge 0$, we define a function space  
$$Q_n = \left\{h \in C^n(\R^d) \mid \widehat{h}(0) = 1, \widehat{h}\,\,\, \parbox{9em}{ has zeros of order at least } \,\,n + 1 \,\,\parbox{7em}{ at each point } \Z^d \setminus \{0\}\right\}.$$
Then $h \in Q_n$ if and only if $h \in C^n{(\R^d)}$, $\sum_{j \in \Z^d}h(x - j) \equiv 1$, where every algebraic polynomial $P_n$ of variables $x_1, \ldots, x_d$ of degree at most $n$ is in the linear span of $\{h(\cdot - j)\}_{j \in \Z^d}$ (\cite{BVR}). 

For example, we have $\chi_{[0,1]^d} \notin Q_0$, while the classical B-splines $B_1$ and $B_{n + 1}$ belong to $Q_0$ and $Q_n$ respectively. 

\begin{definition}  \label{defsubd}
A subdivision scheme converges in $C^n$ if for some $h \in Q_n$, for every $u \in \ell_{\infty}(\Z^d)$, there exists a function $f_u \in C^n(\R^d)$ such that 
 $$\left\|\sum \limits_{k \in \Z^d} [S^q u](k) h(M^q \cdot - k) - f_u(\cdot)\right\|_{C^n(\R^d)} \to 0 \quad as q \to \infty.$$
\end{definition}
\begin{remark} We can always choose $h = B_{n + 1}$ (the classical B-spline of order $n$). Indeed, the convergence does not depend on the choice of initial function $h \in Q_n$, i.e., we can replace the words ``for some $h \in Q_n$'' in Definition \ref{defsubd}  by ``for every $h \in Q_n$''. 
\end{remark}

The operator $S$ is linear and invariant with respect to integer shifts, i.e., applying $S$ to the shift of the sequence $u$ by $k$ we obtain the shift of the sequence $Su$ by $k$.  
Therefore, it is sufficient to know the limit function only for the $\delta$-sequence $\delta(k) = \delta^0_k$, we denote it by $f_{\delta}$ (in case of convergence). Then, for arbitrary sequence $u \in \ell_{\infty}(\Z^d)$, the limit function  has the form  
 $$f_{u}(x) = \sum \limits_{k \in \Z^d}{f_{\delta}(x - k) \cdot u(k)},$$
since $u = \sum \limits_{k \in \Z^d} u(k)\delta(k)$. 
It turns out \cite{CDM} that the function $f_\delta$ satisfies the refinement equation with the coefficients $c_k$ of the subdivision operator: 
$$f_\delta(x) = \sum \limits_{k \in \Z^d}{c_k f_\delta(Mx - k)}.$$

We define the numbers $c_k$ as the coefficients of refinement equations generating tile B-splines. Therefore, the function $f_\delta$ is the tile B-spline $B_n^G$, and for an arbitrary initial sequence the limit function is a linear combination of the integer shifts of $B_n^G$. In particular, the regularity of the limit function coincides with the regularity of the tile B-spline.  

Note that the number of arithmetic operations required for applying the SubD-operator in each step  depends on the number of nonzero coefficients $c_k$ of the refinement equation. The less it is, the faster is the iteration of the subdivision scheme. 

As we obtained in Sections \ref{classic} and \ref{tile_spl}, the classical B-spline $B_n$ of $d$ variables of order $n$ has $(n + 2)^d$ nonzero coefficients. However, in every dimension we can consider the tile B-spline with $n + 2$ coefficients generated by a two-digit tile that is a parallelepiped (such tiles exist in every dimension, see \cite{Zai2}). 
Besides, both the B-splines and the limit surfaces of subdivision schemes coincide with the classical case, since they are the convolutions of the same indicators of parallelepipeds. The only difference is the way the algorithm is organized. We obtain the following theorem: 

\begin{theorem} \label{th_subd} A subdivision scheme in $\R^d$ based on a tile B-spline $B_n^G$, where $G$ is a two-digit tile that is a parallelepiped, has the complexity of one iteration equal to $n + 2$. The classical $d$-variate subdivision scheme of order $n$ based on the product of $d$ univariate B-splines $B_n$ has the complexity of one iteration $(n + 2)^d$. These algorithms generate the same limit surfaces. 
\end{theorem}

It is known (see, for example, \cite{CCS, DL}) that the necessary conditions for convergence of SubD algorithm in $C^n$ are  

1) the corresponding refinement equation has a solution $\varphi \in C^n$. In general case, it is only known that a refinement equation has always a unique solution in the space of tempered distributions $\mathcal{S}'$ up to multiplication by a constant \cite{NPS}. 
 
2) The mask $a$ of the equation satisfies \textit{the sum rules}, i.e., has zeros of order at least $n + 1$ 
at the points $M^{-T}\Delta_{*}$ for all $\Delta_{*} \in D_{*} \setminus \{0\}$, where $D_{*}$ is a digit set corresponding to the transposed matrix $M^T$, and $a(0) = 1$ (see, for example, \cite{CDM}). This condition could be simply rewritten as linear relations on the coefficients $c_k$ of the refinement equation. In particular, the sum rules of order $n = 0$ are equivalent to the property 
$\sum_{k} c_{Mk + \Delta} = 1$ for every vector $\Delta \in D$, where $D$ is the set of digits of the basis tile. 

These conditions are not sufficient, such examples are well-known \cite{KPS}. Nevertheless, if the conditions 1, 2 are satisfied and additionaly the limit refinable function $\varphi$ is stable, i.e., its integer shifts are linearly independent, then the algorithm necessarily converges in $C^n$ (\cite{CDM} for $n = 0$, \cite{P01} for $n \ge 1$).   
It is known that the function $\varphi$ is stable if and only if its Fourier transform has no periodic zeros, i.e., there is no point $\xi \in \R^d$, for which $\widehat \varphi (\xi + k) = 0$ for all $k \in \Z^d$ (see, for example, \cite{CDM}). 

\begin{propos} \label{pr_stab}
For every $n \ge 1$ the B-spline $B_n$ of $d$ variables of order $n$ is continuous, stable and its refinement equation satisfies the sum rules of order $n$. 
\end{propos}
\begin{proof} 
As the convolution of indicators of several compact sets is continuous, the function $B_n$ is continuous. 
The stability follows from the fact that the Fourier transform $\widehat B_n(\xi)$ does not have periodic zeros. Indeed, since the Fourier transform of convolution is the product of Fourier transforms of multipliers, it follows that $\widehat B_n(\xi) = (\widehat B_0(\xi))^{n + 1}$. As the integer shifts of tile are linearly independant, the corresponding Fourier transform $\widehat B_0(\xi)$ does not have a periodic zero, therefore, the Fourier transform of $\widehat B_n(\xi)$ does not have a periodic zero. 

Now we check the sum rules. By the definition of tile, the $B_0$ satisfies the rule of order zero, since $c_{\Delta} = 1$ for each $\Delta \in D$ and the digits are from different cosets $\Z^d / M\Z^d$. 
Let $D_{*}$ be an arbitrary digit set corresponding to the transposed matrix $M^T$. Then the mask $a_0$ satisfies the sum rule of order zero that can also be rewritten in frequency domain as follows: $a_0(M^{-T}\Delta_{*}) = 0$ for all $\Delta_{*} \in D_{*} \setminus \{0\}$ and  $a_0(0) = 1$. 
Since the mask of tile B-spline $B_n$ is equal to $a_n(M^{-T}\Delta_{*}) = a_0(M^{-T}\Delta_{*})^{n + 1}$, the function $B_n$ satisfies the sum rules of order $n$. 
\end{proof}
\begin{cor}\label{cor_conv}
Let the H{\" o}lder regularity of the tile B-spline $B_n$ be $\alpha$. Then the subdivision algorithm based on $B_n$ converges in $C^k$ for every $k \le \alpha$. 
\end{cor}

One of the most important issues of subdivision algorithm theory is the rate of convergence. In \cite{CJ} the rate of convergence of subdivision algorithms in $C^n(\R)$ (generalized rate of convergence) was defined by means of the  difference schemes. 
Then it was generalized to the multivariate case (\cite{CGV}). We use similar definition in terms of the work \cite{P01}. For simplicity, we will assume that the matrix $M$ is isotropic, i.e., it is diagonalizable and all its eigenvalues are equal in moduli (similarly, the general case is considered, see \cite{CP}). 
 
It turns out \cite{CDM, CCS} that the algorithm converges in $C^n$ with exponentional rate: for every $u \in \ell_{\infty}, \|u\|_{\ell_{\infty}} = 1$, $r = 0, \ldots, n$, we have    
$$\|f_q^{(r)} - f^{(r)}\|_{C(\R^d)} \le C \cdot \tau_r^{-q},$$
where $f = f_u$ is a limit function for $u$, $f_q$ is the result of the $q$-th iteration. The exponents of convergence $\tau_0 = \ldots = \tau_{n - 1} = m^{\frac{1}{d}}$, where $m = |\det{M}|$, the value $\tau_n$ depends on the coefficients of subdivision algorithm. 
In particular, we have $$\tau_n= \rho_{C}(T_0|_{W_k}, T_1|_{W_k}) \cdot {m}^{\frac{n}{d}},$$ 
where   
$$\rho_{C}(A_0, A_1) = \lim \limits_{s \to \infty} \max \limits_{\sigma} \|A_{\sigma(1)}\ldots A_{\sigma(s)}\|^{1/s}, \, \sigma\colon \{1, \ldots, s\} \to \{0, 1\}$$ 
is a joint spectral radius of two operators (see Section \ref{par_smooth} for details and definition of $k$).  

\begin{definition} The generalized rate of convergence of the subdivision algorithm is a number $n - \frac{1}{d}\log_m{\tau} = -\frac{1}{d}\log_m \rho$. 
\end{definition}

Then the following fact is well-known \cite{P01} 
\begin{propos} The H{\" o}lder regularity $\alpha_\varphi$ of a limit function $\varphi$ of the subdivision algorithm is at least the generalized rate of convergence. If the function $\varphi$ is stable, these parameters coincide. 
\end{propos}
\begin{cor}
Let the tile B-spline $B_n^G$ be from $C^n(\R^d)$. Then the corresponding subdivision algorithm converges with the generalized rate $n - \frac{1}{d}\log_m{\tau} = -\frac{1}{d}\log_m \rho = \alpha_\varphi$. In particular, we have $\tau = m^{\frac{n - \alpha_\varphi}{d}}$. 
\end{cor}

Since tile B-splines are stable, we can estimate the generalized rate of convergence of their subdivision algorithms. Bear-4 is the smoothest spline among those considered in Table \ref{tab2}.  

\begin{theorem} \label{th_subd}
The subdivision algorithms constructed by the tile B-splines Bear-3 and Bear-4 converge in $C^2$ and $C^3$ respectively. 
\end{theorem}

The classical subdivision algorithms constructed by bivariate B-splines of the corresponding orders do not converge in $C^2$  ($C^3$ respectively). 

Note that the convergence of the algorithm in the space of functions of high regularity strongly influences the quality of generated surface, for example, the convergence in $C^2$ means that in every point the curvature of surface will converge to that of limit surface. In particular, if the limit surface is locally convex at some point, then the surfaces obtained after several iterations are locally convex as well. 

Since Bear-4 is optimal in terms of its rate of convergence and has a small number of nonzero coefficients (only five), let us consider the subdivision algorithm based on Bear-4. That scheme can be applied both in case when we initially have a function given at integers on the plane (for example, in the image processing) and in case when initial points form a rough approximation of a surface. Fig. \ref{shablon} shows a computation template. The values at the points marked with circles are updated by the iteration of the subdivision scheme with a linear combination of values of the neighbouring points. Each iteration the direction of computation is changed by means of the transform $M^{-1}$.   

\begin{figure}[ht]
\centering
\includegraphics[width=0.15\linewidth]{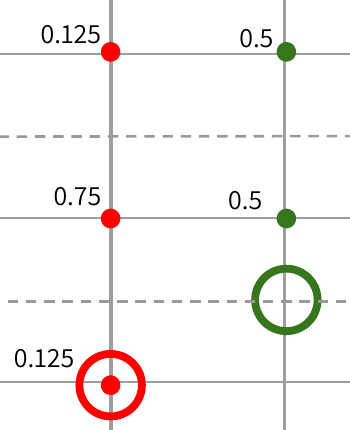}
\caption{The computation template of Bear-4 algorithm.}
\label{shablon}
\end{figure}

First consider the application of the Bear-4 subdivision algorithm to a surface of slightly deformed torus given by a rough approximation (see Fig. \ref{ris:0}). The next figures \ref{ris:1} -- \ref{ris:5} show the results after several iterations of Bear-4. 

Fig. \ref{ris:100} -- \ref{ris:102} illustrate the application of the Bear-4 algorithm to a surface with border with example of a catenoid. 

\begin{figure}[h]
\begin{center}
\begin{minipage}[h]{0.3\linewidth}
\includegraphics[width=1\linewidth]{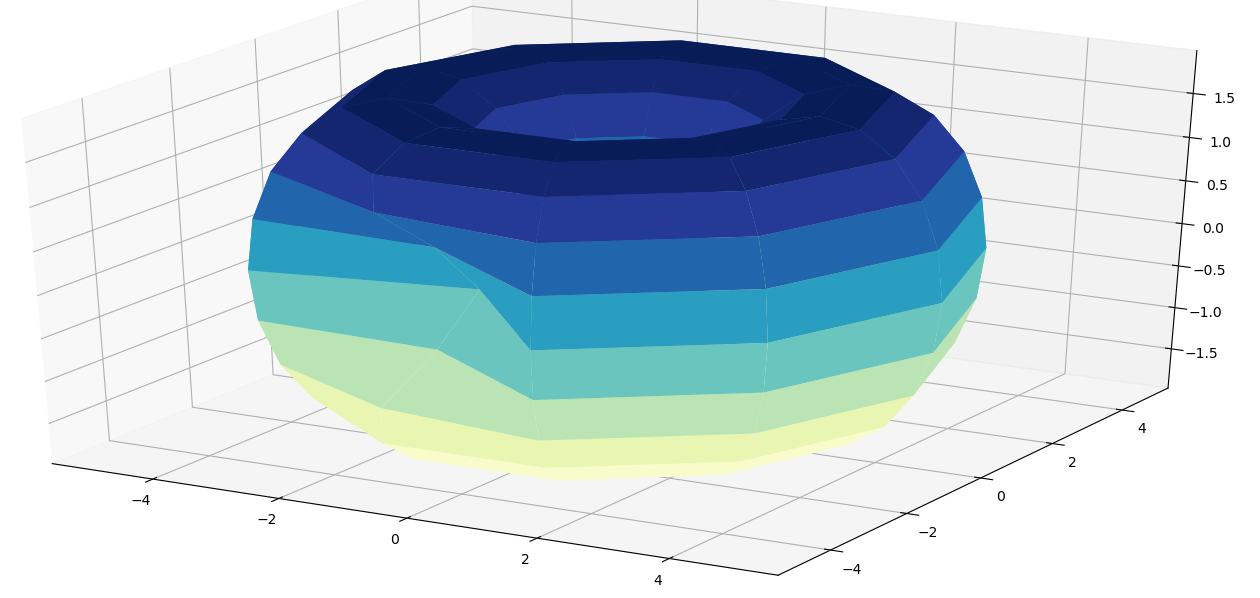}
\caption{Initial approximation.}
\label{ris:0}
\end{minipage}
\hfill
\begin{minipage}[h]{0.3\linewidth}
\includegraphics[width=1\linewidth]{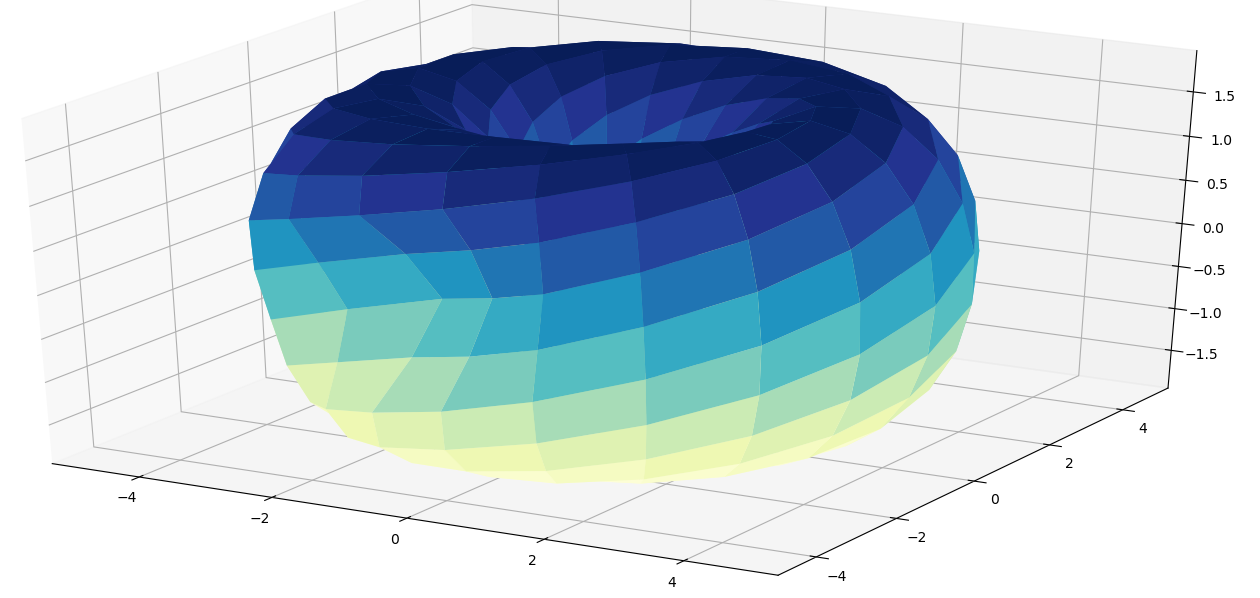}
\caption{After one iteration.}
\label{ris:1}
\end{minipage}
\hfill
\begin{minipage}[h]{0.3\linewidth}
\includegraphics[width=1\linewidth]{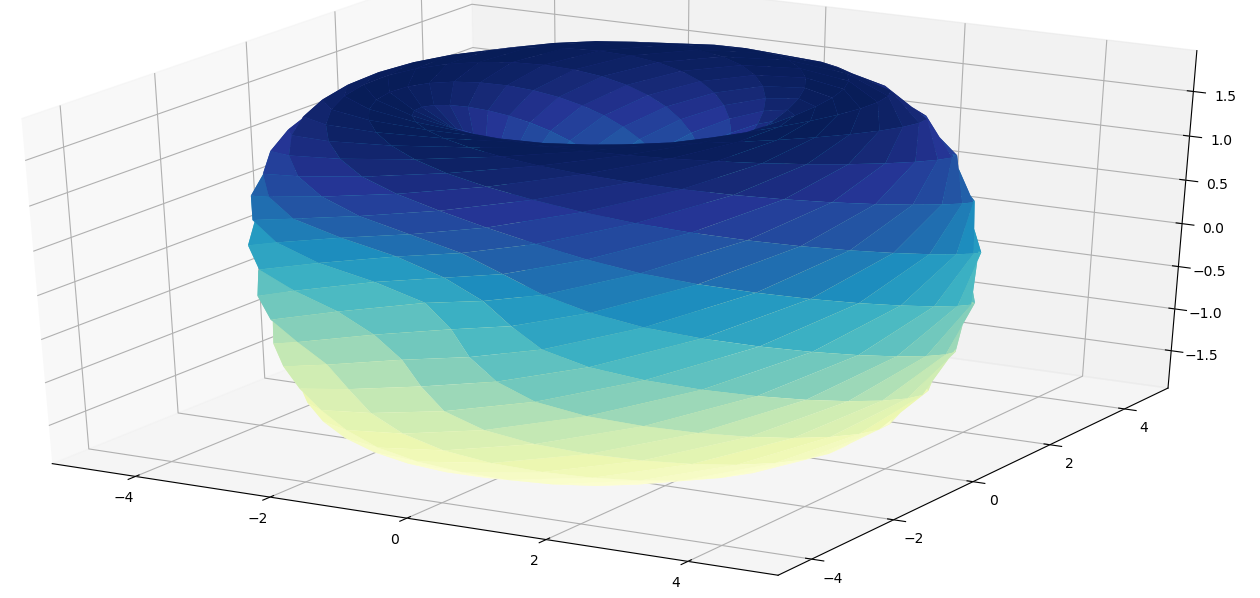}
\caption{After two iterations.}
\label{ris:2}
\end{minipage}
\vfill
\begin{minipage}[h]{0.3\linewidth}
\includegraphics[width=1\linewidth]{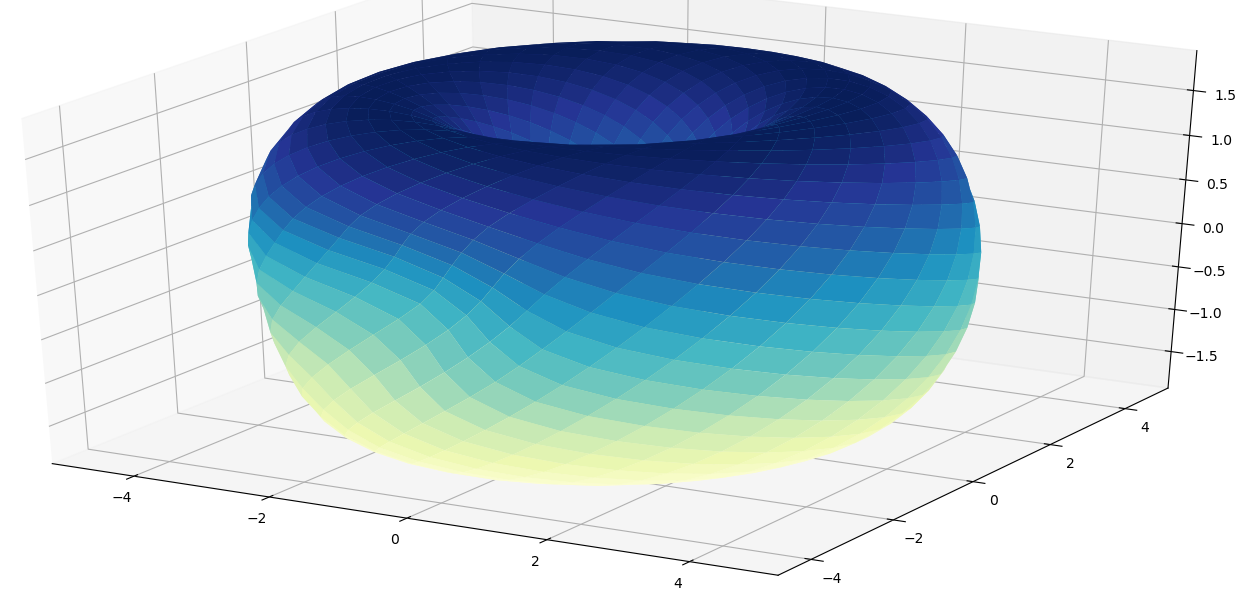}
\caption{After three iterations.}
\label{ris:3}
\end{minipage}
\hfill
\begin{minipage}[h]{0.3\linewidth}
\includegraphics[width=1\linewidth]{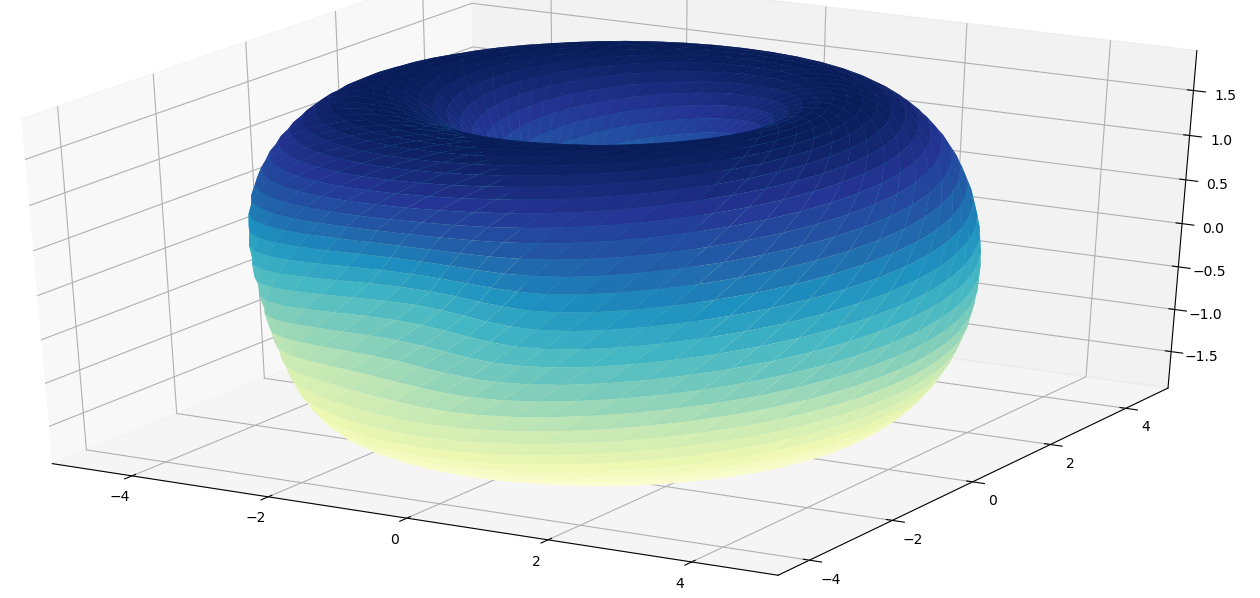}
\caption{After four iterations.}
\label{ris:4}
\end{minipage}
\hfill
\begin{minipage}[h]{0.3\linewidth}
\includegraphics[width=1\linewidth]{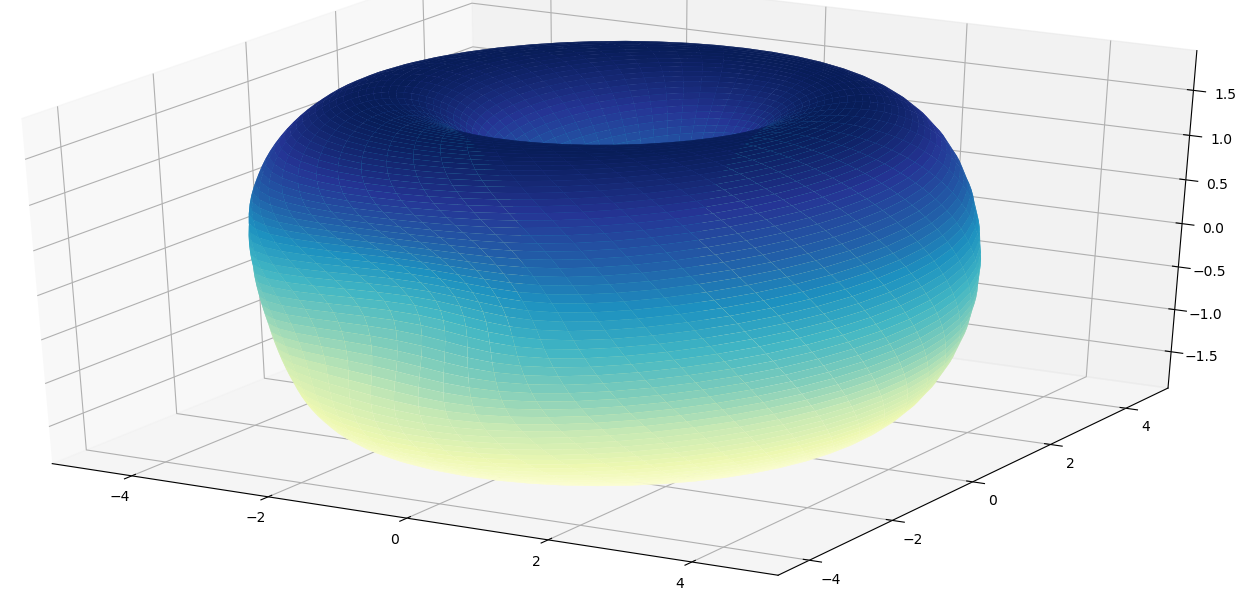}
\caption{After five iterations.}
\label{ris:5}
\end{minipage}
\end{center}
\end{figure}

\begin{figure}[h]
\begin{center}
\begin{minipage}[h]{0.3\linewidth}
\includegraphics[width=1\linewidth]{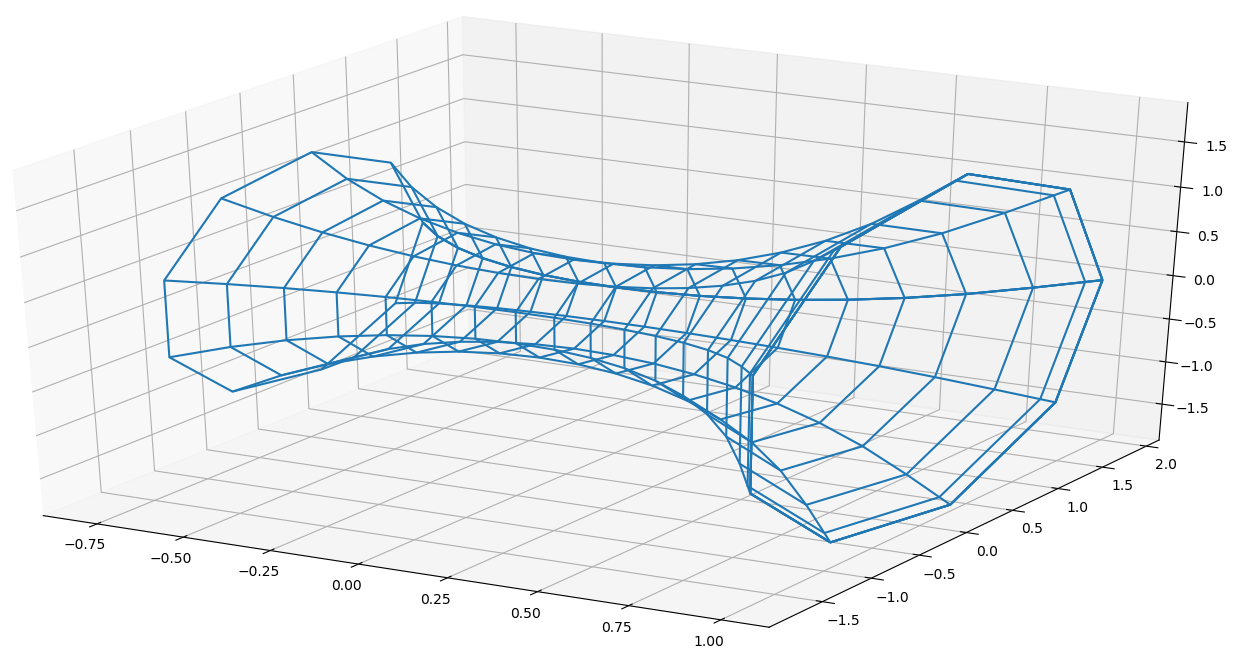}
\caption{Initial approximation.}
\label{ris:100}
\end{minipage}
\hfill
\begin{minipage}[h]{0.3\linewidth}
\includegraphics[width=1\linewidth]{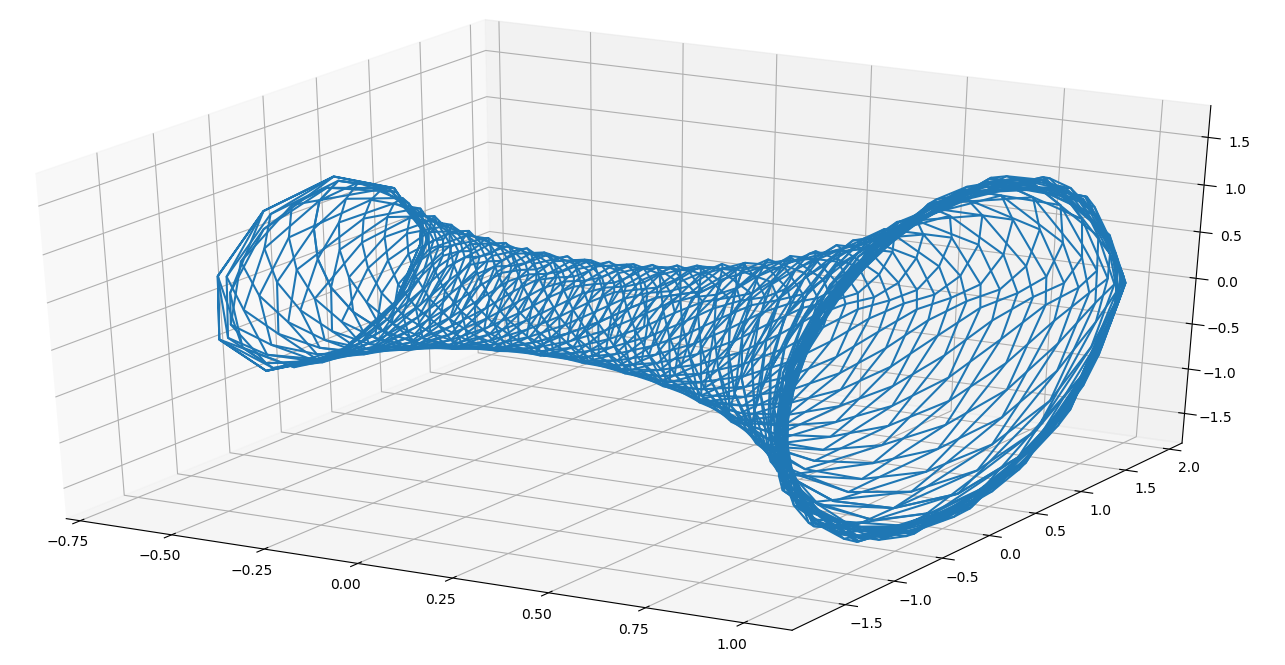}
\caption{After two iterations.}
\label{ris:101}
\end{minipage}
\hfill
\begin{minipage}[h]{0.3\linewidth}
\includegraphics[width=1\linewidth]{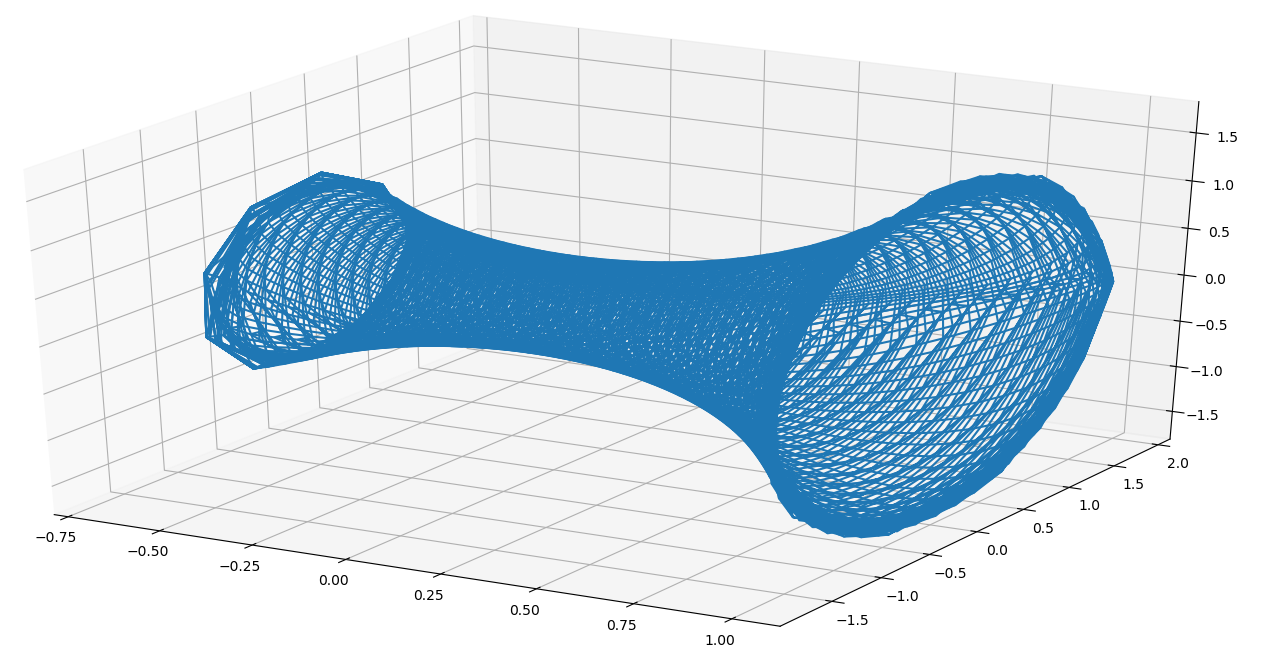}
\caption{After four iterations.}
\label{ris:102}
\end{minipage}
\end{center}
\end{figure}
\end{section}

\begin{section}{Conclusion}
The paper presents an approach for construction and analysis of multivariate B-splines based on convolutions of tiles. We revealed a series of properties of the tile B-splines and investigated in details the case of planar symmetric 2-tiles (Square, Dragon and Bear). They are solutions of refinement equations with a small number of nonzero coefficients that gives them the advantage over the classical multivariate B-splines. The orthogonalization of tile B-splines defines the orthonormal wavelet systems generated by the only wavelet function, for which we obtained the explicit formulas, computed the regularity exponents and estimated the rate of decay at infinity. Using multivariate complex analysis, we estimated the rate of decay and the number of coefficients that are required for approximation of wavelet function with given accuracy. Some of the constructed tile B-splines have a higher regularity than classical B-splines of the same orders. In particular, Bear-4 is three times differentiable in contrast to the corresponding classical B-spline. 
This property is important for applications, in particular, for subdivision algorithms in geometric modeling. For their convergence in $C^n$ we require the  corresponding regularity of the generating function. We provide the examples and numerical results and the implementation for practical application of this work. 
\end{section}

\section{Acknowledgements}
The author is grateful for her advisor V.Yu.~Protasov for his constant support and help in work and for the reviewer for many useful remarks.  
The author is thankful for the developers of the package \cite{Mekh} using which the tiles are constructed.

\appendix

\begin{section}{The tables with coefficients of wavelet functions}

\begin{table}[H]
\small
\begin{center}
\begin{tabular}{c|c|c|c|c|c|c|c|c}
$i$ & 1 & 0 & 3 & 4 & 1 & 3 & 3 & -3\\ \hline
$j$ & 0 & 0 & 0 & 0 & 1 & 1 & -1 & 0\\ \hline
\!\!\!\!\! $c_{i,j}$ \!\!\!\! & 1.15586 & 0.55632 & -0.09441 & -0.06459 & 0.06225 & -0.04478 & -0.03976 & 0.01911 \!\!\!\! \\ \hline \hline

$i$  & 5 & 4 & 4 & 2 & 5 & 0 & -4 & 4\\ \hline
$j$  & 1 & -1 & 1 & -1 & -1 & -2 & 0 & 2 \\ \hline
\!\!\!\!\! $c_{i,j}$ \!\!\!\!  & 0.01591 & 0.01557 & 0.01535 & -0.01304 & 0.01256 & -0.00979 & 0.00935 & 0.00911 \!\!\!\! \\ \hline \hline

$i$  & 2 & -4 & -4 & 6 & -5 & 5 & -5 & 2\\ \hline
$j$  & 1 & -1 & 1 & 2 & -1 & 2 & 1 & -2 \\ \hline
\!\!\!\!\! $c_{i,j}$ \!\!\!\!  & -0.00862 & -0.00644 & -0.00543 & -0.00430 & -0.00418 & -0.0041 & -0.00350 & 0.00339
 \!\!\!\! \\  \hline \hline

 $i$  & 3 & -5 & -5 & -6 & 5 & -6 & 8 & 6\\ \hline
$j$  & 1 & -1 & 1 & 2 & -1 & 2 & 1 & -2 \\ \hline
\!\!\!\!\! $c_{i,j}$ \!\!\!\!  & 0.00306 & -0.00292 & 0.00207 & 0.00172 & -0.00156 & 0.00153 & 0.00151 & -0.00124
 \!\!\!\! \\  \hline \hline

$i$  & 4 & -6 & 5 & 7 & -1 & -7 & 6 & -7\\ \hline
$j$  & 2 & 0 & -2 & -1 & 3 & -2 & -1 & -2  \\ \hline
\!\!\!\!\! $c_{i,j}$ \!\!\!\!  & 0.00123 & -0.00113 & -0.00108 & 0.00102 & -0.00101 & 0.00095 & 0.00091 & 0.00085
 \!\!\!\! \\  \hline \hline

$i$  & -1 & 1 & -5 & -7 & -6 & 4 & 8 & 9\\ \hline
$j$  & -2 & 0 & -2 & 3 & 3 & -1 & 3 & 1  \\ \hline
\!\!\!\!\! $c_{i,j}$ \!\!\!\!  &  0.00081 & 0.00079 & 0.00074 & -0.00073 & 0.00059 & -0.00057 & -0.00057 & -0.00047 
 \!\!\!\! \\  \hline \hline

$i$  & -8 & 4 & 2 & 2 & 10 & 10 & -7 & 5\\ \hline
$j$  & -3 & 3 & 2 & -2 & 2 & 3 & 3 & 3 \\ \hline
\!\!\!\!\! $c_{i,j}$ \!\!\!\!  & -0.00046 & 0.00046 & -0.00045 & -0.00040s & -0.00037 & -0.00036 & -0.00032 & 0.00032
 \!\!\!\! \\  \hline \hline

$i$  & -7 & 1 & 0 & -8 & 10 & 11 & 8 & 6 \\ \hline
$j$  & -2 & -3 & 4 & -3 & 1 & -1 & 2 & -3 \\ \hline
\!\!\!\!\! $c_{i,j}$ \!\!\!\!  &  0.00028 & 0.00027 & 0.00026 & 0.00025 & -0.00022 & 0.00021 & -0.00021 & 0.00019 
 \!\!\!\! \\  \hline \hline

$i$  & -1 & -9 & -9 & 11 & -5 & 4 & 9 & 10 \\ \hline
$j$  & 2 & 1 & -3 & 1 & -4 & 4 & 4 & 4 \\ \hline
\!\!\!\!\! $c_{i,j}$ \!\!\!\!  &  -0.00018 & -0.00018 & 0.00018 & -0.00018 & -0.00016 & 0.00015 & 0.00014 & 0.00012 
 \!\!\!\! \\  \hline \hline

$i$  & 12 & 3 & 11\\ \hline
$j$  & 2 & 5 & -2 \\ \hline
\!\!\!\!\! $c_{i,j}$ \!\!\!\!  & 0.00012 & 0.00011 & 0.00011
 \!\!\!\! \\ \hline \hline

\end{tabular}
\end{center}
\caption{The coefficients of Bear-2 for approximation in $\ell_1$ with the accuracy $0.01$}
\label{tablecoef2_1}
\end{table}

\thispagestyle{empty}

\begin{table}[H]
\small
\begin{center}
\begin{tabular}{c|c|c|c|c|c|c|c|c} 
$i$ & 2 & 1 & 5 & 2 & 0 & 0 & 4 & 6\\ \hline
$j$ & 0 & 0 & 0 & -1 & -1 & 1 & 0 & -1\\ \hline
\!\!\!\!\! $c_{i,j}$ \!\!\!\! & 1.08200 & 0.60379 & -0.13271 & 0.08179 & -0.06971 & -0.06948 & -0.06578 & 0.04453 \!\!\!\! \\ \hline \hline

$i$  & 6 & -3 & -2 & -1 & 8 & 5 & -4 & 3\\ \hline
$j$  & 1 & 0 & 0 & -2 & -1 & -1 & -1 & 2 \\ \hline
\!\!\!\!\! $c_{i,j}$ \!\!\!\!  & 0.04344 & 0.03556 & 0.03408 & 0.02357 & -0.02329 & 0.02213 & -0.02085 & -0.01941 \!\!\!\! \\ \hline \hline

$i$  & -3 & 7 & 5 & -3 & 1 & -5 & -4 & 9\\ \hline
$j$  & -2 & -1 & 1 & -1 & 1 & -2 & 0 & -1 \\ \hline
\!\!\!\!\! $c_{i,j}$ \!\!\!\!  & -0.01931 & -0.01870 & 0.01740 & -0.01580 & -0.01272 & 0.0122 & -0.01128 & 0.01096
 \!\!\!\! \\  \hline \hline

 $i$  & 10 & -5 & 8 & 6 & -6 & 1 & -4 & -5\\ \hline
$j$  & -1 & -1 & 2 & 2 & -1 & 2 & -3 & 0 \\ \hline
\!\!\!\!\! $c_{i,j}$ \!\!\!\!  & 0.01051 & 0.00876 & 0.00874 & -0.00846 & 0.00804 & 0.00796 & 0.00734 & -0.00710
 \!\!\!\! \\  \hline \hline

$i$  & -3 & 9 & -7 & -2 & -6 & -6 & 1 & 11 \\ \hline
$j$  & 2 & -2 & -2 & -3 & -2 & -3 & -1 & -1 \\ \hline
\!\!\!\!\! $c_{i,j}$ \!\!\!\!  & -0.00701 & 0.00677 & -0.00635 & -0.00626 & -0.00607 & -0.00598 & -0.00577 & -0.00505
 \!\!\!\! \\  \hline \hline

$i$  & 0 & 11 & 8 & 10 & 9 & 0 & -8 & -8\\ \hline
$j$  & 3 & -2 & -2 & -2 & 3 & -2 & 1 & -3 \\ \hline
\!\!\!\!\! $c_{i,j}$ \!\!\!\!  & -0.00486 & -0.0049 & 0.00460 & -0.00425 & -0.00408 & 0.00402 & -0.00399 & 0.0040
 \!\!\!\! \\  \hline \hline

$i$  & 6 & -7 & -7 & 7 & -8 & 1 & 12 & 2\\ \hline
$j$  & -3 & -3 & -1 & 3 & -2 & -4 & -2 & -3 \\ \hline
\!\!\!\!\! $c_{i,j}$ \!\!\!\!  & 0.00359 & 0.00347 & -0.00339 & 0.0032 & 0.00323 & -0.00302 & 0.00297 & 0.00293
 \!\!\!\! \\  \hline \hline

$i$  & 13 & -2 & -9 & -1 & -7 & 1 & -9 & -8\\ \hline
$j$  & -2 & 2 & -2 & 2 & -4 & 4 & -3 & -1 \\ \hline
\!\!\!\!\! $c_{i,j}$ \!\!\!\!  & 0.00289 & -0.00280 & 0.00277 & 0.00257 & 0.00237 & 0.00230 & -0.00230 & -0.00228
 \!\!\!\! \\  \hline \hline

$i$  & -10 & -9 & -5 & 4 & 5 & 13 & -1 & -6 \\ \hline
$j$  & -3 & -4 & -4 & 3 & -3 & -1 & -4 & 0 \\ \hline
\!\!\!\!\! $c_{i,j}$ \!\!\!\!  & -0.00221 & -0.00200 & -0.00199 & 0.00198 & 0.00195 & 0.00181 & 0.00177 & 0.00174
 \!\!\!\! \\  \hline \hline

$i$  & 14 & 12 & 7 & -11 & 10 & 8 & -2 & -11 \\ \hline
$j$  & -2 & 4 & -3 & 2 & 4 & -3 & -5 & -4 \\ \hline
\!\!\!\!\! $c_{i,j}$ \!\!\!\!  &  -0.00171 & -0.00161 & -0.00148 & -0.00147 & 0.00145 & -0.00144 & -0.00142 & 0.00138
 \!\!\!\! \\  \hline \hline

$i$  & -10 & 14 & -9 & -1 & 4 & -11 & -10 & 1 \\ \hline
$j$  & -2 & 4 & 0 & -3 & 5 & -3 & 1 & 3 \\ \hline
\!\!\!\!\! $c_{i,j}$ \!\!\!\!  & -0.0014 & 0.00131 & 0.00131 & -0.00129 & 0.00126 & 0.0013 & 0.00113 & -0.00111
 \!\!\!\! \\  \hline \hline

$i$  & -12 & -11 & 14 & -12 & 15 & 14 & -12 & 12 \\ \hline
$j$  & -3 & -2 & 0 & -4 & 0 & -3 & 2 & 0 \\ \hline
\!\!\!\!\! $c_{i,j}$ \!\!\!\!  & 0.00104 & -0.00093 & -0.00091 & -0.00088 & -0.00087 & -0.00085 & 0.00083 & 0.00081 
 \!\!\!\! \\  \hline \hline

\thispagestyle{empty} 

$i$  & -13 & -10 & -9 & -8 & -4 & -5 & -12 & 8 \\ \hline
$j$  & -4 & -5 & -1 & 3 & -5 & 4 & -5 & 4 \\ \hline
\!\!\!\!\! $c_{i,j}$ \!\!\!\!  & -0.00081 & 0.00080 & 0.0008 & 0.00076 & 0.00073 & 0.00072 & -0.00070 & -0.00070
 \!\!\!\! \\  \hline \hline

$i$  & 13 & 5\\ \hline
$j$  & -3 & -4 \\ \hline
\!\!\!\!\! $c_{i,j}$ \!\!\!\!  & -0.00068 & -0.00067
 \!\!\!\! \\ \hline \hline

\end{tabular}
\end{center}
\thispagestyle{empty} 
\caption{The coefficients of Bear-4 for approximation in $\ell_2$ with the accuracy $0.01$}
\label{tablecoef4_2}
\end{table}

\end{section}


\bigskip 
 
 \begin{thebibliography}{NN}
 
 
 
%
%
%
%
%
%
%
%
%
%



%


\bibitem[B10]{B10}
C.\,Bandt,
\newblock {\em Combinatorial topology of three-dimensional self-affine tiles},  
\newblock (2010) arXiv:1002.0710 
\smallskip

\bibitem[B91]{B91}
C.\,Bandt,
\newblock {\em  Self-similar sets. V. Integer matrices and fractal tilings of $\re^n$},  
\newblock Proc. Amer. Math. Soc. 112 (1991), no. 2, 549 -- 562.
\smallskip


\bibitem[BG]{BG}
C.\,Bandt, G.\,Gelbrich, 
\newblock {\em  Classiffication of self-affine lattice tilings},  
\newblock J. London Math. Soc. 50 (1994), no. 3, 581 -- 593.
\smallskip

\bibitem[BN]{BN}
D.\,Blondel, Yu.\,Nesterov,  
\newblock {\em Computationally efficient approximations of the joint spectral radius},  
\newblock SIAM J. Matrix Anal., 27 (2005), no. 1, 256 -- 272.
\smallskip


\bibitem[Boor]{Boor}
C.\,de\,Boor,  
\newblock {\em A Practical Guide to Splines},  
\newblock Springer-Verlag New York, vol. 27, p. 325 (1978). 
\smallskip


\bibitem[BHR]{BHR}
C.\,de\,Boor,  K.\,H\"ollig, S.\,Riemenschneider
\newblock {\em Box splines},  
\newblock Springer Science \& Business Media, vol. 98 (1993). 
\smallskip

\bibitem[BT]{BT}
V.\,Blondel, J.\,Tsitsiklis,  
\newblock {\em Approximating the spectral radius of sets of matrices in the max-algebra is NP-hard},  
\newblock IEEE Trans. Autom. Control, 45 (2000), no. 9, 1762 -- 1765.
\smallskip

\bibitem[BVR]{BVR}
C.\,de Boor, R.\,DeVore, A.\,Ron, 
\newblock {\em The structure of finitely generated shift-invariant spaces in $L_2(\R^d)$},  
\newblock Journal of Functional Analysis 119 (1994), no. 1, 37 -- 78.
\smallskip


\bibitem[BVR93]{BVR93}
C.\,de Boor, R.\,A.\,DeVore, A.\,Ron, 
\newblock {\em On the construction of multivariate (pre) wavelets},  
\newblock Constructive approximation 9 (1993), no. 2-3, 123 -- 166.
\smallskip


\bibitem[BVR94]{BVR94}
C.\,de Boor, R.\,DeVore, A.\,Ron, 
\newblock {\em Approximation from shift-invariant subspaces of $L_2(\R^d)$},  
\newblock Transactions of the American Mathematical Society 341 (1994), no. 2, 787 -- 806.
\smallskip

\bibitem[CC]{CC}
E.\,Catmull, J.\,Clark,  
\newblock {\em Recursively generated B-spline surfaces on arbitrary topological meshes},  
\newblock Computer-aided design, 10 (1978), no. 6, 350 -- 355. 
\smallskip

\bibitem[CCJZ]{CCJZ}
M.\,Charina, C.\,Conti, K.\,Jetter,  G.\,Zimmermann, 
\newblock {\em Scalar multivariate subdivision schemes and box splines},  
\newblock Computer aided geometric design, 28 (2011), no. 5, 285 -- 306.
\smallskip

\bibitem[CCS]{CCS}
M.\,Charina, C.\,Conti, T.\,Sauer,  
\newblock {\em Regularity of multivariate vector subdivision schemes},  
\newblock  Numerical algorithms, 39 (2005), no. 1-3, 97 -- 113.
\smallskip


\bibitem[CDM]{CDM}
A.\,S.\,Cavaretta, W.\,Dahmen, C.\,A.\,Micchelli, 
\newblock {\em Stationary subdivision},  
\newblock Vol. 453, American Mathematical Soc. (1991). 
\smallskip


\bibitem[CJ]{CJ}
C.\,Conti, K.\,Jetter, 
\newblock {\em Concerning order of convergence for subdivision},  
\newblock Numerical Algorithms 36 (2004), no. 4, 345 -- 363.
\smallskip

\bibitem[CGV]{CGV}
A.\,Cohen, K.\,Gr\"ochenig, L.\,F.\,Villemoes, 
\newblock {\em Regularity of multivariate refinable functions},  
\newblock Constructive approximation 15 (1999), no. 2, 241 -- 255.
\smallskip

\bibitem[CHM]{CHM}
C.A.\,Cabrelli, C.\,Heil, U.M.\,Molter,
\newblock {\em Self-similarity and multiwavelets in higher dimensions},
\newblock Memoirs Amer. Math. Soc. 170 (2004), no. 807.
\smallskip


\bibitem[CM]{CM}
M.\,Charina, Th.\,Mejstrik, 
\newblock {\em Multiple multivariate subdivision schemes: matrix and operator approaches},  
\newblock Journal of Computational and Applied Mathematics 349 (2019), 279 -- 291.
\smallskip


\bibitem[CP]{CP}
M.\,Charina, V.Yu.\,Protasov,  
\newblock {\em Regularity of anisotropic refinable functions},  
\newblock Applied and Computational Harmonic Analysis, 47 (2019), no. 3, 795 -- 821.
\smallskip


\bibitem[Daub]{Daub}
I.\,Daubechies,  
\newblock {\em Ten Lectures on Wavelets},  
\newblock CBMS-NSF Regional Conference Series in Applied Mathematics, vol. 61, SIAM, Philadelphia, 1992. 
\smallskip


\bibitem[DL]{DL}
N.\,Dyn, D.\,Levin,  
\newblock {\em Subdivision schemes in geometric modelling},  
\newblock Acta Numerica, 11 (2002), no. 0, 73 -- 144.
\smallskip


\bibitem[FG]{FG}
X.\,Fu, J.-P.\,Gabardo, 
\newblock {\em Self-affine scaling sets in $\re^2$},  
\newblock Memoirs of the American Mathematical Society 233 (2015), 1 -- 97.
\smallskip


\bibitem[G81]{G81}
W.J.\,Gilbert, 
\newblock {\em Radix representations of quadratic fields},  
\newblock J. Math. Anal. Appl. 83 (1981), no. 1, 264 -- 274.
\smallskip


\bibitem[Gel]{Gel}
G.\,Gelbrich, 
\newblock {\em Self-affine Lattice Reptiles with Two Pieces in $\re^n$},  
\newblock Math. Nachr., 178 (1996), no. 1, 129 -- 134.
\smallskip

\bibitem[gitTZ]{gitTZ}
T.\,Zaitseva, 
\newblock \href{https://github.com/TZZZZ/Tile\_Bsplines}{https://github.com/TZZZZ/Tile\_Bsplines}.  
\smallskip

\bibitem[GJ]{GJ}
R.\,Gundy, A.\,Jonsson, 
\newblock {\em Scaling functions on $\re^2$ for dilations of determinant $\pm 2$},  
\newblock Applied and Computational Harmonic Analysis 29 (2010), no. 1, 49 -- 62.
\smallskip



\bibitem[GM]{GM}
K.\,Gr\"ochenig, W.R.\,Madych, 
\newblock {\em  Multiresolution analysis, Haar bases, and self-similar tilings of $\re^n$},  
\newblock IEEE Trans. Inform. Theory 38 (1992), no. 2, 556 -- 568.
\smallskip

\bibitem[GH]{GH}
K.\,Gr\"ochenig, A.\,Haas,
\newblock {\em Self-similar lattice tilings},
\newblock J. Fourier Anal. Appl., 1 (1994), no. 2, 131 -- 170.
\smallskip

\bibitem[GP]{GP}
N.\,Guglielmi, V.\,Protasov,  
\newblock {\em Exact computation of joint spectral characteristics of linear operators},  
\newblock Foundations of Computational Mathematics, 13 (2013), no. 1, 37 -- 97. 
\smallskip

\bibitem[JO]{JO}
Q.\,Jiang, P.\,Oswald,  
\newblock {\em Triangular $\sqrt{3}$-subdivision schemes: the regular case},  
\newblock Journal of computational and applied mathematics, 156.1 (2003), 47 -- 75.
\smallskip

\bibitem[KL00]{KL00}
I.\,Kirat, K.-S.\,Lau, 
\newblock {\em On the connectedness of self-affine tiles},  
\newblock J. Lond. Math. Soc., 62 (2000), no. 1, 291 -- 304.
\smallskip


\bibitem[KPS]{KPS}
A.\,Krivoshein,  V.\,Protasov, M.\,Skopina, 
\newblock {\em Multivariate wavelet frames},  
\newblock Singapore: Springer (2016). 
\smallskip

\bibitem[KR]{KR}
M.\,G.\,Krein, M.\,A.\,Rutman, 
\newblock {\em Linear operators leaving invariant a cone in a Banach space},  
\newblock  Uspekhi Matematicheskikh Nauk, 3 (1948), no. 1, 3 -- 95.
\smallskip

\bibitem[LW95]{LW95}
J. C.\,Lagarias, Y.\,Wang,  
\newblock {\em Haar type orthonormal wavelet bases in $\re^2$},  
\newblock J. Fourier Anal. Appl. 2 (1995), no. 1, 1 -- 14.
\smallskip

\bibitem[LW97]{LW97}
J.\,Lagarias, Y.\,Wang, 
\newblock {\em Integral self-affine tiles in $\re^n$. II. Lattice tilings},  
\newblock J. Fourier Anal. Appl. 3 (1997), no. 1, 83 -- 102.
\smallskip


\bibitem[Mekh]{Mekh}
D.\,Mekhontsev, 
\newblock {\em IFStile software},  
\newblock \href{http://ifstile.com}{http://ifstile.com}
\smallskip


\bibitem[NPS]{NPS}
I.\,Novikov, V.Yu.\,Protasov, M.A.\,Skopina,
\newblock {\em Wavelets theory}, 
\newblock AMS, Translations Mathematical Monographs, 239 (2011).
\newblock  


\bibitem[O]{O}
P.\,Oswald, 
\newblock {\em Designing composite triangular subdivision schemes},  
\newblock Computer Aided Geometric Design 22 (2005), no. 7, 659 -- 679.
\smallskip

\bibitem[OS]{OS}
P.\,Oswald, P.\,Shr\"oder,  
\newblock {\em Composite primal/dual $\sqrt{3}$-subdivision schemes},  
\newblock Computer Aided Geometric Design 20 (2003), no. 3, 135 -- 164.
\smallskip

\bibitem[P01]{P01}
V.\,Protasov, 
\newblock {\em The stability of subdivision operator at its fixed point},  
\newblock SIAM journal on mathematical analysis 33 (2001), no. 2, 448 -- 460.
\smallskip

\bibitem[P06]{P06}
V.\,Yu.\,Protasov, 
\newblock {\em Fractal curves and wavelets},  
\newblock Izvestiya: Mathematics 70 (2006), no. 5,  975. 
\smallskip


\bibitem[P97]{P97}
V.\,Yu.\,Protasov,  
\newblock {\em The generalized spectral radius. A geometric approach},  
\newblock Izvestiya Math., 61 (1997), no. 5, 995 -- 1030.
\smallskip


\bibitem[PZ]{PZ}
V.\,Protasov, T.\,Zaitseva, 
\newblock {\em Self-affine 2-attractors and tiles},  
\newblock Mat. Sb., 213 (2022), no. 6, 71 -- 110.  
\smallskip


\bibitem[Shab]{Shab}
B.\,V.\,Shabat, 
\newblock {\em Introduction to complex analysis: functions of several variables},  
\newblock Vol. 110. American Mathematical Soc., 1992.
\smallskip


\bibitem[Shad]{Shad}
A.Yu.\,Shadrin, 
\newblock {\em The $L_{\infty}$-norm of the $L_2$-spline projector is bounded independently of the knot sequence: A proof of de Boor's conjecture},  
\newblock Acta Mathematica 187 (2001), no. 1,  59 -- 137.
\smallskip


\bibitem[SF]{SF}
G.\,Strang, G.\,Fix, 
\newblock {\em A Fourier analysis of the finite element variational method},  
\newblock Construct. Aspects of Funct. Anal., Springer (2011), 793 -- 840.
\smallskip


\bibitem[Ter]{Ter}
P.\,A.\,Terekhin, 
\newblock {\em Best approximation of functions in by polynomials on affine system.},  
\newblock Sbornik: Mathematics 202 (2011), no. 2, 279.
\smallskip


\bibitem[TM]{TM}
T.\,Mejstrik,  
\newblock {\em Algorithm 1011: Improved Invariant Polytope Algorithm and Applications},  
\newblock ACM Transactions on Mathematical Software (TOMS) 46 (2020), no. 3, 1 -- 26.
\smallskip


\bibitem[VBU]{VBU}
D.\,Van de Ville, T.\,Blu, M.\,Unser,  
\newblock {\em Isotropic polyharmonic B-splines: Scaling functions and wavelets},  
\newblock IEEE Trans. Signal Process. 14 (2005), no. 11, 1798 -- 1813.
\smallskip

\bibitem[Woj]{Woj}
P.\,Wojtaszczyk,  
\newblock {\em A Mathematical Introduction to Wavelets},  
\newblock London Math. Soc. Stud. Texts, vol. 37, Cambridge Univ. Press, Cambridge, New York, Melbourne, Madrid (1997). 
\smallskip

\bibitem[Zai]{Zai}
T.\,Zaitseva, 
\newblock {\em Haar wavelets and subdivision algorithms on the plane},  
\newblock Advances in Systems Science and Applications 17 (2017), no. 3,  49 -- 57.
\smallskip


\bibitem[Zai2]{Zai2}
T.\,I.\,Zaitseva, 
\newblock {\em Simple tiles and attractors}, 
\newblock Sb. Math., 211:9 (2020), 1233. 
\smallskip 


\bibitem[Zakh]{Zakh}
V.\,G.\,Zakharov, 
\newblock {\em Rotation properties of 2D isotropic dilation matrices},  
\newblock Int. J. Wavelets Multiresolut. Inf. Process. 16 (2018), no. 01.
\smallskip

\bibitem[Zakh2]{Zakh2}
V.\,G.\,Zakharov, 
\newblock {\em Elliptic scaling functions as compactly supported multivariate analogs of the B-splines},  
\newblock International Journal of Wavelets, Multiresolution and Information Processing 12 (2014), no. 02, 1450018.
\smallskip


\bibitem[Zube]{Zube}
S.\,Zube, 
\newblock {\em Number systems, $\alpha$-splines and refinement},  
\newblock Journal of computational and applied mathematics 172 (2004), no. 2, 207 -- 231.
\smallskip

 \end{thebibliography}
\end{document}